\documentclass[opre, nonblindrev]{informs3} 

\OneAndAHalfSpacedXI 

\usepackage{graphicx}
\graphicspath{{figs/}}  

\usepackage[english]{babel}
\usepackage[utf8]{inputenc}
\usepackage{amsmath, amsfonts, amssymb}
\allowdisplaybreaks
\usepackage{parskip}
\usepackage{url}
\usepackage{graphicx}
\usepackage{setspace}
\usepackage{subfigure}

\usepackage{natbib}
 \bibpunct[, ]{(}{)}{,}{a}{}{,}%
 \def\bibfont{\small}%
\TheoremsNumberedThrough     
\ECRepeatTheorems
\usepackage{bbm}
\usepackage{tikz}
\usetikzlibrary{arrows.meta, positioning}
\usetikzlibrary{decorations.pathreplacing}
\usetikzlibrary{shapes.geometric}  

\usepackage[unicode=true, bookmarks=false,
breaklinks=false,pdfborder={0 0 1}]
{hyperref}

\hypersetup{
 colorlinks,citecolor=blue,filecolor=blue,linkcolor=blue,urlcolor=blue}
\usepackage{booktabs}
\usepackage{tabularx}
\usepackage{enumitem}  

\newcommand{\prn}[1]{\left({#1}\right)} 
\newcommand{\brk}[1]{\left[{#1}\right]} 

\newcommand{\regret}{\mathrm{Regret}}

\newcommand{\Fi}{\mathcal F_{\mathrm{I}}}
\newcommand{\Fii}{\mathcal F_{\mathrm{II}}}
\newcommand{\loss}{\ell}

\newcommand{\Lossi}{\mathcal L_{\mathrm{I}}}
\newcommand{\Lossii}{\mathcal L_{\mathrm{II}}}

\newcommand{\R}{\mathbb{R}}
\newcommand{\E}{\mathbb{E}}

\renewcommand{\P}{\mathbb{P}}

\newcommand{\real}{\mathbb{R}}

\newcommand{\x}{\boldsymbol{x}}
\newcommand{\y}{\boldsymbol{y}}

\newcommand{\HO}{\textbf{HO}}

\newcommand{\mix}{\tilde{\pi}}

\newcommand{\pow}{f}

\definecolor{arc}{RGB}{128,128,255}

\usepackage[top=2.5cm, left=3cm, right=3cm, bottom=4.0cm]{geometry}

\newcommand*{\circled}[1]{\lower.7ex\hbox{\tikz\draw (0pt, 0pt)%
    circle (.5em) node {\makebox[1em][c]{\small #1}};}}
\usepackage{algorithm}
\usepackage{algorithmicx}
\usepackage{algpseudocode}
\algnewcommand{\algorithmicand}{\textbf{and }}
\algnewcommand{\algorithmicor}{\textbf{or }}
\algnewcommand{\OR}{\algorithmicor}
\algnewcommand{\AND}{\algorithmicand}
\makeatletter
\newenvironment{breakablealgorithm}
  {
  \begin{center}
     \refstepcounter{algorithm}
     \hrule height.8pt depth0pt \kern2pt
     \renewcommand{\caption}[2][\relax]{
      {\raggedright\textbf{\ALG@name~\thealgorithm} ##2\par}%
      \ifx\relax##1\relax 
         \addcontentsline{loa}{algorithm}{\protect\numberline{\thealgorithm}##2}%
      \else 
         \addcontentsline{loa}{algorithm}{\protect\numberline{\thealgorithm}##1}%
      \fi
      \kern2pt\hrule\kern2pt
     }
  }{
     \kern2pt\hrule\relax
  \end{center}
  }

\begin{document}


\RUNAUTHOR{Ao, Fu, Simchi-Levi}

\RUNTITLE{Two-stage Online Reusable Resource Allocation: Reservation, Overbooking and Confirmation Call}


\TITLE{ Two-stage Online Reusable Resource Allocation: \texorpdfstring{\\}{} Reservation, Overbooking and Confirmation Call 
}


\ARTICLEAUTHORS{%
\AUTHOR{Ruicheng Ao\textsuperscript{1} \quad\quad Hengyu Fu\textsuperscript{2} \quad\quad David Simchi-Levi\textsuperscript{1,3,4} }
\AFF{\textsuperscript{1}Institute for Data, Systems, and Society, Massachusetts Institute of Technology, Cambridge, MA 02139, \EMAIL{\texttt{\{aorc, dslevi\}@mit.edu}}\texorpdfstring{\\}
{}\textsuperscript{2}School of Mathematical Sciences, Peking University, Beijing 100871, \EMAIL{\texttt{2100010881@stu.pku.edu.cn}}\texorpdfstring{\\}
{}\textsuperscript{3}Department of Civil and Environmental Engineering, Massachusetts Institute of Technology, Cambridge, MA 02139\texorpdfstring{\\}{}\textsuperscript{4}Operations Research Center, Massachusetts Institute of Technology, Cambridge, MA 02139 } 
}

\ABSTRACT{
Service industries must balance overbooking risks against idle capacity: overbooking causes costly rejections of reserved customers, while underbooking leaves capacity unused. The challenge intensifies when decisions must be made online amid uncertain cancellations, multi-day occupancy durations, and competing walk-in demand. We study a two-stage online model where Stage I manages advance reservations and overbooking, while Stage II handles check-ins from both reserved and walk-in customers. The decision maker can schedule confirmation calls to refine no-show predictions, but must immediately accept or reject each request without knowledge of future arrivals. Cancellation uncertainties and cross-day correlations in occupancy cause any online policy to incur $\Omega(T)$ regret without sufficient walk-in demand (a \textit{busy season}). We develop \textit{decoupled adaptive safety stocks} that hedge overbooking risks using only single-day information. Under busy season conditions, our policy achieves $O(1)$ regret by decoupling the offline optimal into independent single-day problems. Regret decays exponentially with confirmation timing: delaying confirmation by just $O(\log T)$ time suffices for near-optimal performance without sacrificing customer satisfaction. Numerical experiments on synthetic and real hotel booking data from Algarve reveal a sharp phase transition in performance as confirmation timing varies, confirming our theoretical predictions.

}

\maketitle

\section{Introduction}
Service industries—healthcare, hospitality, transportation, and cloud computing—must balance two objectives: avoiding rejections of reserved customers while minimizing idle capacity. This balance is typically achieved through overbooking and accommodating walk-in customers.

Outpatient medical clinics and car rental companies exemplify these complexities. Each examination is scheduled for a specific time slot, and any unfilled appointment results in lost revenue. To reduce idle time given high no-show rates, clinics often overbook patients. On favorable days, this allows all patients to be seen by adjusting schedules or managing short waiting times—\citet{kros2009overbooking} demonstrated that overbooking in clinic scheduling can reduce idle time and achieve cost savings of up to 15\% in a field study. However, on busier days, overbooking can lead to patient dissatisfaction and increased workloads for medical staff. Similarly, car rental firms often overbook reservations when anticipating cancellations or no-shows. When actual demand exceeds expectations, customers may wait for another vehicle or be directed to a different location. Meanwhile, accommodating walk-in customers who require immediate vehicle access—a common service on platforms such as Uber\footnote{\url{https://m.uber.com/go/rent}}—provides flexibility to make lower overbooking levels and reduce risks, but necessitates real-time capacity management.

Beyond healthcare and rentals, hotels and cloud services face similar challenges. Hotels regularly overbook rooms, anticipating cancellations or no-shows to increase revenue. While this strategy typically improves profitability, it can lead to guests arriving to find no available rooms. Hotels must then arrange alternative accommodations, incurring extra costs and potentially harming customer satisfaction and loyalty. Cloud computing services face similar challenges: managing server capacity to handle both scheduled and unexpected user demand is crucial for maintaining service reliability. Effective capacity management requires not only strategic overbooking but also reliable confirmation processes to predict and respond to demand changes. These industries must balance reserved bookings with walk-in demand to maintain service quality and operational efficiency.

Confirmation calls allow firms to better predict no-shows and adjust overbooking levels. By contacting customers before service to verify their plans, firms gain valuable information about demand. Many hotels confirm arrivals 24 to 48 hours before check-in \citep{chen2016cancellation}—a practice reflected in policies at Booking\footnote{\url{https://www.booking.com/content/terms.html}} and Airbnb.\footnote{\url{https://www.airbnb.com/help/article/149}} In healthcare, appointment reminders reduce no-shows and improve utilization \citep{almog2003reduction}. Confirmation timing involves trade-offs: too early, and plans may change; too late, and there is insufficient time to respond to cancellations. Despite these challenges, confirmations provide valuable information for overbooking decisions.

Existing models typically focus on either overbooking or online resource allocation, but rarely their joint control with confirmation processes and walk-in demand. We lack a unified model integrating these elements with rigorous sensitivity analysis of their interactions.

We study a two-stage online model incorporating reservations, walk-ins, overbooking, and confirmations. Stage I manages advance reservations and overbooking; Stage II handles check-ins from both reserved and walk-in customers. Confirmation calls occur before service ends to refine no-show predictions and adjust overbooking. Section 2 formalizes the model. Section 3 establishes the offline benchmark for regret analysis. Sections 4 and 5 present our online policies and theoretical guarantees. Section 6 extends the results to heterogeneous customers with nested capacity protection. Section 7 provides numerical validation, and Section 8 concludes.

\subsection{Main Contributions}

\noindent\textbf{1. Model and Problem Formulation.} We develop a continuous-time two-stage model integrating reservations (Stage I), check-ins (Stage II), overbooking, and confirmation calls. Arrivals and cancellations are random over a $T$-day horizon with time-varying rates. The DM optimizes revenue by balancing overbooking risks against idle capacity. This framework enables comparison of online policies and confirmation schedules.

\noindent\textbf{2. Theoretical Results.} Cancellation uncertainties and inter-day correlations in occupancy create fundamental challenges. We establish four main results:
\begin{enumerate}[label=(\alph*), leftmargin=*]
    \item \textit{Lower bound without busy season.} When walk-in demand is insufficient, \emph{any} policy incurs $\Omega(T)$ regret (Proposition~\ref{prop:lower_bound}). This formally characterizes the minimal walk-in rate threshold for constant regret.

    \item \textit{Upper bound under busy season.} Our DASS policy achieves $O(1)$ regret via confidence-based safety stocks that decouple multi-day dependencies (Theorem~\ref{prop::StageI regret}). The decoupling reduces a $T$-dimensional problem to $T$ independent single-day problems.

    \item \textit{Confirmation timing sensitivity.} Regret decays \emph{exponentially} with post-confirmation time; $O(\log T)$ time suffices for near-optimal performance (Theorem~\ref{prop:HO_gap}). This enables mid-day confirmation without performance loss.

    \item \textit{No capacity scaling.} Unlike prior work requiring $C = \Omega(T)$ for sublinear regret, we achieve $O(1)$ regret with \emph{fixed} capacity—the first such result for reusable resource allocation with overbooking.
\end{enumerate}

\noindent\textbf{3. Empirical Validation.} Numerical experiments reveal a phase transition in regret as confirmation timing varies, validating our sensitivity analysis. Using synthetic data and hotel booking data from Algarve, our algorithm substantially outperforms overbooking strategies without confidence-based safeguards.

\subsection{Related Work}

\noindent\textbf{Overbooking.} Overbooking research spans from classical static models to modern dynamic and online approaches. Foundational work developed static models balancing revenue against no-show risks \citep{beckmann1958decision, thompson1961statistical, rothstein1971airline, liberman1978hotel, belobaba1987air}. More recent approaches apply dynamic programming or linear programming \citep{karaesmen2004overbooking, kunnumkal2012randomized, aydin2013single, dai2019network, ekbatani2022online, freund2023overbooking}; see \citet{freund2023overbooking} for a comprehensive survey. Closer to our setting is work on service systems with occupancy durations and walk-in customers: healthcare \citep{williams1977decision, shonick1977approach, vissers1979selecting, laganga2007clinic, kros2009overbooking}, hospitality and rentals \citep{geraghty1997revenue, kim2006stochastic, birkenheuer2011reservation, alexandrov2012reservations}, and restaurants \citep{tomas2013improving, caglar2014ioverbook, tse2017modeling, roy2022restaurant}. Recent extensions address stochastic durations and dynamic walk-ins \citep{liu2023managing, yang2023optimal, xiao2024robust}. Our work differs by jointly modeling reservation and check-in stages in continuous time while providing near-optimal regret guarantees.

\noindent\textbf{Online Resource Allocation.} Our analysis framework relates to online algorithms benchmarked against offline optima \citep{reiman2008asymptotically, jasin2012re, jasin2013analysis, ferreira2018online, bumpensanti2020re, banerjee2020uniform, vera2021bayesian, jiang2022degeneracy, zhu2023assign, jaillet2024should}; see \citet{balseiro2023survey} for a survey. Particularly relevant is reusable resource allocation, where resources return after occupancy periods \citep{levi2010provably, chen2017revenue, chawla2017stability, owen2018price, rusmevichientong2020dynamic, gong2022online, zhang2022online, balseiro2023dynamic, rusmevichientong2023revenue, simchi2023greedy, dong2024value}. Existing work typically measures performance via competitive ratios under i.i.d. arrivals, requiring capacity to scale with the horizon to achieve sublinear online-offline gaps. In contrast, we achieve $O(1)$ regret without this requirement.

Our problem differs in two key aspects. First, we allow overbooking throughout, introducing additional allocation uncertainty. Second, we achieve $O(1)$ regret without requiring large-capacity asymptotics—walk-in arrivals provide sufficient information to decouple inter-day dependencies. Our Stage I policy employs decoupling techniques similar to inventory control \citep{zipkin2000foundations, miao2022asymptotically, chao2024adaptive}, while Stage II builds on prediction-based methods for settings with occupancy durations and state-dependent arrivals \citep{jasin2012re, vera2021bayesian, feng2021online, zhang2022online, freund2023overbooking, jiang2023constant, ao2024online}.

\noindent\textbf{Confirmation Calls.} Confirmation calls are widely used in reservation-based services to reduce overbooking risks and idle capacity \citep{christensen2001effect, almog2003reduction, stubbs2012methods, chen2016cancellation}. At confirmation, the DM verifies whether reserved customers will arrive. Firms may cancel unconfirmed reservations or impose penalties \citep{chen2016cancellation}. Confirmations can occur before service begins or partway through the service period. After confirmation, the DM has complete information about reservation status, effectively transitioning from online to offline decision-making. \citet{xie2023benefits} study a related ``decision delay'' concept where the DM observes all arrivals up to time $t+d$ when deciding at time $t$. Our setting differs: the DM has no future arrival information until confirmation. We analyze how confirmation timing affects regret, showing that a logarithmic interval between confirmation and service end suffices for near-optimal performance.


\subsection{Notations}
For integer $n\ge 1,$ we denote $[n]=\{1,2,\dots,n\}$ as the set of integers from $1$ to $n$. For $x\in\real$, denote $\lceil x\rceil$ as the smallest integer not smaller than $x$ and $\lfloor x\rfloor$ as the largest integer not greater than $x$. Denote $x_+=\max\{x, 0\}$. For set $S$, denote $|S|$ as its cardinality. For two functions $f(T)$ and $g(T)$, we use $f(T)=O(g(T))$ if there exists constant $c_1>0$ such that $f(T)\le c_1g(T)$ as $T\to+\infty$ and $f(T)=\Omega(g(T))$ if there exists constant $c_2>0$ such that $f(T)\ge c_2g(T)$ as $T\to +\infty$.

\section{Problem Setup}
We develop a two-stage reusable resource allocation model with overbooking. Figure~\ref{fig::timeline} illustrates the decision timeline for a single service day. Section \ref{sec::setup} presents the model framework. Section \ref{sec:arrival} describes the customer arrival process. Section \ref{sec:benchmark} formulates the online problem and establishes the benchmark for regret analysis. Table \ref{tab:notation_board} summarizes notation.

\begin{figure}[ht]
\centering
\begin{tikzpicture}[x=1cm, y=1cm]

\draw[ultra thick, -{Latex[length=3mm]}] (0,0) -- (14.2,0) node[anchor=north west] {Time};

\node[align=center] at (11, 1.4) {Confirmation call};

\filldraw [red] (11,1.1) circle (2pt);

\foreach \x in {0, 6, 11} {
    \draw[thick] (\x, 0.1) -- (\x, -0.1);
}

\node at (0, -0.5) {Day $k-k_0$};
\node at (6, -0.5) {Day $k$};
\node at (11, -0.5) {Time $k+v$};
\node at (13.6, -0.5) {Day $k+1$};

\draw[dashed] (6, 1) -- (6, -0.2);
\draw[dashed] (11, 1) -- (11, -0.2);
\draw[dashed] (13.6, 1) -- (13.6, -0.2);

\node at (2.8,0.9) {Type I Customers};
\node at (8.5, 0.9) {Type I \& II Customers};
\node at (12.3, 0.9) {Final Decision};

\draw[dotted, thick, red] plot[smooth, tension=0.8] coordinates {(0,0.0) (1,0.1) (2,0.2) (3,0.35) (4,0.6) (5,0.85) (6,1)};

\draw[dotted, thick, blue] plot[smooth, tension=0.8] coordinates {(6,0.7) (7.2,0.4) (8.4,0.7) (9.6,0.4) (11,0.7)};
\draw[dotted, thick, blue] plot[smooth, tension=0.8] coordinates {(11,0.3) (12,0.2) (13.6,0.1)};

\draw [decorate, decoration={brace, mirror, amplitude=10pt}, thick] (0,-1) -- (6,-1)
    node [midway, below, yshift=-10pt] {Stage I: Reservations prior to Day $k$};
\draw [decorate, decoration={brace, mirror, amplitude=10pt}, thick] (6,-1) -- (13.6,-1)
    node [midway, below, yshift=-10pt] {Stage II: Check-in on Day $k$};

\end{tikzpicture}

\caption{Timeline of two-stage control for Day $k$. Reservations arrive during Stage I (days $k-k_0$ to $k$). On Day $k$, both reserved and walk-in customers arrive during Stage II. The confirmation call at time $k+\nu$ reveals which reservations will actually show up, enabling the manager to finalize capacity allocation.
 \label{fig::timeline}}
\end{figure}

\textbf{Illustrative example.} Consider a hotel, as depicted in Figure~\ref{fig::timeline}. Guests make reservations days or weeks in advance (Stage I); the red curve shows accepted reservations accumulating over time. Some guests cancel before their stay. On Day $k$, reserved guests check in while walk-in guests also arrive (Stage II); the blue curves depict walk-in traffic patterns. A confirmation call at time $k+\nu$---marked by the red dot---reveals which reservations will show up. The manager faces a fundamental trade-off: accept too few reservations and risk empty rooms if walk-ins are sparse; accept too many and risk overbooking if cancellations are fewer than expected. This tension between \emph{overbooking risk} and \emph{idle capacity risk} lies at the heart of our model.

\subsection{The Two-stage Online Reusable Resource Allocation Model}\label{sec::setup}
\subsubsection{Continuous-time Framework}
The model operates over a continuous time horizon $[1-k_0, T+1]$, where $T$ denotes the total number of service days and $k_0 \in \mathbb{N}$ is the maximum advance booking window (e.g., $k_0 = 7$ for one-week advance reservations). For each day $k \in [T]$, the interval $(k, k+1)$ represents Day $k$ (e.g., 12:00 a.m. to 11:59 p.m.). Each day $k$ has its own two-stage structure: Stage I (reservations for Day $k$) spans $[k - k_0, k]$, while Stage II (check-ins on Day $k$) spans $(k, k+1)$. Decision windows for different days overlap---while Day $k$ is in Stage II, reservations for Days $k+1, k+2, \ldots$ are already arriving in their respective Stage I periods. Both stages have sequential arrivals. A real number $t$ denotes time $t - \lfloor t \rfloor$ on day $\lfloor t \rfloor$.

\subsubsection{Customer Types: Reservation and Walk-in}
Customers are classified into two types: advance reservations (\textbf{Type I}) and walk-ins (\textbf{Type II}). Figure~\ref{fig::customer_journey} depicts their distinct journeys through the two-stage process. Type I customers intending to check in on day $k \in [T]$ make reservations within $[k - k_0, k]$. Cancellations may occur before time $k+1$ with no penalty. Customers with uncancelled reservations arrive for check-in in Stage II. Type II customers arrive within $[k,k+1)$ (Stage II) without prior reservation. Both types may occupy resources for multiple days. The DM observes each customer's occupancy duration upon receiving the reservation or check-in request.

\begin{figure}[ht]
\centering
\begin{tikzpicture}[
    box/.style={rectangle, draw, rounded corners, minimum width=2.2cm, minimum height=0.6cm, align=center, font=\small},
    decision/.style={diamond, draw, aspect=2.5, inner sep=1pt, align=center, font=\small},
    arrow/.style={-{Latex[length=2mm]}, thick},
    typeI/.style={draw=red!70!black, thick},
    typeII/.style={draw=blue!70!black, thick},
    stagebox/.style={fill=gray!8, draw=none}
]

\fill[gray!12] (-6.5, -0.3) rectangle (6.5, -3.2);
\node[font=\scriptsize\bfseries, gray!60!black] at (-4.8, -0.6) {Stage I};

\fill[gray!6] (-6.5, -3.2) rectangle (6.5, -7.5);
\node[font=\scriptsize\bfseries, gray!60!black] at (-4.8, -3.5) {Stage II};

\node[font=\small\bfseries] (header) at (0, 0.6) {Customer Arrivals};

\node[box, typeI] (typeI) at (-2.8, -1.0) {Type I};
\node[right, font=\scriptsize, red!70!black, xshift=2pt] at (typeI.east) {$t \in [k{-}k_0, k]$};

\node[decision, typeI] (cancel) at (-2.8, -2.5) {\scriptsize Cancel?};
\node[font=\scriptsize\itshape, red!70!black] (exit1) at (-4.8, -2.5) {Exit};

\node[box, typeII] (typeII) at (2.8, -1.0) {Type II};
\node[right, font=\scriptsize, blue!70!black, align=left, xshift=2pt] at (typeII.east) {(not yet arrived)\\ \textit{waiting...}};

\node[box, typeI] (confirm) at (-2.8, -4.2) {Confirmation};
\node[right, font=\scriptsize, red!70!black, xshift=2pt] at (confirm.east) {at $k{+}\nu$};

\node[decision, typeI] (showup) at (-2.8, -5.8) {\scriptsize Show?};
\node[font=\scriptsize\itshape, red!70!black] (exit2) at (-4.8, -5.8) {Exit};

\node[box, typeII] (walkin) at (2.8, -4.2) {Walk-in};
\node[right, font=\scriptsize, blue!70!black, xshift=2pt] at (walkin.east) {$u \in [k, k{+}1)$};

\node[box, minimum width=3.2cm, fill=white] (checkin) at (0, -7.0) {Check-in};

\draw[arrow, red!70!black] (header) -- (-2.8, 0.2) -- (typeI);
\draw[arrow, red!70!black] (typeI.south) ++(0,-0.05) -- (cancel);
\draw[arrow, red!70!black] (cancel.west) -- node[above, font=\scriptsize] {Yes} (exit1);
\draw[arrow, red!70!black] (cancel.south) -- node[right, font=\scriptsize, xshift=2pt] {No} ++(0, -0.6) -- (confirm);
\draw[arrow, red!70!black] (confirm.south) ++(0,-0.05) -- (showup);
\draw[arrow, red!70!black] (showup.west) -- node[above, font=\scriptsize] {No} (exit2);
\draw[arrow, red!70!black] (showup.south) -- node[right, font=\scriptsize, xshift=2pt] {Yes} ++(0, -0.4) -| (-0.8, -6.65);

\draw[arrow, blue!70!black] (header) -- (2.8, 0.2) -- (typeII);
\draw[arrow, blue!70!black, dashed] (typeII.south) ++(0,-0.05) -- (walkin);
\draw[arrow, blue!70!black] (walkin.south) ++(0,-0.05) -- ++(0, -1.55) -| (0.8, -6.65);

\end{tikzpicture}
\caption{Customer journey flow. Type I customers (red) make advance reservations during Stage I and may cancel; those who confirm at $k{+}\nu$ proceed to check-in. Type II customers (blue) are not yet present during Stage I; they arrive directly as walk-ins during Stage II. Both types undergo capacity checks at check-in.}
\label{fig::customer_journey}
\end{figure}

\subsubsection{Confirmation Call}
At time $k+\nu$ with $\nu\in(-k_0,1]$, the DM makes confirmation calls to all reserved customers for Day $k$. For instance, $\nu = -0.5$ corresponds to a call at noon the day before service, while $\nu = 0.5$ corresponds to midday on the service day itself. When $\nu < 0$ ($\nu \geq 0$), calls occur in Stage I (Stage II). Customers who have not cancelled by $k+\nu$ are contacted; cancellations are finalized during the call, and remaining customers arrive with certainty. From time $k+\nu$ onward, the DM has complete information about Type I arrivals on Day $k$. We treat $\nu$ as a predetermined constant and perform sensitivity analysis on this parameter.

\subsubsection{Manager's Online Decision Making: Balancing Reservations and Walk-ins}
The DM makes immediate accept/reject decisions for each reservation or walk-in request. For Day $k$, this involves selecting which reservations from $[k - k_0, k]$ to accept in Stage I and which Type II customers from $[k, k + 1)$ to accept in Stage II. There are $\mathcal{C}$ units of capacity each day (e.g., $C = 100$ hotel rooms or clinic appointment slots). Once occupied, a unit cannot be reallocated on the same day.

In Stage I, the DM may accept reservations beyond available capacity. When a Type I customer arrives at time $u \in [k, k + 1)$ but no resources are available, the DM incurs an overbooking loss of $\loss^{(k)}$. Each occupied resource on Day $k$ generates revenue $r^{(k)}$; idle resources forgo this revenue. The DM may accept Type II customers even when overbooking risk is positive. Admission control across both stages balances overbooking risk against capacity utilization. Figure~\ref{fig::overbooking_tradeoff} contrasts the two extremes: a conservative policy leaves rooms idle, forgoing revenue $r^{(k)}$, while an aggressive policy risks overbooking penalties $\loss^{(k)}$.

\begin{figure}[ht]
\centering
\begin{tikzpicture}[x=0.45cm, y=0.45cm]

\def\rs{0.7}

\begin{scope}[shift={(-8, 0)}]
    \node[font=\small\bfseries] at (2.5, 6.0) {Conservative Policy};
    \draw[thick, gray] (-0.3, -0.3) rectangle (5.3, 4.3);
    \foreach \x in {0,1,2,3,4} {
        \fill[black!75] (\x+0.15, 0.15) rectangle (\x+\rs, \rs);
    }
    \foreach \x in {0,1,2,3,4} {
        \fill[black!75] (\x+0.15, 1.15) rectangle (\x+\rs, 1+\rs);
    }
    \foreach \x in {0,1,2} {
        \fill[black!75] (\x+0.15, 2.15) rectangle (\x+\rs, 2+\rs);
    }
    \foreach \x in {3,4} {
        \draw[gray, thick] (\x+0.15, 2.15) rectangle (\x+\rs, 2+\rs);
    }
    \foreach \x in {0,1,2,3,4} {
        \draw[gray, thick] (\x+0.15, 3.15) rectangle (\x+\rs, 3+\rs);
    }
    \node[font=\scriptsize, align=center] at (2.5, -1.2) {Revenue loss: $r^{(k)} \times 7$};
\end{scope}

\begin{scope}[shift={(3, 0)}]
    \node[font=\small\bfseries] at (2.5, 6.0) {Aggressive Policy};
    \draw[thick, gray] (-0.3, -0.3) rectangle (5.3, 4.3);
    \foreach \y in {0,1,2,3} {
        \foreach \x in {0,1,2,3,4} {
            \fill[black!75] (\x+0.15, \y+0.15) rectangle (\x+\rs, \y+\rs);
        }
    }
    \node[font=\normalsize, red!70!black] at (1.5, 5.0) {\texttimes};
    \node[font=\normalsize, red!70!black] at (2.5, 5.0) {\texttimes};
    \node[font=\normalsize, red!70!black] at (3.5, 5.0) {\texttimes};
    \node[font=\scriptsize, red!70!black, right] at (4.0, 5.0) {turned away};
    \node[font=\scriptsize, align=center] at (2.5, -1.2) {Overbooking loss: $\loss^{(k)} \times 3$};
\end{scope}

\node[font=\scriptsize] at (-4.0, -2.8) {$\blacksquare$ Occupied};
\node[font=\scriptsize, gray] at (0, -2.8) {$\square$ Idle};
\node[font=\scriptsize, red!70!black] at (4.0, -2.8) {\texttimes~Overbooked};

\end{tikzpicture}
\caption{The overbooking tradeoff. A conservative policy (left) accepts fewer reservations, leaving rooms idle and forgoing revenue $r^{(k)}$ per empty room. An aggressive policy (right) overbooks, filling all rooms but incurring penalty $\loss^{(k)}$ when confirmed arrivals exceed capacity.}
\label{fig::overbooking_tradeoff}
\end{figure}

\paragraph{Notation conventions.}
Superscripts indicate customer type and day: $(j,k)$ denotes Type $j$ ($j=1$ for reservations, $j=2$ for walk-ins) on Day $k$. Subscripts indicate time or customer index. For example, $\lambda^{(1,k)}(t)$ is the arrival rate of Type I customers at time $t$ for Day $k$, and $B_t^{(k)}$ is the number of accepted reservations at time $t$ for Day $k$. Key notation is summarized below; Table \ref{tab:notation_board} provides a complete reference.
$T$ denotes the horizon (days), $C$ is the capacity per day,
$\lambda^{(1,k)}(t)$ and $\lambda^{(2,k)}(u)$ are arrival rates for Type I (reservations) and Type II (walk-ins) on Day $k$,
$\nu$ is the confirmation timing offset,
$\delta$ is the expected per-day occupancy fraction (turnover rate),
$\loss^{(k)}$, $r^{(k)}$ are the overbooking loss and revenue on Day $k$,
and $\iota = \log(CT)$ is a confidence parameter used in probabilistic bounds.
Table \ref{tab:notation_board} provides a comprehensive summary.

\begin{table}[h!]
\centering
\caption{Notations Used in the Model}
\renewcommand{\arraystretch}{1.2} 
\begin{tabular}{|>{\centering\arraybackslash}m{3.5cm}|p{12cm}|}
\hline
\textbf{\small Notation} & \textbf{Description} \\ \hline
\small $T$ & Total number of days in the time horizon. \\ \hline
\small $k (\in \mathbb{Z})$ & Discrete time variable to represent a specific Day $k$ in the time horizon. \\ \hline
\small $[k, k+1)$ & Time interval associated with the check-in control phase (Stage II) on Day $k$. \\ \hline
\small $k_0 (\in \mathbb{N})$ & Maximum number of days in advance that a reservation (by a Type I customer) can be made. \\ \hline
\small $C$ & Total available capacity. \\ \hline
\small $\mathcal{L}^{(k)}$ & Overbooking penalty incurred if a Type I customer arrives to check in without resources available. \\ \hline
\small $r^{(k)}$ & Revenue from unit resource occupancy on Day $k$. \\ \hline
\small $t (\le T+1)$ & Continuous time point used in Stage I. \\ \hline
\small $u (\in [0,1])$ & Continuous time point used in Stage II. \\ \hline
\small $\Lambda^{(1,k)}(t,k)$ & Total number of booking requests in time interval $(t,k)$. \\ \hline
\small $\lambda^{(1,k)}(t)$ & The arrival rate at time $t$ of Type I customers in Stage I for Day $k$. \\ \hline
\small $p^{(k)}(t)$ & The probability of cancellation of a reservation made at $t$ for Day $k$ in Stage I. \\ \hline
\small $(X_i^{(0,k)}, Y_i^{(0,k)}, D_i^{(0,k)})$ & The booking status of the $i$-th Type I customer for a Day $k$ check-in during Stage I, where $X_i^{(0,k)} \in [0,1]$ and $Y_i^{(0,k)} \in \{0,1\}$. \\ \hline
\small $B^{(k)}$ & The number of accepted reservations on Day $k$ not cancelled in Stage I. \\ \hline
\small $q^{(1,k)}$ & The probability of reservation cancellation for Day $k$ in Stage II. \\ \hline
\small $\gamma^{(k)}(u)$ & The distribution of Type I's check-in timing in Stage II. \\ \hline
\small $\Lambda^{(2,k)}[u_1,u_2]$ & Total number of Type II customers in time interval $[u_1,u_2]$ in Stage II. \\ \hline
\small $\lambda^{(2,k)}(u)$ & The arrival rate of Type II customers in Stage II on Day $k$. \\ \hline
\small $(X_i^{(1,k)}, Y_i^{(1,k)}, D_i^{(1,k)})$ & The check-in status of the $i$-th Type I customer arriving on Day $k$, where $X_i^{(1,k)} \in [0,1]$ and $Y_i^{(1,k)} \in \{0,1\}$. \\ \hline
\small $(X_i^{(2,k)}, D_i^{(2,k)})$ & The check-in status of the $i$-th Type II customer arriving on Day $k$, where $X_i^{(2,k)} \in [0,1]$. \\ \hline
\end{tabular}
\label{tab:notation_board}
\end{table}

\subsection{Customer Arrival Process}\label{sec:arrival}
\subsubsection{Type I Customers in Stage I}
Stage I receives bookings from Type I customers. For any $k \in [T]$, $\Lambda^{(1,k)}[t,k]$ denotes the total number of booking requests for Day $k$ in interval $[t,k]$. We assume $\Lambda^{(1,k)}$ follows a Poisson process with measurable time-varying rate $\lambda^{(1,k)}(t)$.
Customers make independent reservation and cancellation decisions. As time approaches Day $k$, cancellations become less likely. A Type I customer booking at time $t$ keeps the reservation during Stage I with probability $p^{(k)}(t)$, where $p^{(k)}:[k-k_0,k]\rightarrow [0,1]$ is nondecreasing with $p^{(k)}(k)=1$.

Common specifications for $p^{(k)}(t)$ include:
\begin{itemize}
    \item \textit{Linear}: $p^{(k)}(t) = (t - (k-k_0))/k_0$, where cancellation probability decreases linearly with lead time.
    \item \textit{Exponential}: $p^{(k)}(t) = 1 - \exp(-\beta(t - (k-k_0)))$ for some rate $\beta > 0$.
\end{itemize}
For example, with $k_0=7$ days and linear $p^{(k)}(t)$, a customer booking 5 days ahead ($t=k-5$) has $p^{(k)}(k-5) = 2/7 \approx 0.29$ probability of keeping the reservation through Stage~I, while one booking 0.5 days ahead has $p^{(k)}(k-0.5) = 6.5/7 \approx 0.93$. This captures the intuition that last-minute bookings rarely cancel. Figure~\ref{fig::cancellation_prob} plots retention probability $p^{(k)}(t)$ for both specifications.

\begin{figure}[ht]
\centering
\begin{tikzpicture}[x=1cm, y=1cm]

\draw[ultra thick, -{Latex[length=2.5mm]}] (0,0) -- (9,0) node[anchor=north west] {Time};
\draw[ultra thick, -{Latex[length=2.5mm]}] (0,0) -- (0,4.8) node[anchor=south east] {$p^{(k)}(t)$};

\node[left] at (0, 0) {\small 0};
\node[left] at (0, 4) {\small 1};
\draw[thick] (-0.1, 4) -- (0.1, 4);

\node[below] at (0.5, 0) {\small $k{-}k_0$};
\node[below] at (5.5, 0) {\small $k$};
\node[below] at (7.5, 0) {\small $k{+}\nu$};
\draw[thick] (0.5, -0.1) -- (0.5, 0.1);
\draw[thick] (5.5, -0.1) -- (5.5, 0.1);
\draw[thick] (7.5, -0.1) -- (7.5, 0.1);

\draw[dashed, gray] (5.5, 0) -- (5.5, 4.6);
\node[above, gray, font=\scriptsize] at (3, 4.5) {Stage I};
\node[above, gray, font=\scriptsize] at (6.5, 4.5) {Stage II};

\draw[thick, blue] plot[smooth, tension=0.8] coordinates {(0.5, 0) (2, 1.2) (3.5, 2.4) (5, 3.6) (5.5, 4)};
\draw[thick, blue] (5.5, 4) -- (8.5, 4);

\draw[thick, red!70!black, densely dashed] plot[smooth, tension=0.7] coordinates {(0.5, 0) (1.5, 1.8) (2.5, 2.8) (3.5, 3.4) (4.5, 3.8) (5.5, 4)};
\draw[thick, red!70!black, densely dashed] (5.5, 4) -- (8.5, 4);

\filldraw[black] (7.5, 4) circle (2pt);
\draw[thick, dotted, gray] (7.5, 0) -- (7.5, 4);
\node[right, font=\small] at (7.6, 3.4) {$q^{(1,k)} = p^{(k)}(k{+}\nu)$};

\draw[thick, gray!50] (6.5, 0.4) rectangle (9.5, 1.6);
\draw[thick, blue] (6.7, 1.3) -- (7.3, 1.3);
\node[right, font=\scriptsize] at (7.4, 1.3) {Linear};
\draw[thick, red!70!black, densely dashed] (6.7, 0.7) -- (7.3, 0.7);
\node[right, font=\scriptsize] at (7.4, 0.7) {Exponential};

\draw [decorate, decoration={brace, mirror, amplitude=5pt}, thick] (0.5,-1.2) -- (2.5,-1.2)
    node [midway, below, yshift=-6pt, font=\scriptsize, align=center] {Early booking};
\draw [decorate, decoration={brace, mirror, amplitude=5pt}, thick] (4,-1.2) -- (5.5,-1.2)
    node [midway, below, yshift=-6pt, font=\scriptsize, align=center] {Late booking};

\end{tikzpicture}
\caption{Retention probability as a function of booking time. The probability $p^{(k)}(t)$ of keeping a reservation increases as booking time $t$ approaches Day $k$. Early bookings have low retention probability; late bookings rarely cancel. After Day $k$, $p^{(k)}(t) = 1$: at the confirmation call ($k{+}\nu$), remaining customers show up with certainty.}
\label{fig::cancellation_prob}
\end{figure}

Booking requests in Stage I are modeled as 3-dimensional vectors:
\[
(X^{(0,k)}_1, Y^{(0,k)}_1,D^{(0,k)}_1),~ \dots,~ (X^{(0,k)}_{\Lambda^{(1,k)}(k-k_0,k)}, Y^{(0,k)}_{\Lambda^{(1,k)}(k-k_0,k)},D^{(0,k)}_{\Lambda^{(1,k)}(k-k_0,k)}).
\]

Here, $X^{(0,k)}_i \in [k-k_0,k]$ is the arrival time of the $i$-th reservation request. The binary variable $Y^{(0,k)}_i$ equals $0$ for cancellation in Stage I and $1$ otherwise. $D_i^{(0,k)}$ denotes the occupancy duration for the $i$-th booking. Conditioned on $X^{(0,k)}_i$, we have $Y^{(0,k)}_i \sim \text{Bernoulli}(p^{(k)}(X^{(0,k)}_i))$. The integrated rate of non-cancelled Type I customers is
\begin{align*}
    \lambda^{(1,k)}=\int_{k-k_0}^{k}\lambda^{(1,k)}(t)(1-p^{(k)}(t)) dt.
\end{align*}

\subsubsection{Customers in Stage II}

Let $B^{(k)}$ denote the number of reservations for Day $k$ not cancelled during Stage I. Each Type I customer's check-in information is modeled as a 3-dimensional vector:
\[
(X^{(1,k)}_1, Y^{(1,k)}_1, D^{(1,k)}_1),~ (X^{(1,k)}_2, Y^{(1,k)}_2,D^{(1,k)}_2),~ \dots,~ (X^{(1,k)}_{B^{(k)}}, Y^{(1,k)}_{B^{(k)}},D^{(1,k)}_{B^{(k)}}).
\]
Here, $X^{(1,k)}_i \in [0,1]$ is the time when the $i$-th Type I customer confirms check-in or cancellation. The binary variable $Y^{(1,k)}_i$ equals 0 for cancellation and 1 for check-in. We assume $X_i^{(1,k)} \sim \gamma^{(k)}$ and $Y_i^{(1,k)} \sim \text{Bernoulli}(q^{(1,k)})$. $D_i^{(1,k)}$ denotes the occupancy duration.

The Stage II confirmation rate $q^{(1,k)}$ relates to $p^{(k)}(t)$ via $q^{(1,k)} = p^{(k)}(k+\nu)$---the survival probability evaluated at confirmation time $k+\nu$. When confirmation occurs after the service day begins ($\nu > 0$), $p^{(k)}$ is capped at 1, so $q^{(1,k)} = 1$: customers who survived Stage I almost surely show up. When confirmation occurs before the service day ($\nu < 0$), residual cancellation risk remains and $q^{(1,k)} < 1$.

For Type II customers on Day $k$, $\Lambda^{(2,k)}[u_1,u_2]$ denotes the number of arrivals in interval $(u_1, u_2)$. We assume $\Lambda^{(2,k)}$ follows a time-varying Poisson process with rate $\lambda^{(2,k)}(u)$ for $u \in [0,1]$. The integrated arrival rate is
\[
\lambda^{(2,k)} = \int_{0}^{1} \lambda^{(2,k)}(u) \, du.
\]
The total number of Type II customers on Day $k$ follows $\text{Poisson}(\lambda^{(2,k)})$. Each Type II customer's check-in information is modeled as a 2-dimensional vector (Type II customers do not cancel):
\[
(X^{(2,k)}_1,D^{(2,k)}_1),~ (X^{(2,k)}_2,D^{(2,k)}_2),~ \dots,~ (X^{(2,k)}_{\Lambda^{(2,k)}[0,1]},D^{(2,k)}_{\Lambda^{(2,k)}[0,1]}).
\]
Here, $X^{(2,k)}_i \in [0,1]$ is the arrival time of the $i$-th Type II customer. $D_i^{(j,k)}$ denotes the occupancy duration. For $j\in\{0,1,2\}$, we consider two occupancy duration distributions widely used in service systems (e.g., \citet{lei2020real}, \citet{dong2024value}): \begin{itemize}
    \item[1)] \textbf{Geometric}: $D_i^{(j,k)}\sim \text{Geometric}(q)$, where $q\in [0,1]$.
    \item[2)] \textbf{Constant}: $D_i^{(j,k)}\equiv d$, where $d\in \mathbb{N}$.
\end{itemize}

For analytical convenience, we define the parameter $\delta$ to capture the expected per-day resource occupancy. For geometric duration, $\delta = 1-q$ represents the probability that a customer departs each day; for constant duration $d$, we set $\delta = 1/d$. This parameter $\delta$ measures the rate at which resources turn over: higher $\delta$ indicates faster turnover, while lower $\delta$ indicates longer average occupancy. The parameter $\delta$ will be used throughout our analysis to characterize capacity constraints and arrival rate requirements. Figure~\ref{fig::capacity_evolution} shows how daily occupancy combines continuing stays with new check-ins, with approximately $\delta C$ rooms turning over each day.

\begin{figure}[ht]
\centering
\begin{tikzpicture}[x=1.6cm, y=0.5cm]

\draw[ultra thick, -{Latex[length=2.5mm]}] (-0.3,0) -- (7.5,0) node[anchor=north west] {Days};
\draw[ultra thick, -{Latex[length=2.5mm]}] (0,0) -- (0,8.5) node[anchor=south east, font=\small] {Rooms};

\draw[thick, dashed, gray] (-0.2, 7) -- (5.3, 7);
\node[left, font=\small, gray] at (-0.2, 7) {$C$};

\foreach \x/\label in {0.5/{$k{-}2$}, 1.5/{$k{-}1$}, 2.5/{$k$}, 3.5/{$k{+}1$}, 4.5/{$k{+}2$}} {
    \node[below, font=\small] at (\x, 0) {\label};
}

\begin{scope}[shift={(0.5, 0)}]
    \fill[black!25] (-0.35, 0.1) rectangle (0.35, 2);
    \fill[black!60] (-0.35, 2) rectangle (0.35, 5);
    \draw[thick] (-0.35, 0.1) rectangle (0.35, 5);
\end{scope}

\begin{scope}[shift={(1.5, 0)}]
    \fill[black!25] (-0.35, 0.1) rectangle (0.35, 3.5);
    \fill[black!60] (-0.35, 3.5) rectangle (0.35, 5.5);
    \draw[thick] (-0.35, 0.1) rectangle (0.35, 5.5);
\end{scope}

\begin{scope}[shift={(2.5, 0)}]
    \fill[black!25] (-0.35, 0.1) rectangle (0.35, 4);
    \fill[black!60] (-0.35, 4) rectangle (0.35, 6);
    \draw[thick] (-0.35, 0.1) rectangle (0.35, 6);
\end{scope}

\begin{scope}[shift={(3.5, 0)}]
    \fill[black!25] (-0.35, 0.1) rectangle (0.35, 4.3);
    \fill[black!60] (-0.35, 4.3) rectangle (0.35, 6.2);
    \draw[thick] (-0.35, 0.1) rectangle (0.35, 6.2);
\end{scope}

\begin{scope}[shift={(4.5, 0)}]
    \fill[black!25] (-0.35, 0.1) rectangle (0.35, 3.8);
    \fill[black!60] (-0.35, 3.8) rectangle (0.35, 5.8);
    \draw[thick] (-0.35, 0.1) rectangle (0.35, 5.8);
\end{scope}

\node[font=\scriptsize, right] at (5.6, 6.5) {New check-ins};
\fill[black!60] (5.35, 6.3) rectangle (5.55, 6.7);

\node[font=\scriptsize, right] at (5.6, 5.5) {Continuing stays};
\fill[black!25] (5.35, 5.3) rectangle (5.55, 5.7);

\draw[thick, <->, gray] (1.0, 5) -- (1.0, 7);
\node[left, font=\scriptsize, gray] at (0.95, 6) {$\approx \delta C$};

\end{tikzpicture}
\caption{Capacity evolution with multi-day stays. Each day's occupancy combines continuing stays (light gray) and new check-ins (dark gray). Approximately $\delta C$ rooms turn over each day---this turnover rate determines both available capacity for new arrivals and the scale of walk-in demand needed for the busy season assumption.}
\label{fig::capacity_evolution}
\end{figure}

\begin{example}[Notation in context]
To anchor the notation, consider a hotel with $C=500$ rooms operating over $T=365$ days, with a booking window of $k_0=30$ days. Guests stay for a geometrically distributed number of nights with $q=0.8$, so $\delta = 1-q = 0.2$ (20\% of guests depart each day, corresponding to an average stay of $1/\delta = 5$ nights). The confidence parameter is $\iota = \log(CT) = \log(182500) \approx 12$.

To illustrate a specific customer's journey, consider Customer $i$ who makes a reservation at time $t=8.5$ for Day $k=10$ (i.e., 1.5 days in advance). Her booking is represented by the tuple $(X_i^{(0,10)}, Y_i^{(0,10)}, D_i^{(0,10)}) = (8.5, 1, 3)$:
\begin{itemize}
    \item $X_i^{(0,10)} = 8.5$: booking time (Day 8, afternoon).
    \item $Y_i^{(0,10)} = 1$: she keeps the reservation through Stage I (does not cancel before Day 10).
    \item $D_i^{(0,10)} = 3$: she stays for 3 nights (Days 10, 11, 12).
\end{itemize}
Since she booked 1.5 days ahead, her Stage I survival probability is $p^{(10)}(8.5) \approx 0.85$---relatively high because late bookings rarely cancel. If she does not cancel by the confirmation call at time $10+\nu$, she becomes a confirmed arrival.

This example also illustrates the capacity dynamics. With $\delta = 0.2$ and $C = 500$, the steady-state turnover is approximately $\delta C = 100$ departures per day. The busy season assumption (Assumption~\ref{assump:busy_season}) requires walk-in demand $\lambda^{(2,k)} \gtrsim \iota + \sqrt{C\delta\iota} \approx 12 + \sqrt{500 \times 0.2 \times 12} \approx 47$ customers per day---less than half the turnover rate. The threshold scales as $O(\sqrt{C\delta\log T})$, growing much slower than capacity.
\end{example}


\subsection{Analysis Framework and Offline Benchmark}\label{sec:benchmark}
Let $\pi_1$ and $\pi_2$ denote admissible policies for Stage I and Stage II. $\pi_1$ governs acceptance of booking requests in Stage I; $\pi_2$ governs acceptance of check-in requests in Stage II. Denote $z_i^{(0,k)}$ as the acceptance decision for the $i$-th booking request for Day $k$ in Stage I ($z_i^{(0,k)} = 1$ for accept, $0$ for reject). Similarly, $z_i^{(1,k)}$ and $z_i^{(2,k)}$ denote acceptance decisions for the $i$-th Type I and Type II arrivals in Stage II. We write $\pi_1=\{z_i^{(0,k)}\}_{i,k}$ and $\pi_2=\{z_i^{(1,k)},z_i^{(2,k)}\}_{i,k}$. Define $\Pi=\Pi_1\otimes\Pi_2$ as the set of admissible policies satisfying capacity constraints.

The decision maker faces a sequential optimization problem: minimize total expected loss over $T$ days by balancing two competing costs. The first cost arises from overbooking—when a reservation holder arrives but no room is available, incurring penalty $\loss^{(k)}$ per rejected customer. The second cost arises from idle capacity—unused resources forfeit potential revenue $r^{(k)}$ per unit. The challenge lies in making accept/reject decisions without advance knowledge of arrivals, while multi-day occupancy durations create complex inter-day dependencies in both capacity constraints and loss functions.

The two-stage online allocation problem is
\begin{equation}
    \label{eq:prob_stage_I}
    \begin{aligned}
        \min_{(\pi_1,\pi_2)\in\Pi} \quad & \mathcal{L}_1(\pi_1,\pi_2) := \mathbb{E} \left[ \sum_{k=1}^T \mathcal{L}_1^{(k)}(\pi_2, \{B^{(k')}(\pi_1)\}_{k'=1}^{k}) \right] \\
        \text{s.t.} \quad & B^{(k)}(\pi_1) = \sum_{i=1}^{\Lambda^{(1,k)}[k-k_0,k]} Y_i^{(0,k)} z_i^{(0,k)}, \quad \forall k \in [T], \\
        &\sum_{k'=1}^k \left( \sum_{i=1}^{B^{(k')}(\pi_1)} Y_i^{(1,k')} z_i^{(1,k')} \mathbbm{1}\{D_i^{(1,k')} \geq k-k'\} \right. \\
        &\qquad\qquad \left. + \sum_{i=1}^{\Lambda^{(2,k')}[0,1]} z_i^{(2,k')} \mathbbm{1}\{D_i^{(2,k')} \geq k-k'\} \right) \leq C, \quad \forall k \in [T], \\
        & z_{i_1}^{(0,k)}, z_{i_2}^{(1,k)}, z_{i_3}^{(2,k)} \in \{0,1\}, \quad \forall k \in [T], i_1 \in [\Lambda^{(1,k)}[k-k_0,k]],\\
        &\qquad\qquad  \, i_2 \in [B^{(k)}(\pi_1)], \, i_3 \in [\Lambda^{(2,k)}[k-k_0,k]].
    \end{aligned}
\end{equation}

where the loss function $\Lossii^{(k)}(\pi_2,\{B^{(k')}(\pi_1))\}_{k'=1}^{k})$ for Stage II is given by
\begin{equation}
    \label{eq:loss_stage_II}
    \begin{aligned}
        &\Lossii^{(k)}(\pi_2,\{B^{(k')}(\pi_1))\}_{k'=1}^{k})
        = \loss^{(k)}\sum_{i=1}^{B^{(k)}}Y_i^{(1,k)}(1-z_i^{(1,k)})\\ 
        &\quad+r^{(k)}\prn{C-\sum_{k'=1}^k\prn{\sum_{i=1}^{B^{(k')}}Y_i^{(1,k')}z_i^{(1,k')}\mathbbm 1\{D_i^{(1,k')}\ge k-k'\}+\sum_{i=1}^{\Lambda^{(2,k')}[0,1]}z_i^{(2,k')}\mathbbm 1\{D_i^{(2,k')}\ge k-k'\} } }.
    \end{aligned}
\end{equation}
The first term $\loss^{(k)}\sum_{i=1}^{B^{(k)}}Y_i^{(1,k)}(1-z_i^{(1,k)})$ captures overbooking loss—customers turned away despite holding confirmed reservations. The second term $r^{(k)}(C - \text{occupied})$ captures idle capacity loss—revenue foregone from unoccupied rooms. Varying occupancy durations create correlations across days in both capacity constraints and loss functions, yielding a high-dimensional state space. Dynamic programming becomes computationally infeasible and offers limited insight into effective policy design. We instead propose intuitive online policies that provide structural insights while achieving near-optimal performance relative to the offline benchmark (which assumes complete advance knowledge of arrivals).


We benchmark our policies against the Hindsight Optimal (\HO), which has full advance knowledge of all bookings, check-ins, and cancellations. The \HO~ benchmark optimally controls both stages subject to capacity constraints. Let filtrations $\{\Fi^{(k)},\Fii^{(k)}\}$ contain all realized information from Stage I and Stage II up to (and including) Day $k$. Decompose the hindsight optimal policy as $\pi^\star=(\pi_1^{\star},\pi_2^{\star})$, where $\pi_1^{\star}$ and $\pi_2^{\star}$ access filtrations $\{\Fi^{(k)}\}_{k=1}^{T}$ and $\{\Fii^{(k)}\}_{k=1}^{T}$. The regret is
\begin{align}
    \label{equ::expected regret}\regret^T(\pi_1,\pi_2,\Fi^{(T)},\Fii^{(T)})={\Lossi(\pi_1,\pi_2)-\Lossi(\pi^{\star}_1,\pi^{\star}_2)}.
\end{align}
Decompose expected regret into contributions from Stage I and Stage II decisions:
\begin{align}
\regret^T(\pi_1,\pi_2)&=\Lossi(\pi_1,\pi^{\star}_2)-\Lossi(\pi^{\star}_1,\pi^{\star}_2)+\Lossi(\pi_1,\pi_2)-\Lossi(\pi_1,\pi^{\star}_2) \label{equ::regret decomposition}\\
&=\sum_{k=1}^T\E_{\Fi^{(T)},\Fii^{(T)}}\brk{\underbrace{\Lossii^{(k)}(\pi^{\star}_2,\{B^{(k')}(\pi_1))\}_{k'=1}^{k})-\Lossii^{(k)}(\pi^{\star}_2,\{B^{(k')}(\pi^{\star}_1))\}_{k'=1}^{k})}_{\mathcal{R}_1^{(k)}:~\text{Regret from Stage I on Day $k$}}}. \label{equ::Stage I regret}\\
&\quad+ \sum_{k=1}^T\E_{\Fi^{(T)},\Fii^{(T)}}\brk{\underbrace{\Lossii^{(k)}(\pi_2,\{B^{(k')}(\pi_1))\}_{k'=1}^{k})-\Lossii^{(k)}(\pi^{\star}_2,\{B^{(k')}(\pi_1))\}_{k'=1}^{k})}_{\mathcal{R}_2^{(k)}:~\text{Regret from Stage II on Day $k$}}}\label{equ::Stage II regret}
\end{align}
The following sections provide regret guarantees and sensitivity analysis on interactions between model elements.

We analyze the two stages separately. Section~\ref{sec:Stage I} addresses Stage~I: we establish a lower bound showing $\Omega(T)$ regret without sufficient walk-ins, then present our DASS-I policy achieving $O(1)$ regret under the \emph{busy season} condition. Section~\ref{sec:Stage II} addresses Stage~II: we formalize the decoupling property of DASS-I and present DASS-II with sensitivity analysis on confirmation timing. Section~\ref{sec:nested} extends to heterogeneous customers with nested capacity protection.


\section{Regret Analysis for Stage I}\label{sec:Stage I}

We now analyze Stage I—the reservation acceptance phase. Our analysis reveals a fundamental dichotomy: constant regret is achievable if and only if walk-in demand is sufficiently high, a condition we term the \emph{busy season}. We first establish this necessity via a lower bound (Section \ref{sec:lower_bound}), then present our DASS-I policy achieving $O(1)$ regret under the busy season condition (Section \ref{sec:algorithm_stage_i}).

We present our Stage I policy—the booking control process. Section \ref{sec:lower_bound} introduces the busy season assumption and demonstrates that any online policy incurs $\Omega(T)$ regret when the busy season condition fails. Section \ref{sec:algorithm_stage_i} details our Decoupled Adaptive Safety Stock-I (\textbf{DASS-I}) algorithm and shows that it achieves constant regret for Stage I decisions as specified in \eqref{equ::Stage I regret}.

\subsection{Busy Season and Lower Bound}\label{sec:lower_bound}
This section provides intuition on the role of walk-in customers in the reusable resource allocation problem.

\textbf{Why do we need walk-ins?} Without walk-in customers, the manager faces a dilemma: accept too few reservations and risk idle rooms; accept too many and risk overbooking. The fundamental difficulty is that cancellations are random—even with perfect knowledge of arrival rates, the manager cannot predict exactly how many reservations will materialize. Walk-ins provide a ``safety valve'': if reservations fall short of expectations, walk-ins can fill the gap; if reservations exceed expectations, the manager can simply reject walk-ins. The following assumption quantifies the minimum walk-in rate needed for this safety valve to work effectively.

\begin{assumption}[Busy seasons]
\label{assump:busy_season}
    For any $k\in[T]$, we assume that:
    \begin{itemize}
        \item[(i)] The reservation arrival rate $\lambda^{(1,k)}$ slightly exceeds expected departure rates.
        \item[(ii)] The walk-in arrival rate $\lambda^{(2,k)}$ is at least $\Omega(\iota+\sqrt{C\delta\iota})$, where $\delta$ measures expected occupancy duration and $\iota=O(\log(CT))$ is a confidence parameter.
    \end{itemize}
   Here $\delta:=$ (a) $1 - q$ when occupancy duration follows a geometric distribution with parameter $q$, or (b) $1/d$ when occupancy duration is constant $d$. See Appendix \ref{appendix::busy_season} for precise conditions.
\end{assumption}

\begin{example}[Busy season threshold calculation]
Consider $C=500$ rooms, $T=365$ days, and $\delta = 0.2$ (geometric with $q=0.8$, average stay of 5 nights). The confidence parameter is $\iota = \log(CT) = \log(182500) \approx 12.1$, corresponding to a failure probability of $\exp(-\iota) \approx 5.5 \times 10^{-6}$ per day---negligible over the planning horizon. The busy season threshold becomes:
\[
\lambda^{(2,k)} \gtrsim \iota + \sqrt{C\delta\iota} \approx 12.1 + \sqrt{500 \times 0.2 \times 12.1} \approx 12.1 + 34.8 = 47 \text{ walk-ins/day}.
\]
With turnover rate $\delta C = 100$ departures per day, this threshold is less than half the daily capacity turnover---easily met by most urban hotels. The threshold scales as $O(\sqrt{C\delta\log T})$, so larger capacity or longer horizons require proportionally more walk-ins, but only at a square-root rate.
\end{example}

Condition (i) is common in service systems (e.g., \citet{chen2017revenue}, \citet{dong2024value}). Condition (ii) requires sufficient walk-in demand to offset uncertainties from cancellation and no-show, with the confidence bound of level $\exp(-\iota)$ over horizon $T$ (see \citet{lattimore2020bandit}). To underscore the tightness of Assumption \ref{assump:busy_season} up to logarithmic factors, 
we demonstrate that the performance gap between any online policy and the hindsight optimal policy can grow linearly in $T$ when the uncertainty of either occupancy duration or cancellation cannot be offset by simultaneously exploiting reservations and walk-ins. In contrast to the $O(1)$ regret achieved in revenue management with non-reusable resources (e.g., \citet{jasin2012re, bumpensanti2020re, vera2021bayesian}), our reusable resource allocation model introduces two additional sources of uncertainty: cancellation and occupancy duration. Since Type I reservations can be cancelled without penalty before confirmation, the decision maker must overbook to offset potential idle resource loss. However, to avoid overbooking loss, the total number of reservations cannot exceed idle capacities by too much. With overlapped occupancy durations across days, numerous reservations for upcoming days make the decision maker more conservative with current-day reservations. This tradeoff results in a positive probability of daily loss, accumulating to $\Omega(T)$ total regret—explaining the $\Omega(T/C)$ regret observed in existing online reusable resource allocation problems (e.g., \citet{feng2021online, gong2022online, zhang2022online}). Typically, these models require capacity to scale linearly with $T$ to mitigate such uncertainties, whereas service systems with reusable resources repeatedly allocate limited resources to possibly overcrowded arrivals.

\begin{example}[Lower bound illustration]
Consider $C=500$ rooms, geometric duration with $\delta = 0.2$, and suppose reservation demand is $\lambda^{(1,k)} = 120$ (slightly above turnover rate $\delta C = 100$). If walk-in demand is only $\lambda^{(2,k)} = 20$ customers per day (violating the busy season threshold of approximately 47):
\begin{itemize}
    \item \textit{Conservative strategy}: Accept 100 reservations to avoid overbooking. Expected show-ups $\approx 100 \times 0.9 = 90$, leaving 10 rooms idle. With only 20 walk-ins expected, the gap can often be filled, but variance causes positive probability of loss each day.
    \item \textit{Aggressive strategy}: Accept 120 reservations. Expected show-ups $\approx 108$, risking overbooking when cancellations are fewer than expected.
\end{itemize}
In either case, the manager faces constant probability of loss on each day---from overbooking when aggressive, from idling when conservative---accumulating to $\Omega(T)$ total regret. Only with sufficient walk-ins ($\lambda^{(2,k)} \gtrsim 47$) can the manager use a conservative strategy while relying on walk-ins to fill capacity gaps.
\end{example}

We now formally state the lower bound in Proposition \ref{prop:lower_bound}.
\begin{proposition}[Loss lower bound]\label{prop:lower_bound}
    Assume $\lambda^{(2,k)}\le \sqrt{\delta C\cdot \iota }$ for any $k \in [T]$. Then there exist instances and constant $c>0$ such that for any online policy $({\pi}_1,{\pi}_2)$, the regret grows linearly in $T\exp(-c\iota)$:
    \begin{align*}
        \mathrm{Regret}^T(\pi_1,\pi_2)\ge \sum_{k=1}^{T}\min(\ell^{(k)}_{\text{over}},r^{(k)})\ge \min \{\min(\ell^{(k)}_{\text{over}},r^{(k)})\}_{k=1}^{T}\exp(-c\iota)\cdot T .
    \end{align*}
    Hence $\mathrm{Regret}^T(\pi_1,\pi_2) = \Omega(T)$ when $\iota = O(1).$
\end{proposition}
The proof decomposes regret into overbooking and idling losses. Under the assumption in Proposition \ref{prop:lower_bound}, when the DM overbooks in Stage I, there is a constant probability of incurring an overbooking loss. Conversely, if no reservations are overbooked, there remains probability of order $\exp(-c\iota)$ of an idling loss due to insufficient walk-ins. Hence, regardless of the policy choice, a positive probability of loss occurs each day, accumulating to $\Omega(T)$ total regret over the horizon. Detailed proofs are in Appendix \ref{appendix:proof}. Our result demonstrates the important role of walk-ins in offsetting uncertainties from cancellations and occupancy durations, underscoring the necessity of adaptive algorithms that exploit both customer types while maintaining safety guards.

\subsection{Algorithm Framework for Stage I}\label{sec:algorithm_stage_i}
We design DASS-I to decouple occupancy correlations across days through adaptive safety stocks. The policy maintains dynamic thresholds that balance acceptance against available resources while accounting for cancellation uncertainty and occupancy duration.

\begin{figure}[ht]
\centering
\begin{tikzpicture}[x=1cm, y=0.8cm]

\draw[ultra thick, -{Latex[length=2.5mm]}] (0,0) -- (9.5,0) node[anchor=north west] {Time $t$};
\draw[ultra thick, -{Latex[length=2.5mm]}] (0,0) -- (0,6) node[anchor=south east] {Bookings};

\node[below] at (0, 0) {$k-k_0$};
\node[below] at (8, 0) {$k$};
\draw[thick] (8, 0.1) -- (8, -0.1);

\fill[gray!20, opacity=0.5] plot[smooth, tension=0.6] coordinates {(0, 2) (2, 2.5) (4, 3.5) (6, 4.5) (8, 5)} -- (8, 6) -- (0, 6) -- cycle;
\node[font=\bfseries, gray!60!black] at (1.5, 5.5) {Reject};
\node[font=\bfseries, gray!60!black] at (7.5, 1.5) {Accept};

\draw[thick, dashed, red!70!black] (0, 5) -- (9, 5) node[midway, above, font=\small] {$\hat{C}^{(k)}$ (Conservative Capacity)};

\draw[ultra thick, blue] plot[smooth, tension=0.6] coordinates {(0, 2) (2, 2.5) (4, 3.5) (6, 4.5) (8, 5)};
\node[above left, blue, font=\small, rotate=10] at (6.2, 4.6) {Acceptance Threshold $\hat{B}_t^{(k)}$};

\draw[thick, blue, dashed] plot[smooth, tension=0.6] coordinates {(0, 1.2) (2, 1.8) (4, 2.8) (6, 3.9) (8, 4.8)} node[below right, font=\scriptsize] {$\mathbb{E}[\text{Show-ups}]$};

\draw[<->, thick, gray] (4.5, 3.08) -- (4.5, 3.75);
\node[right, gray, font=\scriptsize, align=left] at (4.6, 3.4) {Safety\\Stock};

\draw[thick, black] (0, 0) -- (1, 0) -- (1, 1) -- (2.5, 1) -- (2.5, 2) -- (4, 2) -- (4, 3) -- (5.5, 3) -- (5.5, 4) -- (6.5, 4) -- (6.5, 5);
\node[font=\small] at (4.75, 2.5) {Accepted $B_t^{(k)}$};

\end{tikzpicture}
\caption{Schematic of DASS-I Booking Control. The policy maintains a dynamic acceptance threshold $\hat{B}_t^{(k)}$ (solid blue line) that lies above the expected show-ups (dashed blue line) by a safety margin. Reservations (black stepped line) are accepted as long as they stay below this threshold and the conservative capacity $\hat{C}^{(k)}$. As time progresses and cancellation risk decreases ($p^{(k)}(t) \to 1$), the safety stock shrinks, allowing the threshold to converge towards capacity.}
\label{fig::dass_I_schematic}
\end{figure}

\textbf{Intuition: safety stocks for overbooking control.} The key question for Stage I is: ``How many reservations can we safely accept for Day $k$?'' The answer depends on how many reservations will actually show up, which is uncertain due to cancellations. The threshold $\hat{B}_t^{(k)}$ answers a related question: ``How many confirmed guests can we \emph{expect} from current reservations?'' Since cancellations are random, we cannot use the expected value alone—we must add a safety margin to avoid overbooking with high probability. This margin shrinks as $p^{(k)}(t) \to 1$ (closer to service time, fewer cancellations possible), allowing more reservations to be accepted as uncertainty resolves. Figure \ref{fig::dass_I_schematic} illustrates this mechanism, showing how the acceptance threshold stays above the expected show-up curve by a safety margin that tightens over time.

For service starting on Day $k$, we maintain a dynamic threshold $\hat{B}^{(k)}_{t}$ that upper-bounds the expected number of non-cancelled Type I reservations with high probability. Let ${B}^{(k)}_{t}$ denote the number of accepted Type I reservations not yet cancelled at time $t\le k$. The threshold takes the form
\begin{align}\label{eq:hat_B_t}
    \hat{B}^{(k)}_{t} \approx p^{(k)}(t) B^{(k)}_{t} + \text{safety stock},
\end{align}
where $p^{(k)}(t)$ is the confirmation rate and the safety stock accounts for randomness in cancellations. Similarly, we estimate the remaining resource capacity as
\begin{align}\label{eq:hat_C}
\hat{C}^{(k)} \approx C - \delta \lambda^{(1,k)} - \text{safety buffer},
\end{align}
where $\delta$ measures expected occupancy duration. A booking request at time $t$ for Day $k$ is rejected if $\hat{B}^{(k)}_{t} \ge \hat{C}^{(k)}$. The exact formulas with precise safety stock and buffer terms are given in \eqref{eq:hat_B_t_full} and \eqref{eq:hat_C_geometric}--\eqref{eq:hat_C_constant} in Appendix \ref{appendix::DASS_formulas}.

The safety stock and buffer terms are derived using Bernstein's inequality to upper-bound deviations from expected values with probability $1-\exp(-\iota)$, where $\iota = \Theta(\log(CT))$ is the confidence parameter. Specifically, the safety stock in $\hat{B}_t^{(k)}$ accounts for the binomial randomness in cancellations: since each of the $B_t^{(k)}$ reservations cancels independently with probability $1-p^{(k)}(t)$, we add a term scaling as $\sqrt{B_t^{(k)} \cdot \iota}$ to ensure $\hat{B}_t^{(k)}$ exceeds the true confirmed count with high probability. By controlling the reservation pool $B_t^{(k)}$ to stay within the conservative capacity $\hat{C}^{(k)}$, we ensure with high probability that Day $k$ does not overbook. Crucially, this guarantee depends only on Day $k$'s state—not on future days—enabling the decoupling formalized in Lemma \ref{lemma::decoupling} (Section \ref{sec:Stage II}).

Based on these quantities, we present our Decoupled Adaptive Safety Stock-I (\textbf{DASS-I}) policy in Algorithm \ref{alg:DASS_stage_I}. We denote this Stage I policy as $\hat{\pi}_1$. As ``decoupled'' suggests, our Stage I decision for Day $k$ does not rely on the status of other days—we construct adaptive safety stocks to hedge uncertainties of cancellation and departures using only Day $k$ information. These safety stocks eliminate correlations between occupancy durations on different days, discussed further in Section \ref{sec:Stage II}.
\bigskip
\begin{breakablealgorithm}
    \caption{Decoupled Adaptive Safety Stock-I (\textbf{DASS-I}) for Day $k$}
    \label{alg:DASS_stage_I}
    \begin{algorithmic}[1]
        \Require{Cancellation probabilities $p^{(k)}(\cdot)$, $q^{(1,k)}$, capacity $C$, occupancy parameters $q$ or $d$, confidence parameter $\iota$.}
        \State{Initialize: $B_{k-k_0}^{(k)}\leftarrow 0$.}
        \State{Compute estimated capacity $\hat{C}^{(k)}$ via \eqref{eq:hat_C_geometric} or \eqref{eq:hat_C_constant}.}
    \For{each reservation request arriving at time $t\in [k-k_0,k]$ for day $k$}
            \State{Compute booking threshold $\hat B_{t}^{(k)}$ based on current bookings $B_t^{(k)}$ via \eqref{eq:hat_B_t_full}.}
            \If{$\hat B_{t}^{(k)} < \hat{C}^{(k)}$}
                \State{Accept the reservation and update $B_t^{(k)}\leftarrow B_t^{(k)}+1$.}
            \Else
                \State{Reject the reservation.}
            \EndIf
        \EndFor
    \end{algorithmic}
\end{breakablealgorithm}
\bigskip

The thresholds $\hat{C}^{(k)}$ and $\hat{B}_t^{(k)}$ are designed as high-probability bounds: $\hat{C}^{(k)}$ conservatively estimates available capacity accounting for future departures, while $\hat{B}_t^{(k)}$ upper-bounds expected confirmed bookings with a safety margin for cancellation uncertainty.

\begin{example}[DASS-I in action]
Consider Day $k=10$ with $C=100$ rooms, geometric duration ($q=0.8$, so $\delta = 0.2$), and confidence parameter $\iota = 10$. At time $t=9.5$ (half a day before service), suppose we have accepted $B_t^{(10)}=80$ reservations for Day $10$, and the confirmation rate is $p^{(10)}(9.5)=0.9$.
\begin{itemize}
    \item \textit{Expected confirmations:} $0.9 \times 80 = 72$ guests expected to show up.
    \item \textit{Safety margin:} $\sqrt{B_t^{(10)} \cdot p^{(10)}(t)(1-p^{(10)}(t)) \cdot \iota} \approx \sqrt{80 \times 0.9 \times 0.1 \times 10} \approx 8.5$.
    \item \textit{Threshold:} $\hat{B}_t^{(10)} \approx 72 + 8.5 = 80.5$.
\end{itemize}
Suppose the conservative capacity is $\hat{C}^{(10)} = 82$ (accounting for ongoing occupancies from previous days). Since $\hat{B}_t^{(10)} = 80.5 < 82 = \hat{C}^{(10)}$, the next reservation request is \emph{accepted}. If instead $\hat{C}^{(10)} = 79$, then $\hat{B}_t^{(10)} = 80.5 \geq 79$, so the request is \emph{rejected} to avoid overbooking risk.
\end{example}

How does the algorithm design enable $O(1)$ regret? Three design choices in DASS-I are essential:
\begin{enumerate}[label=(\roman*)]
    \item \textit{Bernstein-based safety stock}: The $\sqrt{B_t^{(k)} \cdot \iota}$ term in $\hat{B}_t^{(k)}$ ensures that confirmed bookings stay below $\hat{C}^{(k)}$ with probability $1 - e^{-\iota}$, controlling overbooking risk. This is tighter than Hoeffding bounds when $B_t^{(k)}$ is moderate.

    \item \textit{Conservative capacity $\hat{C}^{(k)}$}: Reserving $\delta\lambda^{(1,k)}$ capacity for ongoing occupancies from previous days prevents cascading overbooking across days. The buffer accounts for geometric (memoryless) or constant ($C/d$ quota) duration structures.

    \item \textit{Day-independent thresholds}: Both $\hat{B}_t^{(k)}$ and $\hat{C}^{(k)}$ depend only on Day $k$ data—not on other days' states. This enables decoupling: Day $k$'s analysis is self-contained.
\end{enumerate}
Together, these bound Stage I regret by $O(CT \cdot e^{-\iota})$. Setting $\iota = \Theta(\log(CT))$ yields $O(1)$ total regret.

We characterize the performance of DASS-I in Theorem \ref{prop::StageI regret}. Without loss of generality, for geometrically distributed occupancy duration, we assume the service system runs with full capacity on Day $0$ to align with the busy season setting.
\begin{theorem}
 [Regret upper bound for Stage I]\label{prop::StageI regret}
    Suppose on Day $0$ the system is at full capacity under geometric duration, or there are at most $C/d$ customers that checked in $k$ days before Day $1$ for each $k\in[d]$. Then our policy $\hat{\pi}_1$ yields expected regret in Stage I bounded by
    \begin{align}
     \E\left[\sum_{k=1}^{T}\mathcal{R}_1^{(k)}(\hat{\pi}_1)\right]  &\le   5CT\exp(-\iota)\cdot \max_{k}r^{(k)}+3CT\exp(-\iota)\cdot \max_{k}\ell^{(k)}+{C}_0.
    \end{align}
\end{theorem}
Here ${C}_0=0$ for geometric duration and ${C}_0=\sum_{k=1}^{d}\frac{C(d-k)}{d}\cdot r^{(k)}$ for constant duration. The proof is in Appendix \ref{appendix::stage I}. Leveraging the safety stock in \eqref{eq:hat_B_t} as a buffer, we avoid overbooking with high probability. When the walk-in rate $\lambda^{(2,k)}$ exceeds this buffer, idling loss from vacant rooms is compensated by walk-in arrivals on most days, yielding constant regret.


\section{Regret Guarantee and Sensitivity Analysis for Stage II}\label{sec:Stage II}
We first elaborate on the decoupling technique in \textbf{DASS-I}, showing that with high probability, the original reusable resource allocation problem decomposes into a series of non-reusable resource allocation subproblems. We then present the Stage II allocation algorithm \textbf{DASS-II}, along with its regret guarantee and sensitivity analysis of confirmation call timing.


\subsection{Decision Decoupling}

\textbf{Why decoupling matters.} Multi-day occupancy creates dependencies: accepting a guest on Day $k$ reduces capacity on Days $k+1, k+2, \ldots$ until the guest departs. In principle, optimizing Day $k$ requires anticipating all future days—an exponentially complex task. The key insight of our approach is that under DASS-I, these dependencies vanish with high probability. Each day can be optimized independently, reducing a $T$-dimensional problem to $T$ independent single-day problems. This dramatic simplification is possible because (i) DASS-I's safety stocks ensure overbooking events are rare, and (ii) the memoryless property of geometric duration (or the fixed quota $C/d$ for constant duration) makes future capacity independent of past acceptance decisions.

Under DASS-I, Stage II decisions can be made independently for each day, treating the problem as if resources are non-reusable. The key insight is that our Stage I policy controls non-cancelled customers tightly enough to prevent overbooking, while the memoryless property of geometric duration (or the fixed capacity $C/d$ per day for constant duration) eliminates dependencies between days. The following lemma formalizes this decoupling result.
 \begin{lemma}\label{lemma::decoupling}
     Suppose the walk-in rate requirement \eqref{equ::require_lambda_appendix} holds for any $k\in[T]$. Under our Stage I \textbf{DASS-I} policy $\hat{\pi}_1$, with probability at least $1-8T\exp(-\iota)$, on any Day $k$, the decoupled policy
      \begin{align*}
     \hat{\pi}_2^{\star}=\pi_{2}^{(1)\star} \times \pi_{2}^{(2)\star} \times \cdots \times \pi_{2}^{(T)\star}
 \end{align*}
 achieves constant regret relative to the overall hindsight optimal policy $\pi_2^{\star}$, where both apply \textbf{DASS-I} in Stage I:
 \begin{align*}
  \sum_{k=1}^T\brk{\Lossii^{(k)}(\hat{\pi}^{\star}_2,\{B^{(k')}(\hat\pi_1))\}_{k'=1}^{k})-\Lossii^{(k)}(\pi^{\star}_2,\{B^{(k')}(\hat\pi_1))\}_{k'=1}^{k})}\le C_0.
 \end{align*}
 Here ${C}_0=0$ for geometric duration and ${C}_0=\sum_{k=1}^{d}\frac{C(d-k)}{d}\cdot r^{(k)}$ for constant duration. For $k=1,2,\dots,T$, the decoupled single-day offline optimal policy $\hat\pi_2^{(k)\star}$ on Day $k$ is
 \begin{align*}
    \begin{aligned}    &\hat{\pi}_2^{(k)\star}=\argmin_{\pi_2\in\Pi^{(k)\star}_2} \Lossii^{(k)}(\pi^{(k)}_2, \{B^{(k')}(\hat\pi_1)\}_{k'=1}^{k}) , \\
       & \text{s.t.} \quad  B^{(k)}(\pi_1) = \sum_{i=1}^{\Lambda^{(1,k)}(k-k_0,k)} Y_i^{(0,k)} z_i^{(0,k)}, \\
& \sum_{i=1}^{B^{(k)}(\pi_1)} Y_i^{(1,k)} z_i^{(1,k)}   + \sum_{i=1}^{\lambda^{(2,k)}} z_i^{(2,k)}  \le \Tilde{C}^{(k)},
    \end{aligned}
 \end{align*}
where $\Tilde{C}^{(k)}$ is a predetermined room capacity defined as 
\begin{align}
    \Tilde{C}^{(k)}:= \begin{cases}
  &C-  \sum_{k'=1}^{k-1} \Big( \sum_{i=1}^{B^{(k')}(\hat\pi_1)} Y_i^{(1,k')} z_i^{(1,k')} \mathbbm{1}\{D_i^{(1,k')} \geq k-k'\}  \\&\quad+ \sum_{i=1}^{\Lambda^{(2,k')}[0,1]} z_i^{(2,k')} \mathbbm{1}\{D_i^{(2,k')} \geq k-k'\} \Big) ~~(\textbf{Geometric})  \\
 & \frac{C}{d}~~(\textbf{Constant})
\end{cases}  \label{equ::allocated capacity}
\end{align}
and $\pi_2^{(k)}\in\Pi^{(k)\star}_2$ is taken over all offline policies with full access to Stage I and Stage II information up to Day $k$, including results from previous decisions before Day $k$: $B^{(k')}(\hat\pi_1)$, $z_i^{(1,k')}$ and $z_i^{(2,k')}$ for $k' <k$.
 \end{lemma}
 The proof relies on two key observations. First, DASS-I's safety stock construction ensures that the number of confirmed Type I customers on Day $k$ stays within the conservative threshold $\hat{C}^{(k)}$ with probability at least $1-8T\exp(-\iota)$. This tight control prevents overbooking while preserving a large proportion of Type I customers for Stage II. Second, under geometric duration, the memoryless property implies that on Day $k$, the number of ongoing occupancies from previous days is independent of Day $k$'s acceptance decisions—each customer departs with probability $q$ regardless of when they checked in. For constant duration $d$, we enforce a fixed quota of at most $C/d$ acceptances per day, which similarly decouples across days by capping total occupancy at $C$. Together, these properties ensure that the single-day optimizer $\hat{\pi}_2^{(k)\star}$ acting on residual capacity $\tilde{C}^{(k)}$ achieves near-global-optimality. Detailed proof is in Appendix \ref{appendix::decoupling}. For constant duration, $C_0$ quantifies loss due to the hard capacity constraint of providing $C/d$ rooms on each day, which may cause idle rooms during the first $d$ days.

\begin{example}[Decoupling in action]
Consider Day $k=10$ with $C=500$, geometric duration ($\delta=0.2$), and suppose Days 1--9 have already been processed under DASS-I. \textit{Without decoupling}, optimizing Day 10 would require tracking how many guests from each previous day are still occupying rooms---a state space of size $O(C^9)$ in the worst case. \textit{With decoupling}, we observe that:
\begin{itemize}
    \item DASS-I ensures each Day $k' < 10$ accepted at most $\hat{C}^{(k')}$ confirmed guests (with high probability), so overbooking cascades are rare.
    \item Under geometric duration, each guest from Day 9 departs with probability $q=0.8$ before Day 10, independently of when they arrived. This memoryless property means ongoing occupancies from Days 1--9 sum to a random variable independent of Day 10's decisions.
\end{itemize}
The residual capacity $\tilde{C}^{(10)}$ (observed at the start of Day 10) captures all relevant history. Day 10's optimization reduces to: ``given $\tilde{C}^{(10)} = 410$ available rooms and $B^{(10)} = 450$ reservations, how should we allocate rooms to Type I arrivals and walk-ins?'' This single-day problem has a simple threshold solution, avoiding the exponential state space of the joint problem.
\end{example}

  This optimal decoupled single-day policy $\pi_{2,k}^{\star}$ makes Stage II decisions on Day $k$ that greedily maximize revenue on Day $k$ with access to full customer information,\footnote{The revenue is maximized under a tighter capacity constraint for constant duration.} regardless of potential effects on future days caused by occupancy durations. The single-day optimal policy has a simple threshold structure: it prioritizes Type I customers and allocates remaining capacity to walk-ins. See Appendix \ref{appendix::Stage II} for the explicit characterization and proof.

Lemma~\ref{lemma::decoupling} shows that optimizing each day independently achieves near-global-optimality. This motivates DASS-II: instead of solving a complex multi-day problem, we design a policy that approximates the single-day offline optimal $\hat{\pi}_2^{(k)\star}$ using only online information. The residual capacity $\tilde{C}^{(k)}$ (known at the start of Day $k$) serves as a single-day budget, and DASS-II tracks expected arrivals against this budget.

\subsection{Stage II Decision: Regret and Sensitivity Analysis}
We now present our Decoupled Adaptive Safety Stock-II (\textbf{DASS-II}) policy for Stage II—the check-in control process, which exploits the decoupling established above. For a fixed Day $k$, suppose we received $B^{(k)}$ non-cancelled Type I customers in advance. For any time $u \in [0,v)$ before the confirmation call, denote the number of Type I customers who checked in and those who cancelled during $[0,u]$ by $B^{(k)}_{1,u}$ and $B^{(k)}_{2,u}$, and denote $W^{(k)}_{1,u}$ as the number of accepted walk-in customers during this period. We propose an online policy $\hat\pi_2$ that uses conditional expectations of total show-ups to decide whether to accept a new walk-in customer. After time $v$, the confirmation call reveals the exact number of Type I customers that will check in after time $v$, denoted $B^{(k)}_{3,v}$.

The expected show-ups at time $u$ is computed as
\begin{align}\label{equ::StageII standard simple}
\hat{N}^{(k)}_{u} \approx \begin{cases}
    \text{(Type I check-ins)} + q^{(1,k)} \cdot \text{(pending reservations)} \\
    \quad + \text{(Type II check-ins)} + \alpha \cdot \text{(future walk-ins)}, & u < v,\\
    \text{(confirmed Type I)} + \text{(Type II check-ins)}, & u \ge v,
\end{cases}
\end{align}
where $q^{(1,k)}$ is the confirmation rate and $0<\alpha <1$ is a tunable parameter representing the approximate proportion of walk-in customers we accept. The precise formula is given in \eqref{equ::StageII standard full}.

\begin{figure}[ht]
\centering
\begin{tikzpicture}[
    node distance=2cm,
    box/.style={rectangle, draw, rounded corners, minimum width=2.5cm, minimum height=1cm, align=center, font=\small},
    decision/.style={diamond, draw, aspect=2, inner sep=2pt, align=center, font=\small},
    arrow/.style={-{Latex[length=2mm]}, thick},
    line/.style={thick}
]

\draw[ultra thick, ->, gray] (-1, 0) -- (10, 0) node[right] {Stage II Time $u$};
\draw[thick, gray] (4.5, 0.2) -- (4.5, -0.2) node[below] {$v$ (Confirmation)};

\node[above, font=\bfseries] at (1.5, 0.2) {Phase 1: Estimation};
\node[above, font=\bfseries] at (7.5, 0.2) {Phase 2: Certainty};

\node[box, fill=blue!10] (arrival) at (4.5, 2.5) {New Walk-in Request\\ at time $u$};

\node[decision] (check) at (4.5, -3.5) {$\hat{N}_u^{(k)} < \tilde{C}^{(k)}$?};
\node[box, fill=green!10] (accept) at (2, -5.5) {Accept};
\node[box, fill=red!10] (reject) at (7, -5.5) {Reject};

\draw[arrow] (arrival) -- (check);
\draw[arrow] (check) -| node[above, near start] {Yes} (accept);
\draw[arrow] (check) -| node[above, near start] {No} (reject);

\node[align=left, font=\scriptsize, anchor=east] (est1) at (3.5, -1.5) {
    \textbf{Estimate} $\hat{N}_u^{(k)}$ uses:\\
    $\bullet$ Confirmed: $B_{1,u}^{(k)}$\\
    $\bullet$ Pending: $\mathbf{q^{(1,k)} \cdot B_{2,u}^{(k)}}$\\
    $\bullet$ Future Walk-ins: $\alpha \cdot \mathbb{E}[W]$
};
\draw[dashed, gray] (est1) -- (2, 0);

\node[align=left, font=\scriptsize, anchor=west] (est2) at (5.5, -1.5) {
    \textbf{Exact} $\hat{N}_u^{(k)}$ uses:\\
    $\bullet$ Total Confirmed: $B_{\text{final}}^{(k)}$\\
    $\bullet$ Pending Risk: \textbf{Resolved}\\
    $\bullet$ Only Walk-in uncertainty remains
};
\draw[dashed, gray] (est2) -- (7, 0);

\end{tikzpicture}
\caption{Logic flow of DASS-II Check-in Control. When a walk-in arrives, the policy compares the expected total show-ups $\hat{N}_u^{(k)}$ against the residual capacity $\tilde{C}^{(k)}$. Before confirmation (left), $\hat{N}_u^{(k)}$ relies on probabilistic estimates of pending reservations. After confirmation (right), the exact number of Type I arrivals is known, eliminating the primary source of uncertainty and allowing tighter capacity utilization.}
\label{fig::dass_II_schematic}
\end{figure}

\noindent\textit{Choosing $\alpha$.} Our theoretical analysis requires $\alpha < 1/2$ to ensure the regret bound in Theorem \ref{prop:HO_gap}. In practice, we recommend setting $\alpha$ close to 1/2 as a default. The optimal choice of $\alpha$ scales approximately with $\log(\loss^{(k)}/r^{(k)})$: when the overbooking penalty $\loss^{(k)}$ is large relative to revenue $r^{(k)}$, increase $\alpha$ toward $1/2$ to reduce overbooking risk by lowering the expected Type II acceptance rate. Conversely, when $\loss^{(k)}/r^{(k)}$ is small, a smaller $\alpha$ (e.g., $\alpha \approx 0.3$) increases Type II acceptance and improves capacity utilization. This trade-off is analyzed in detail in the proof of Theorem \ref{prop:HO_gap} (Appendix \ref{appendix::Stage II}).

We accept the walk-in customer arriving at time $u$ only if $\hat{N}^{(k)}_u < \Tilde{C}^{(k)}$. For any Type I customer arriving at time $u$, we offer a room if the number of newly occupied rooms $B^{(k)}_{1,u}+W^{(k)}_{1,u}$ is less than the room capacity $\Tilde{C}^{(k)}$. Algorithm \ref{alg:DASS_stage_II} details our decision process, and Figure \ref{fig::dass_II_schematic} provides a visual flowchart of this decision logic, highlighting the shift from estimation to certainty at the confirmation call.

\bigskip
\begin{breakablealgorithm}
    \caption{Decoupled Adaptive Safety Stock-II (\textbf{DASS-II}) for Day $k$}
    \label{alg:DASS_stage_II}
    \begin{algorithmic}[1] 
        \Require{Arrival rates $\lambda^{(2,k)}(\cdot)$ of Type II, distribution of cancellation $q^{(1,k)}(\cdot)$, occupancy duration parameter $q$ ($d$) reward $r^{(k)}$, overbooking loss $\loss^{(k)}$, capacity $C$, time horizon $T$.}
        \State{Receive $B^{(k)}$ non-cancelled reservations from Algorithm \ref{alg:DASS_stage_I}.}
        \State{Set the room capacity $\Tilde{C}^{(k)}$ as per \eqref{equ::allocated capacity}.}
    \For{$i=1,2,\dots,\Lambda^{(1,k)}[k-k_0,k]$}
            \State{Observe the \( i \)-th reservation request for day \( k \) at time $u_{i}=X_i^{(2,k)}$.}
            \State{Provide rooms for Type I customers who checked in during $[u_{i-1},u_i)$ ($u_0=0$).}
          \State{Compute expected show-ups $\hat{N}^{(k)}_{u}$ via \eqref{equ::StageII standard full}.}
            \If{\( \hat{N}^{(k)}_{u} < \Tilde{C}^{(k)}  \)}
                \State{Accept the \(i \)-th reservation request.}
            \EndIf
        \EndFor
                
    \end{algorithmic}
\end{breakablealgorithm}
\bigskip

The expected show-ups $\hat{N}^{(k)}_{u}$ serves as an adaptive safety stock to avoid wrong decisions compared to the decoupled single-day offline optimal policy $\pi^{(k)\star}$. This approach is inspired by LP-based approximation in non-reusable resource allocation problems (e.g., \citet{jasin2012re}, \citet{rusmevichientong2020dynamic}, \citet{vera2021bayesian}). In reusable resource allocation, such approximations are generally less accurate due to additional uncertainties from cancellations and occupancy durations. However, the decoupling technique enables us to approximate the complex system through single-day decomposition and corresponding estimators up to constant order, bounding wrong decisions relative to the hindsight optimal policy with high confidence.

\begin{example}[DASS-II in action]
Consider Day $k$ with residual capacity $\tilde{C}^{(k)} = 400$, confirmation time $v = 0.5$, confirmation rate $q^{(1,k)} = 0.9$, and walk-in rate $\lambda^{(2,k)} = 50$ (uniform over $[0,1]$). Suppose we received $B^{(k)} = 420$ non-cancelled reservations from Stage I. At time $u = 0.3$ (before confirmation), a walk-in customer arrives. At this point:
\begin{itemize}
    \item \textit{Type I check-ins so far:} $B^{(k)}_{1,0.3} = 110$ guests have checked in.
    \item \textit{Pending reservations:} $420 - 110 = 310$ still pending; expected to show: $0.9 \times 310 = 279$.
    \item \textit{Walk-ins accepted:} $W^{(k)}_{1,0.3} = 10$ walk-ins already accepted.
    \item \textit{Future walk-ins:} With $\alpha = 0.4$, expected future acceptance is $0.4 \times 50 \times (1-0.3) = 14$.
    \item \textit{Expected total:} $\hat{N}^{(k)}_{0.3} \approx 110 + 279 + 10 + 14 = 413$.
\end{itemize}
Since $\hat{N}^{(k)}_{0.3} = 413 > 400 = \tilde{C}^{(k)}$, the walk-in is \emph{rejected}. After confirmation at $v = 0.5$, the manager learns exactly how many pending reservations will arrive, reducing uncertainty and enabling more precise acceptance decisions.
\end{example}

\begin{example}[DASS-II after confirmation]
Continuing the previous example, suppose confirmation at $v=0.5$ reveals that of the 310 pending reservations, exactly 270 will show up (the rest cancelled). At time $u=0.6$ (after confirmation), another walk-in arrives:
\begin{itemize}
    \item \textit{Type I status:} 110 already checked in + 270 confirmed arrivals = 380 total Type I guests.
    \item \textit{Walk-ins accepted:} $W^{(k)}_{1,0.6} = 15$ (including some after confirmation).
    \item \textit{Expected total:} $\hat{N}^{(k)}_{0.6} = 380 + 15 = 395$.
\end{itemize}
Since $\hat{N}^{(k)}_{0.6} = 395 < 400 = \tilde{C}^{(k)}$, the walk-in is \emph{accepted}. Notice the key difference: before confirmation, $\hat{N}$ included the uncertain term $q^{(1,k)} \times (\text{pending})$; after confirmation, this term becomes exact. The formula in \eqref{equ::StageII standard simple} simplifies because no estimation is needed for Type I arrivals---only the remaining walk-in budget matters.
\end{example}


To demonstrate the power of decoupling and dynamic threshold policy, we conduct sensitivity analysis of confirmation timing and show that Stage II regret decays exponentially with the interval length between confirmation timing and day end. Consequently, our algorithm achieves constant regret for Stage II decision \eqref{equ::Stage II regret} with only logarithmic time left for confirmation each day. The following theorem quantifies this exponential decay in terms of walk-in rates before and after the confirmation call.

\begin{theorem}[Stage II regret bound]
\label{prop:HO_gap}
   For any $k\in[T]$, given the number of non-cancelled bookings $B^{(k)}$ from Stage I, by setting $\alpha<1/2$, the expected Stage II loss between $\hat{\pi}_2$ and $\hat{\pi}_2^{\star}$ on Day $k$ is bounded by
    \begin{align*}
        \E_{\Fii^{(k)}}\left[\Lossii^{(k)}(\hat{\pi}_2,\{B^{(k')}(\hat\pi_1))\}_{k'=1}^{k})-\Lossii^{(k)}(\hat{\pi}^{\star}_2,\{B^{(k')}(\hat\pi_1))\}_{k'=1}^{k})\right] = O\left(\lambda^{(2,k)}_{[0,v]}\cdot \exp\left(-\Omega(\lambda^{(2,k)}_{[v,1]})\right)\right),
    \end{align*}
    where $\lambda^{(2,k)}_{[0,v]}=\int_{0}^{v}\lambda^{(2,k)}(s) ds$ and $\lambda^{(2,k)}_{[v,1]}=\int_{v}^{1}\lambda^{(2,k)}(s) ds$ are the walk-in rates before and after confirmation time $v$, respectively. The regret decays exponentially with the post-confirmation walk-in rate $\lambda^{(2,k)}_{[v,1]}$. See Appendix \ref{appendix::Stage II} for the exact bound and a tighter bound when $ \gamma^{(k)}(s)\propto \lambda^{(2,k)}(s)$.
\end{theorem}

\noindent\textit{Proof intuition.} The key difficulty is that DASS-II must estimate the optimal threshold $\tilde{C}^{(k)\star}$ without knowing how many Type I customers will arrive before confirmation time $v$. The proof bounds estimation errors using martingale concentration inequalities applied to the Type I arrival process—since arrivals follow a Poisson process with known distribution $\gamma^{(k)}$, we can establish high-probability bounds on the deviation between predicted and actual Type I arrivals. The exponential decay arises because after confirmation, the policy observes $\lambda^{(2,k)}_{[v,1]}$ walk-in customers arriving in $[v,1]$—a Poisson process generating $\Theta(\lambda^{(2,k)}_{[v,1]})$ samples. These samples drive the estimation error down exponentially via Chernoff-Hoeffding bounds, yielding the $\exp(-\Omega(\lambda^{(2,k)}_{[v,1]}))$ factor.

\begin{example}[Exponential decay with confirmation timing]
Consider $\lambda^{(2,k)} = 50$ walk-ins per day (uniformly distributed). The regret bound $O(\lambda^{(2,k)}_{[0,v]} \cdot \exp(-\Omega(\lambda^{(2,k)}_{[v,1]})))$ behaves as follows for different confirmation times:
\begin{center}
\begin{tabular}{cccl}
$v$ & $\lambda^{(2,k)}_{[0,v]}$ & $\lambda^{(2,k)}_{[v,1]}$ & Regret bound (order of magnitude) \\
\hline
0.5 & 25 & 25 & $25 \cdot e^{-25} \approx 10^{-9}$ (negligible) \\
0.8 & 40 & 10 & $40 \cdot e^{-10} \approx 2 \times 10^{-3}$ (small) \\
0.95 & 47.5 & 2.5 & $47.5 \cdot e^{-2.5} \approx 4$ (noticeable)
\end{tabular}
\end{center}
The exponential decay dominates: even with late confirmation ($v=0.8$), regret remains negligible. Only with very late confirmation ($v \geq 0.95$) does regret become substantial. This phase transition---from negligible to noticeable regret as $v$ crosses a threshold---is consistent with the empirical results in Section~\ref{sec:numerical_experiments}.
\end{example}

The proof is detailed in Appendix \ref{appendix::Stage II}. Since with high probability in Stage I we ensure $q^{(1,k)}B^{(k)}\le \delta C$, and the walk-in rate satisfies $\lambda^{(2,k)}\ge 5\iota+ 4\sqrt{3\iota{\delta C}}$, as long as $\lambda_{[v,1]}^{(2,k)} \gtrsim \iota+ \sqrt{\iota{\delta C}}$, the Stage II regret is bounded by $\Tilde{O}(1)$.

For the homogeneous case where $\lambda^{(2,k)}(s)=\lambda^{(2,k)}$ and $\gamma^{(k)}(s)=1$, the regret decays exponentially with $1-v$, the interval length between confirmation time $v$ and day end. To achieve $O(1)$ total Stage II regret, it suffices that
\begin{align}\label{equ::require Stage II call timing}
    (1-v)\lambda^{(2,k)} \gtrsim \iota + \sqrt{\delta C \cdot \iota}.
\end{align}
Hence we need only schedule confirmation calls $O(\sqrt{\delta C\log(CT)/\lambda^{(2,k)}})$ before the end of each day. See Appendix \ref{appendix::Stage II} for the exact bound.

By combining Theorem \ref{prop::StageI regret} and \ref{prop:HO_gap}, we achieve a constant regret guarantee of our \textbf{DASS} policy.
\begin{corollary}[Total regret]\label{coro::total}
    If the walk-in customer flow and the Type I customers' arrival flow satisfy $\lambda^{(2,k)}(s)=\lambda^{(2,k)}$ and $\gamma^{(k)}(s)=1$ for any $s\in[0,1]$ and $\iota\ge\log(CT)$, as long as the confirmation time $v$ satisfies \eqref{equ::require Stage II call timing} for each day $k\in[T]$, our regret achieves a constant expected total regret of
    \begin{align*}
        \regret^T(\hat\pi_1,\hat\pi_2) =  O(1).
    \end{align*}
\end{corollary}

\section{Extension to Heterogeneous Customers}\label{sec:nested}

We extend the model to $m$ customer classes with nested capacity protection. Classes are indexed by $j\in[m]$ with revenues in \textbf{increasing order} $r_1^{(k)}<r_2^{(k)}<\cdots<r_m^{(k)}$ and overbooking losses $\ell_1^{(k)}<\ell_2^{(k)}<\cdots<\ell_m^{(k)}$ on Day $k$. Nested booking limits satisfy $C^{(k)}_1\le C^{(k)}_2\le\cdots\le C^{(k)}_m=C$.

Class 1 (lowest revenue) is restricted to $C_1^{(k)}$ rooms. Class $j$ can access any room up to limit $C_j^{(k)}$, sharing inventory with lower classes. Class $m$ (highest revenue) has access to the entire capacity $C$.

\subsection{Key Challenges}
Extension to nested multi-class control introduces three structural challenges beyond simply adding indices.

First, a decision for one class has a \emph{domino effect} across protection levels. For a given day $k$ and level $j$, the cumulative accepted reservations from the top-$j$ classes are
\[
B_t^{(\le j,k)}:=\sum_{i=1}^j B_t^{(i,k)},
\]
and the safety stock threshold we use is
\begin{align}
    \hat{B}_t^{(\le j,k)}
    =\mu_t^{(\le j,k)}+\frac{\iota\bigl(1-\bar p_t^{(\le j,k)}\bigr)}{3}
      +\sqrt{\left(\frac{\iota\bigl(1-\bar p_t^{(\le j,k)}\bigr)}{3}\right)^2
             +2\iota\,\sigma_t^{2(\le j,k)}}, \tag{16}
\end{align}
where
\[
\mu_t^{(\le j,k)}=\sum_{i=1}^j p^{(i,k)}(t)B_t^{(i,k)},\quad
\sigma_t^{2(\le j,k)}=\sum_{i=1}^j B_t^{(i,k)}p^{(i,k)}(t)\bigl(1-p^{(i,k)}(t)\bigr),\quad
\bar p_t^{(\le j,k)}=\frac{\sum_{i=1}^j B_t^{(i,k)}p^{(i,k)}(t)}{\sum_{i=1}^j B_t^{(i,k)}}.
\]
Accepting a class-$j$ request increases $B_t^{(\le\ell,k)}$ for all $\ell\ge j$, and the admissibility condition
\[
\hat{B}_t^{(\le \ell,k)}<\hat C_\ell^{(k)}\quad \text{for all }\ell\in\{j,\dots,m\}
\]
must hold simultaneously. Thus, a single acceptance decision for class $j$ is constrained by $m-j+1$ simultaneous safety conditions. A violation in \emph{any} higher-level aggregate bucket ($\ell \ge j$) necessitates a rejection, effectively coupling the decision space of all classes.

Second, the safety stock is driven by the \emph{aggregate variance} of heterogeneous Bernoulli streams, not by a sum of per-class safety stocks. In the single-class case, the buffer width scales like
\[
\text{Safety} \;\approx\; \sqrt{\text{Var}\!\bigl(B^{(k)}\bigr)}.
\]
In the nested setting we have
\[
\text{Var}\!\bigl(B^{(\le j,k)}\bigr)=\sum_{i=1}^j B_t^{(i,k)}p^{(i,k)}(t)\bigl(1-p^{(i,k)}(t)\bigr),
\]
so the correct buffer scales as
\[
\text{Safety}_j \;\approx\; \sqrt{\sum_{i=1}^j \text{Var}\!\bigl(B^{(i,k)}\bigr)}\neq \sum_{i=1}^j \sqrt{\text{Var}\!\bigl(B^{(i,k)}\bigr)}.
\]
Running $m$ independent single-class buffers would therefore be overly conservative. The nested estimator \eqref{eq:nested_B} exploits this pooling effect while still guaranteeing all $m$ protections with one coupled construction.

Third, the \emph{busy season} requirement must hold at each bottleneck level of the hierarchy. In the single-class model, it suffices to assume
\[
\lambda^{(2,k)} \gtrsim \iota+\sqrt{C\delta\iota}
\]
to ensure walk-ins can fill the Stage I safety buffer. In the nested model, the overall system may be busy while high-value segments are not. Assumption \ref{assump:nested_busy} requires, for every $j$,
\[
\lambda^{(2,\le j,k)} \ge \frac{\delta C^{(k)}_j}{q_{\min}}+c_1\iota+c_2\sqrt{C^{(k)}_j\iota},
\]
ensuring the safety buffer for each level can be filled with high probability. The regret analysis is driven by the weakest (most capacity-constrained) level.

These three features — domino effects across nested thresholds, non-additive variance structure, and multi-level busy-season requirements — distinguish the multi-class setting from the single-class case and require new analytical techniques.

\subsection{Model and Arrivals}
Stage I reservations of class $j$ arrive as a non-homogeneous Poisson process with rate $\lambda^{(1,j,k)}(t)$, $t\in[k-k_0,k]$, survive Stage I with probability $p^{(j,k)}(t)$, and show up in Stage II with probability $q^{(1,j,k)}$. Class-$j$ walk-ins arrive at rate $\lambda^{(2,j,k)}(u)$, $u\in[0,1]$. Durations follow the same geometric ($q$) or constant ($d$) assumptions and are known on arrival. Aggregate quantities over top $j$ classes use the prefix $(\le j)$, e.g., $\lambda^{(2,\le j,k)}(u)=\sum_{i=1}^j\lambda^{(2,i,k)}(u)$.

\subsection{Nested DASS-I}
Let $B_t^{(j,k)}$ be the accepted, not-yet-cancelled class-$j$ bookings at time $t$. We define the aggregate bookings for the first $j$ classes as $B_t^{(\le j,k)}=\sum_{i=1}^j B_t^{(i,k)}$.

We calculate the aggregate expected confirmation and variance as:
\[
\mu_t^{(\le j,k)}=\sum_{i=1}^j p^{(i,k)}(t)B_t^{(i,k)},\quad \sigma_t^{2(\le j,k)}=\sum_{i=1}^j B_t^{(i,k)}p^{(i,k)}(t)\bigl(1-p^{(i,k)}(t)\bigr).
\]
Let the weighted mean probability be $\bar p_t^{(\le j,k)}=\mu_t^{(\le j,k)}/B_t^{(\le j,k)}$. The cumulative safety threshold is defined as:
\begin{align}
    \hat{B}_t^{(\le j,k)}=\mu_t^{(\le j,k)}+\frac{\iota(1-\bar p_t^{(\le j,k)})}{3}+\sqrt{\left(\frac{\iota(1-\bar p_t^{(\le j,k)})}{3}\right)^2+2\iota\sigma_t^{2(\le j,k)}}. \label{eq:nested_B}
\end{align}

Underestimated capacities $\hat C_j^{(k)}$ are computed analogously to \eqref{eq:hat_C} with capacity $C$ replaced by $C^{(k)}_j$. A class-$j$ reservation is accepted \textbf{if and only if} the aggregate condition holds for all levels $\ell \ge j$:
\[
\hat{B}_t^{(\le \ell,k)}<\hat C_\ell^{(k)}\quad \forall \ell\in\{j,\dots,m\}.
\]
This ensures that accepting a low-value customer does not violate the protection levels set for higher-value aggregates, yielding \textbf{N-DASS-I}.

\subsection{Nested DASS-II}
Fix Day $k$. Let $B^{(j,k)}$ be non-cancelled bookings of class $j$ at Stage II start. For $u<v$, define cumulative check-ins $B^{(\le j,k)}_{1,u}$, accepted walk-ins $W^{(\le j,k)}_{1,u}$, and the weighted forecast
\begin{align}
    \hat N_u^{(\le j,k)}=B^{(\le j,k)}_{1,u}+ \sum_{i=1}^j q^{(1,i,k)}\!\left(B^{(i,k)}-B^{(i,k)}_{1,u}\right) + W^{(\le j,k)}_{1,u}+ \alpha_j \int_u^1 \lambda^{(2,\le j,k)}(s)ds. \label{eq:nested_stageII}
\end{align}
For $u\ge v$, the confirmation call reveals remaining reservations, so $\hat N_u^{(\le j,k)}=B^{(\le j,k)}_{1,v}+B^{(\le j,k)}_{3,v}+W^{(\le j,k)}_{1,u}$. A class-$j$ walk-in at time $u$ is accepted if $\hat N_u^{(\le \ell,k)}<\tilde C_\ell^{(k)}$ for all $\ell\in\{j,\dots,m\}$, where $\tilde C_\ell^{(k)}$ generalizes \eqref{equ::allocated capacity} with $C$ replaced by $C^{(k)}_\ell$ (and $C/d$ for constant durations). This defines \textbf{N-DASS-II}.

\subsection{Busy Season and Regret}
We generalize Assumption \ref{assump:busy_season} as follows.

\begin{assumption}[Multi-class busy season]\label{assump:nested_busy}
There exist absolute constants $c_1,c_2>0$ such that for every day $k$ and nesting level $j\in[m]$,
\begin{align}
    \lambda^{(2,\le j,k)} \;\ge\; \frac{\delta C^{(k)}_j}{q_{\min}}+c_1\iota+c_2\sqrt{C^{(k)}_j\iota}, \label{eq:nested_busy}
\end{align}
where $q_{\min}=\min_{i\le j} q^{(1,i,k)}$, $\delta=1-q$ (geometric) or $1/d$ (constant), and $\iota\ge \log(CTm)$. The reservation rates $\lambda^{(1,j,k)}$ satisfy \eqref{eq:require_booking_appendix} with $C$ replaced by $C^{(k)}_j$.
\end{assumption}

\begin{theorem}[Nested regret]\label{thm:nested_regret}
Under Assumption \ref{assump:nested_busy} and the duration assumptions, (\textbf{N-DASS-I}, \textbf{N-DASS-II}) achieves
\[
    \E[\regret^T(\hat\pi^{\mathrm{N}}_1,\hat\pi^{\mathrm{N}}_2)] \le O(mCT\,e^{-\iota})+C_0,
\]
where $C_0=0$ for geometric durations and $C_0=\sum_{k=1}^{d}\frac{C(d-k)}{d}\cdot \max_j r_j^{(k)}$ for constant durations. Choosing $\iota=\Theta(\log(CTm))$ yields $O(1)$ expected regret. Proof details are in Appendix \ref{appendix:nested}.
\end{theorem}

\subsubsection{Proof Sketch}\label{sec:nested_proof_strategy}

The three challenges are addressed as follows:

\begin{enumerate}
    \item \textbf{Domino Effect $\rightarrow$ Union Bounds:} The hierarchical booking limits $\{C_j^{(k)}\}$ impose $(m-j+1)$ simultaneous constraints for each class $j$ request. We apply union bounds over all levels in Stage I, decoupling the $m$ constraints into $m$ independent concentration inequalities, each holding with probability $1-O(e^{-\iota})$.

    \item \textbf{Aggregate Variance $\rightarrow$ Pooled Estimators:} Since $\sqrt{\sum_{i=1}^j \mathrm{Var}(X_i)} \neq \sum_{i=1}^j \sqrt{\mathrm{Var}(X_i)}$, we cannot simply sum safety stocks. We construct an aggregate estimator $\hat{B}_t^{(\le j,k)}$ that pools demand across classes $1,\dots,j$ and applies Bernstein's inequality to the total variance. The safety threshold in \eqref{eq:nested_B} reflects this non-linear aggregation.

    \item \textbf{Hierarchical Busy Season $\rightarrow$ Bottleneck Analysis:} Each nesting level $j$ has its own busy-season condition \eqref{eq:nested_busy}. Our proof in Appendix \ref{appendix:nested} analyzes idling and overbooking separately at each level, then applies union bounds to show all levels remain feasible with high probability. For practical verification, if the class mix is uniformly bounded away from zero ($\lambda^{(2,\le j,k)}\ge \theta_j\lambda^{(2,k)}$ for all $j$ with constants $\theta_j>0$), it suffices to check \eqref{eq:nested_busy} only for the highest class $j=m$.
\end{enumerate}

\begin{remark}
The proof structure mirrors the single-class case in three parts: (i) Stage I regret via union bounds over $m$ nesting levels; (ii) decoupling into single-day offline policies using the capacity caps $\tilde C^{(k)}_j$; (iii) Stage II concentration for the weighted forecasts \eqref{eq:nested_stageII}. Sensitivity to confirmation timing follows Theorem \ref{prop:HO_gap} with class-dependent weights $\{\alpha_j\}$ (see Appendix \ref{appendix:nested}).
\end{remark}

\section{Numerical Experiments}\label{sec:numerical_experiments}
\subsection{Interplay of Booking Level, Walk-ins and Confirmation Call}
We investigate how booking decisions, walk-in rates, and confirmation timing interact through numerical experiments. To isolate the effects of Stage I decisions, we use fixed duration $d=1$ so decisions are independent across days. We focus on regret for a single day $k$ with idle capacity $C=200$ and cancellation probability $q^{(1,k)}=0.5$.

\subsubsection{Sensitivity to booking level and walk-in rate}
We vary non-cancelled reservations $B^{(k)}\in[300,500]$ and walk-in rate $\lambda^{(2,k)}\in[0,50]$ to understand their joint effect on regret. For each pair $(B^{(k)}, \lambda^{(2,k)})$, we run 1000 Stage II simulations with confirmation timing $v=0$ (offline optimal). Figure \ref{fig::walk in with Bk} shows regret is minimized when $\lambda^{(2,k)}\approx C/q^{(1,k)}-B^{(k)}\approx 2\sqrt{C}$, meaning expected non-cancelled Type I customers slightly underutilize capacity and walk-ins fill the gap. This validates our safety stock construction and Assumption \ref{assump:busy_season}.

\subsubsection{Sensitivity to confirmation call timing}
We simulate Stage II with the optimal $B^{(k)}$ from Figure \ref{fig::walk in with Bk} under different confirmation timings. Arrival times for both customer types follow ${\rm Beta}(6,6)$. We fix $B^{(k)}=360$ and run 5 independent experiments. Figure \ref{fig:regret v} shows regret decreases as confirmation calls occur earlier, with a sharp phase transition around $v=0.6$. This validates Theorem \ref{prop:HO_gap}: regret decays exponentially with the post-confirmation interval $(1-v)$. The sharpness of this transition arises from the exponential decay factor $\exp(-\Omega(\lambda^{(2,k)}_{[v,1]}))$ in Theorem \ref{prop:HO_gap}—once the post-confirmation walk-in flow $\lambda^{(2,k)}_{[v,1]}$ exceeds the threshold $\Theta(\iota + \sqrt{\iota \delta C})$, the estimation error collapses rapidly due to concentration inequalities. The observed $v=0.6$ threshold corresponds to $(1-v)\lambda^{(2,k)} \approx 0.4 \times 50 = 20 \gtrsim \iota + \sqrt{\delta C \cdot \iota}$ (with $\iota=2$, $\delta=0.5$, $C=200$), matching our theoretical requirement \eqref{equ::require Stage II call timing}. Practically, this implies that confirmation calls need not occur at the start of the day—even mid-day confirmation ($v \approx 0.5$) suffices for near-optimal performance, reducing operational burden.

\begin{figure}
    \centering
    \includegraphics[width=0.9\linewidth]{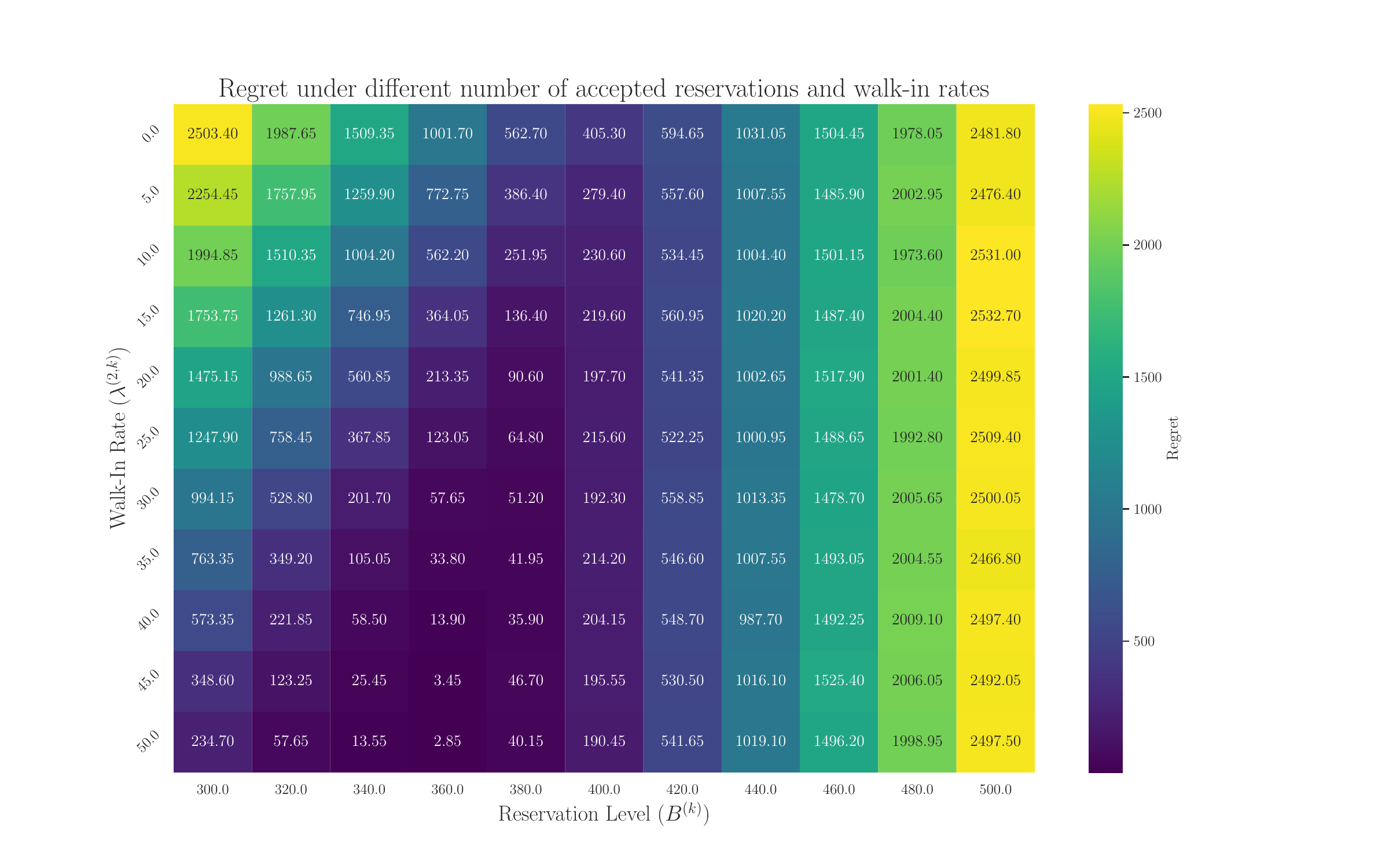}
    \caption{Regret under different number of accepted reservations $B^{(k)}$ and walk in rates $\lambda^{(2,k)}$ in a single day. We run the simulation with the room capacity $C=200$ and the cancel probability $q^{(1,k)}=0.5$. To emphasize the impact of Stage I decision, i.e., the control of $B^{(k)}$, we set the confirmation call timing to be $v=0$, i.e., conducting the offline optimal policy in Stage II.  We compute the expected Stage II regret with $B^{(k)}\in[300,500]$ and $\lambda^{(1,k)}\in[0,50]$ through $1000$ simulations on Stage II for each pair of $(B^{(k)},\lambda^{(2,k)})$.}
    \label{fig::walk in with Bk}
\end{figure}

\begin{figure}
    \centering
    \includegraphics[width=0.7\linewidth]{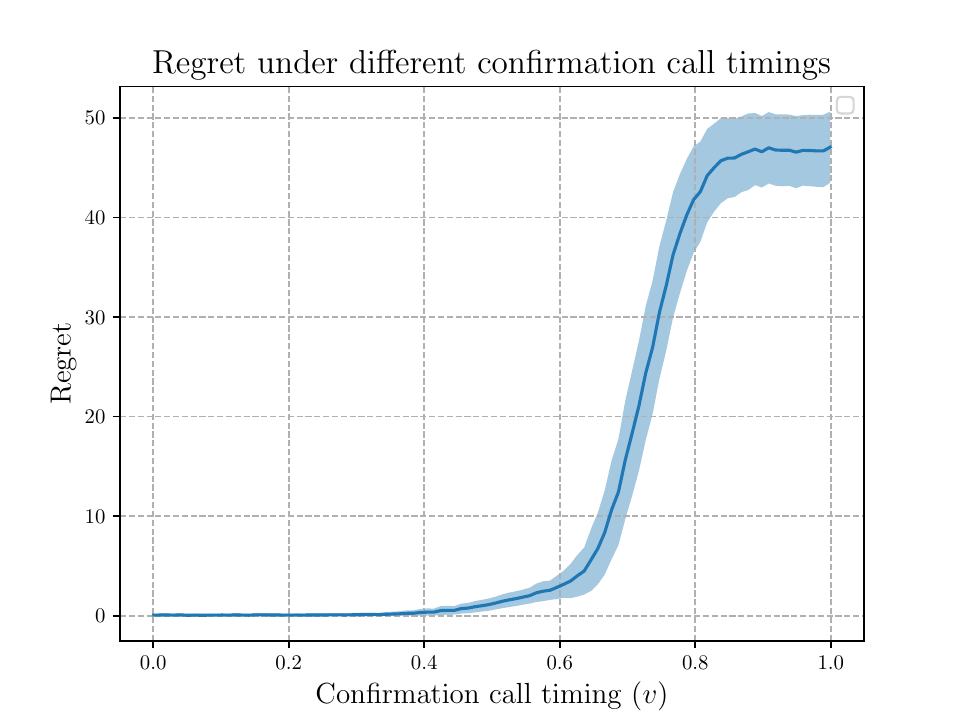}
    \caption{Regret under different confirmation call timings. We set the distribution of the arrival time for each Type I or Type II customer to follow a beta distribution of ${\rm Beta}(6,6)$.  We fix the number of non-cancelled reservations $B^{(k)}=360$ and simulate $5$ independent decision processes to compute the expected regret and the standard error.}
    \label{fig:regret v}
\end{figure}

Having validated key sensitivities on single-day scenarios, we now evaluate DASS over longer horizons against heuristic policies.

\subsection{Synthetic Experiments}\label{sec:numerical_synthetic}

We evaluate \textbf{DASS} (Algorithms \ref{alg:DASS_stage_I} and \ref{alg:DASS_stage_II}) over $T=1000$ days with geometric occupancy durations ($q=0.3$), capacity $C=100$, non-cancellation rate $q^{(1,k)}=0.4$, reservation rate $\lambda^{(1,k)}=300$, and walk-in rate $\lambda^{(2,k)}=30$. Stage II arrival times follow ${\rm Beta}(6,6)$. We set hyperparameters $\iota=2$ and $\alpha=0.4$.

We benchmark against static threshold heuristics:
\[
\hat{B}_t^{(k)} = \frac{(1+\beta)(1-q) C}{q^{(1,k)}}~~~\text{(Stage I)},\quad N_u^{(k)}=q^{(1,k)}B^{(k)}~~~\text{(Stage II)},
\]
where $\beta \in \{-0.2,-0.1,0,0.1,0.2\}$ controls overbooking tolerance. We test confirmation timings $v\in\{0,0.5,0.7,1.0\}$ with 5 independent runs per configuration.

Figure \ref{fig::syn regret} shows \textbf{DASS} outperforms all heuristics across all confirmation timings. With $v=0$, \textbf{DASS} achieves near-zero regret, validating Theorem \ref{prop::StageI regret}. As $v$ increases, all policies degrade, but \textbf{DASS} maintains superior performance due to adaptive safety stocks, with the performance gap widening at later confirmation times ($v=0.7, 1.0$).
\begin{figure}
    \centering
    \subfigure[$v=0$]{\includegraphics[width=0.48\linewidth]{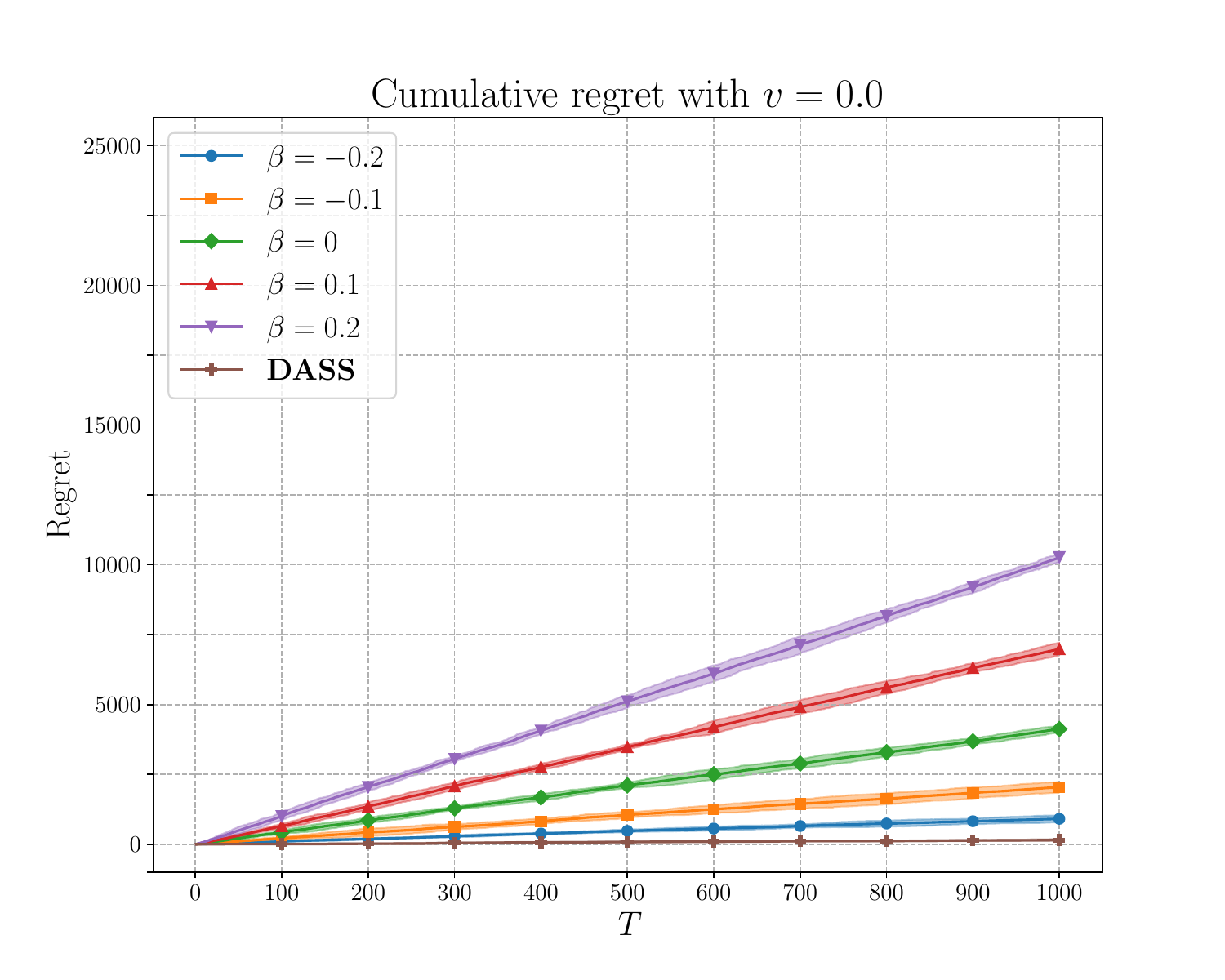}}
    \hfill
    \subfigure[$v=0.5$]{\includegraphics[width=0.48\linewidth]{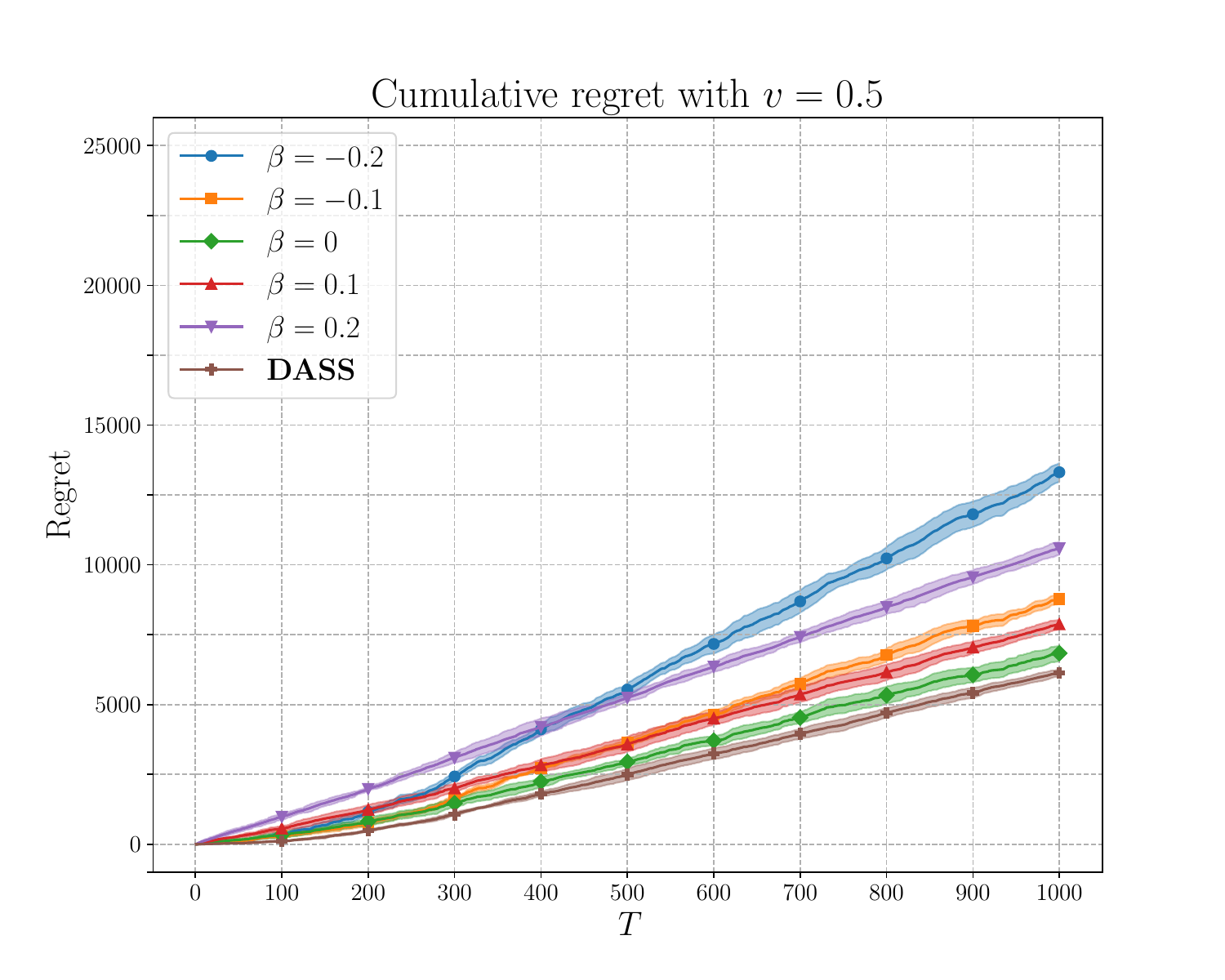}}
    \\[4ex]
    \subfigure[$v=0.7$]{\includegraphics[width=0.48\linewidth]{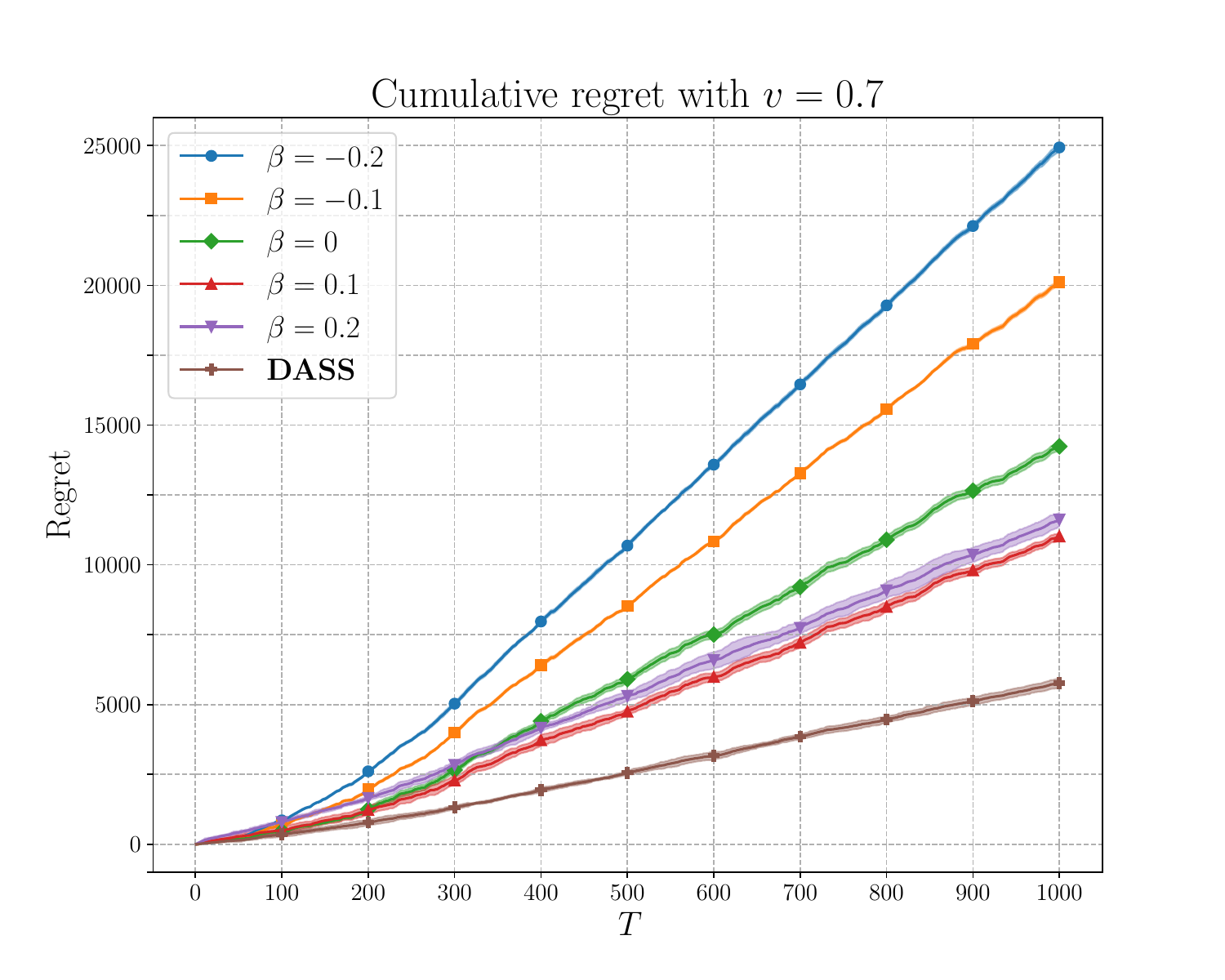}}
    \hfill
    \subfigure[$v=1.0$]{\includegraphics[width=0.48\linewidth]{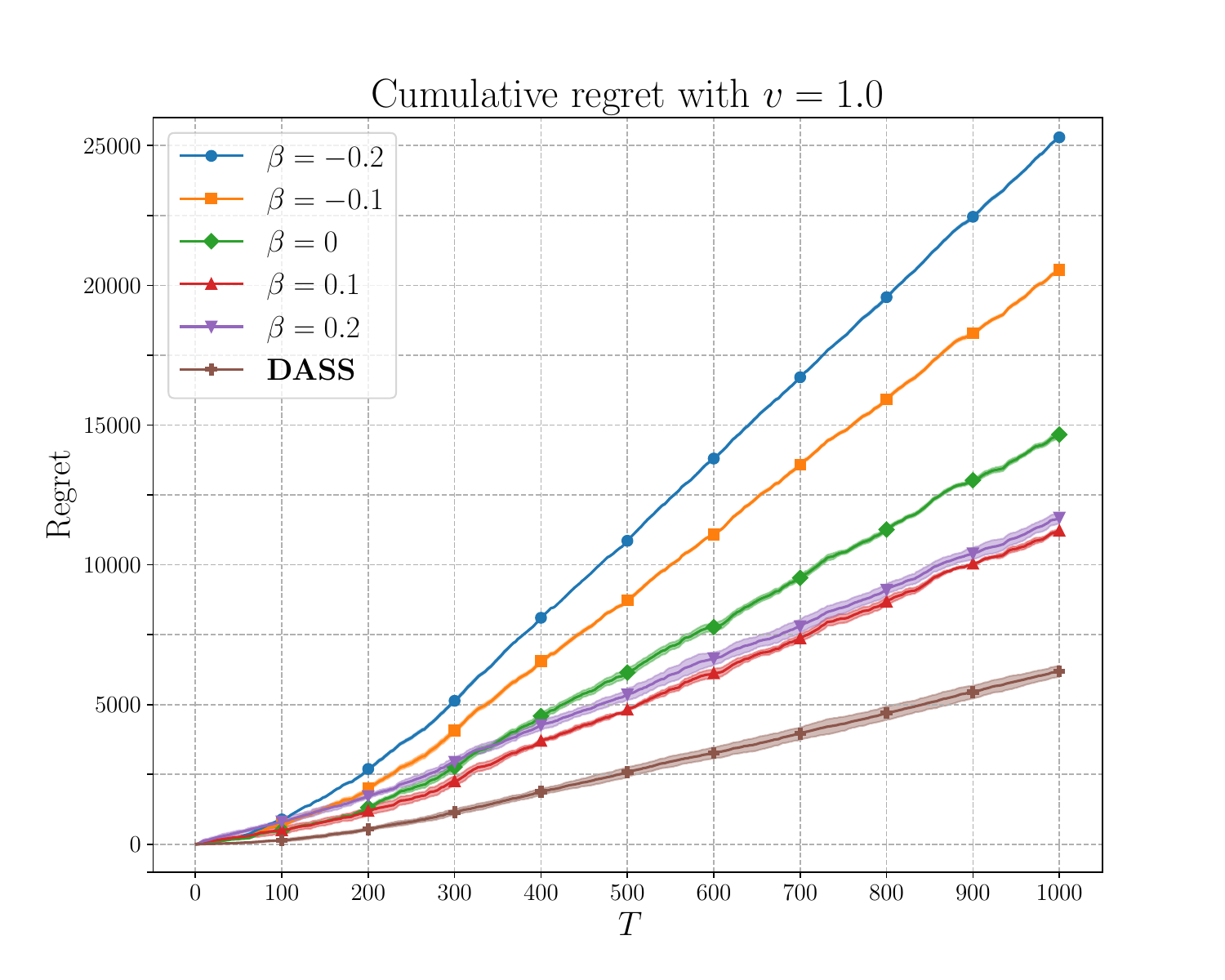}}
    \caption{Cumulative regret with different confirmation call timings. We generate the customer flow with $C=100$, $q=0.3$, $q^{(1,k)}=0.4$, $\lambda^{(1,k)}=300$ and $\lambda^{(2,k)}=30$, and set the distribution of the arrival time of customers in Stage II to follow ${\rm Beta}(6,6)$. For the hyperparameters of our algorithm, we choose $\iota=2$ and $\alpha=0.4$. We simulate $5$ independent customer flows within $T=1000$ days to compute the expected regret and the standard error of each algorithm.}\label{fig::syn regret}
\end{figure}

\subsection{Application to Real Data from an Algarve Resort Hotel}
Finally, we validate DASS on real hotel data to confirm its practical applicability. We apply \textbf{DASS} to a dataset from a resort hotel in Algarve, Portugal.

\subsubsection{Data Description}
The Algarve resort hotel dataset \citep{antonio2019hotel} records daily customer data from July 1, 2015, to July 31, 2017, for a 70-room hotel. The data includes accepted reservations (lead time, cancellation date, check-in date, occupancy duration) and walk-ins with durations.

Figures \ref{fig:reservation_walk-ins} and \ref{fig:occupancy} show data from October 1, 2015, to July 31, 2016. Arrivals exhibit strong seasonality, with walk-ins typically one-third to one-half of reservations. Overbooking occurs consistently, validating the busy season assumption.
\begin{figure}[ht]
    \centering
    \includegraphics[width=1\linewidth]{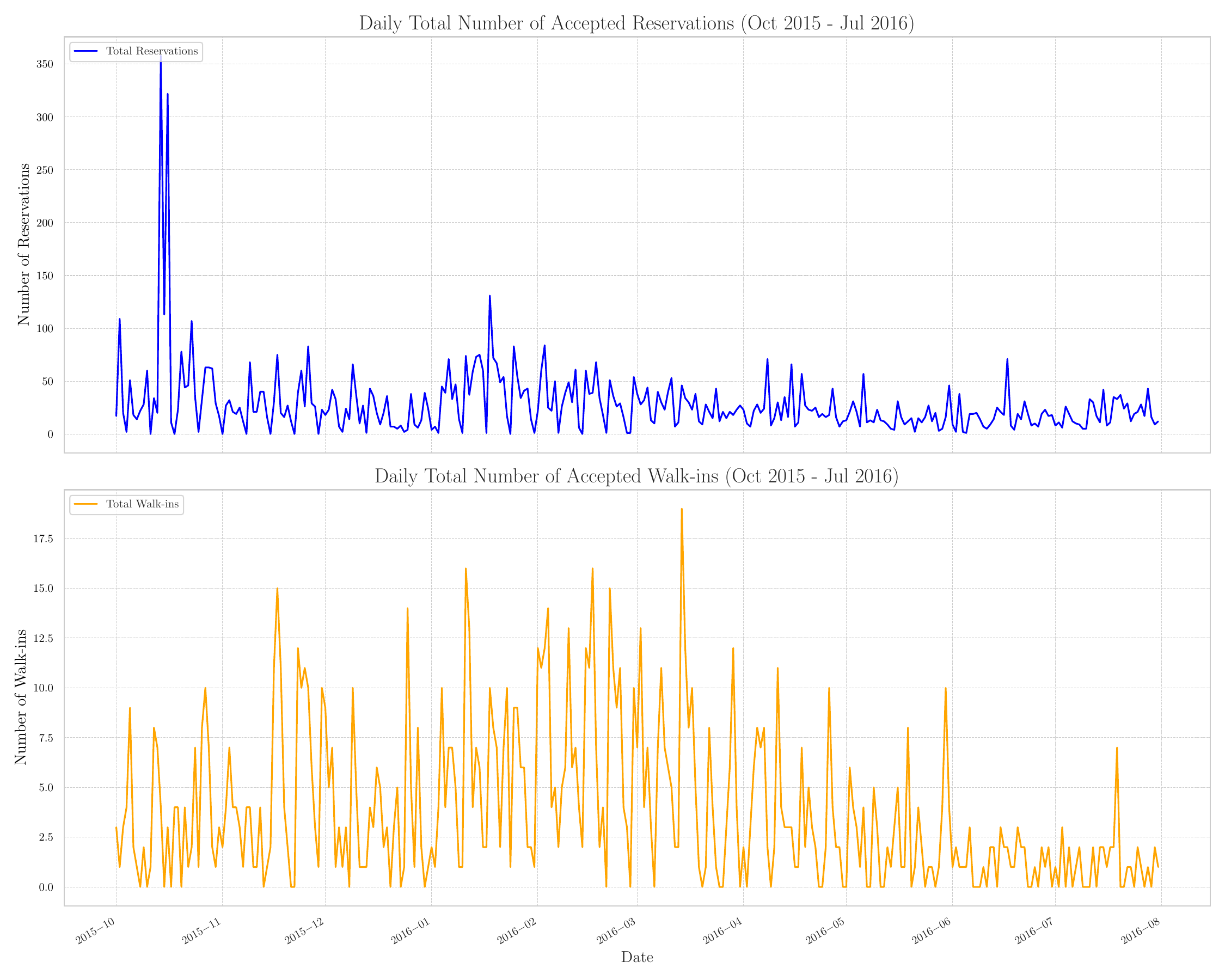}
    \caption{Illustration of the Algarve resort hotel data with daily arrivals.}
    \label{fig:reservation_walk-ins}
\end{figure}
\begin{figure}[ht]
    \centering
    \includegraphics[width=1\linewidth]{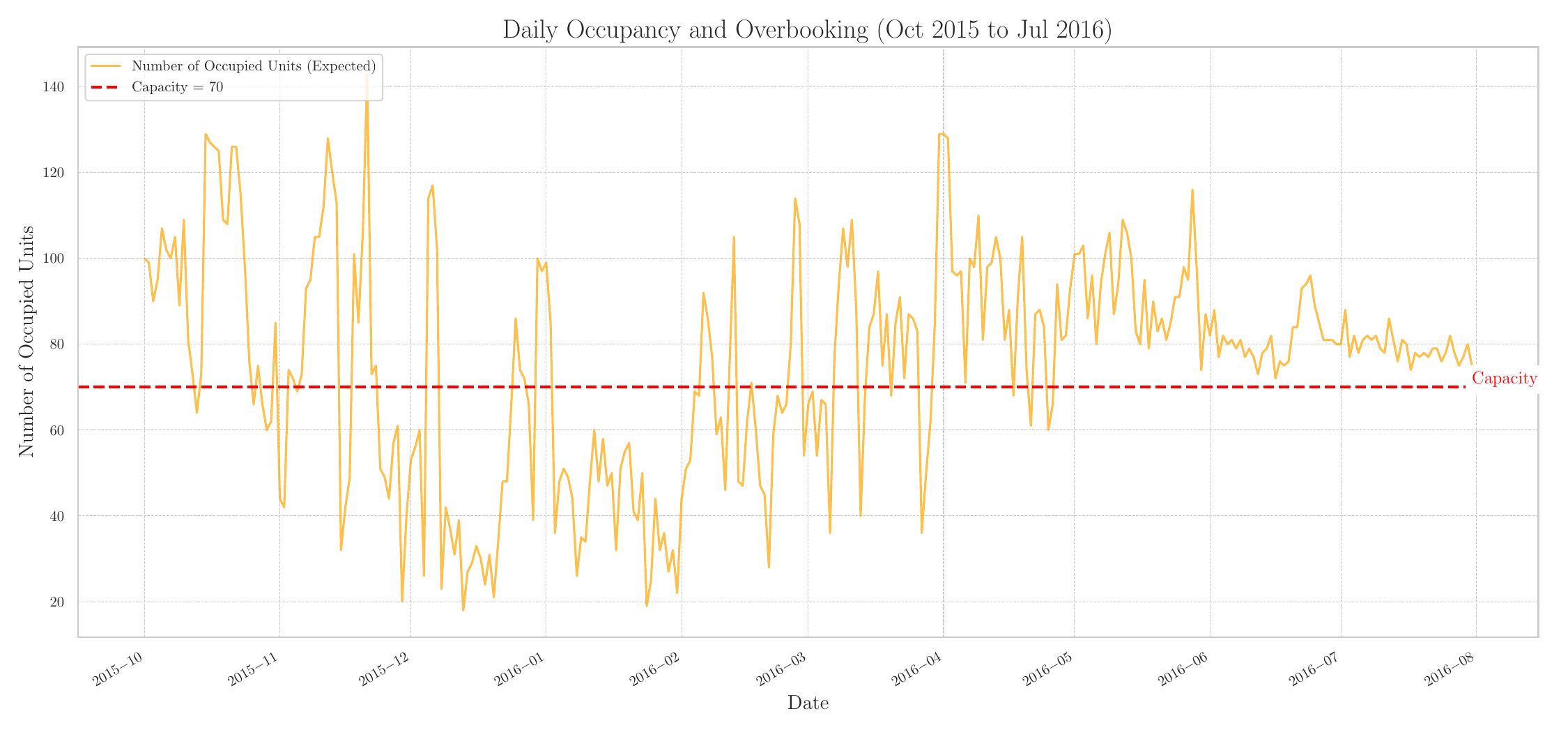}
    \caption{Illustration of the Algarve resort hotel data with daily confirmed rooms to be occupied vs total capacity.}
    \label{fig:occupancy}
\end{figure}

\subsubsection{Data Augmentation}
The dataset lacks rejected reservation requests, rejected walk-ins, and intra-day check-in flows—information needed to simulate our model. We fit parametric distributions to observed data and generate a synthetic three-year dataset with complete information.

Figure \ref{fig:fit} shows fitted distributions and observed histograms (blue). We use Gamma for lead time, Weibull for cancellation interval (time from reservation to cancellation), Geometric for occupancy duration, and Poisson Mixture for walk-ins. The fitted distributions (red) match the data well.

\begin{figure}[ht]
    \centering
    \includegraphics[width=1\linewidth]{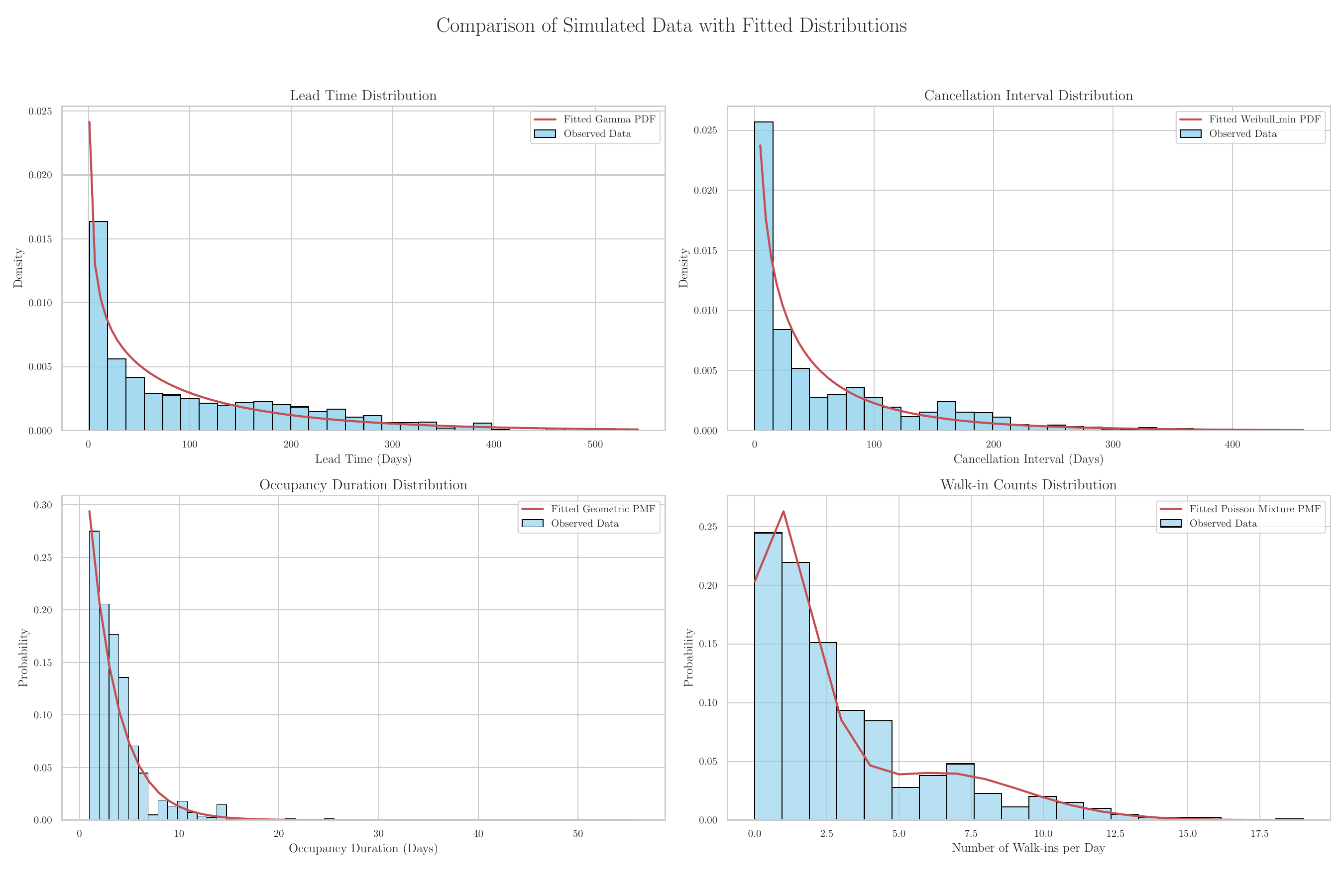}
    \caption{Fitted distributions for lead time, cancellation time, occupancy duration, and walk-in rates. Blue: observed histograms. Red: fitted parametric distributions.}
    \label{fig:fit}
\end{figure}

\subsubsection{Results}
We run \textbf{DASS} and baseline heuristics from Section \ref{sec:numerical_synthetic} on the synthetic dataset, choosing hyperparameters $(\alpha,\iota)$ to minimize regret. We simulate 5 independent runs over $T=1000$ days.

Figure \ref{fig::real} shows \textbf{DASS} outperforms heuristics across all confirmation timings. At $v=0.7$, \textbf{DASS} achieves 15\% higher revenue than the best heuristic.

\begin{figure}
    \centering
    \subfigure[Cumulative revenue with  $v=0.7$]{\includegraphics[width=0.48\linewidth]{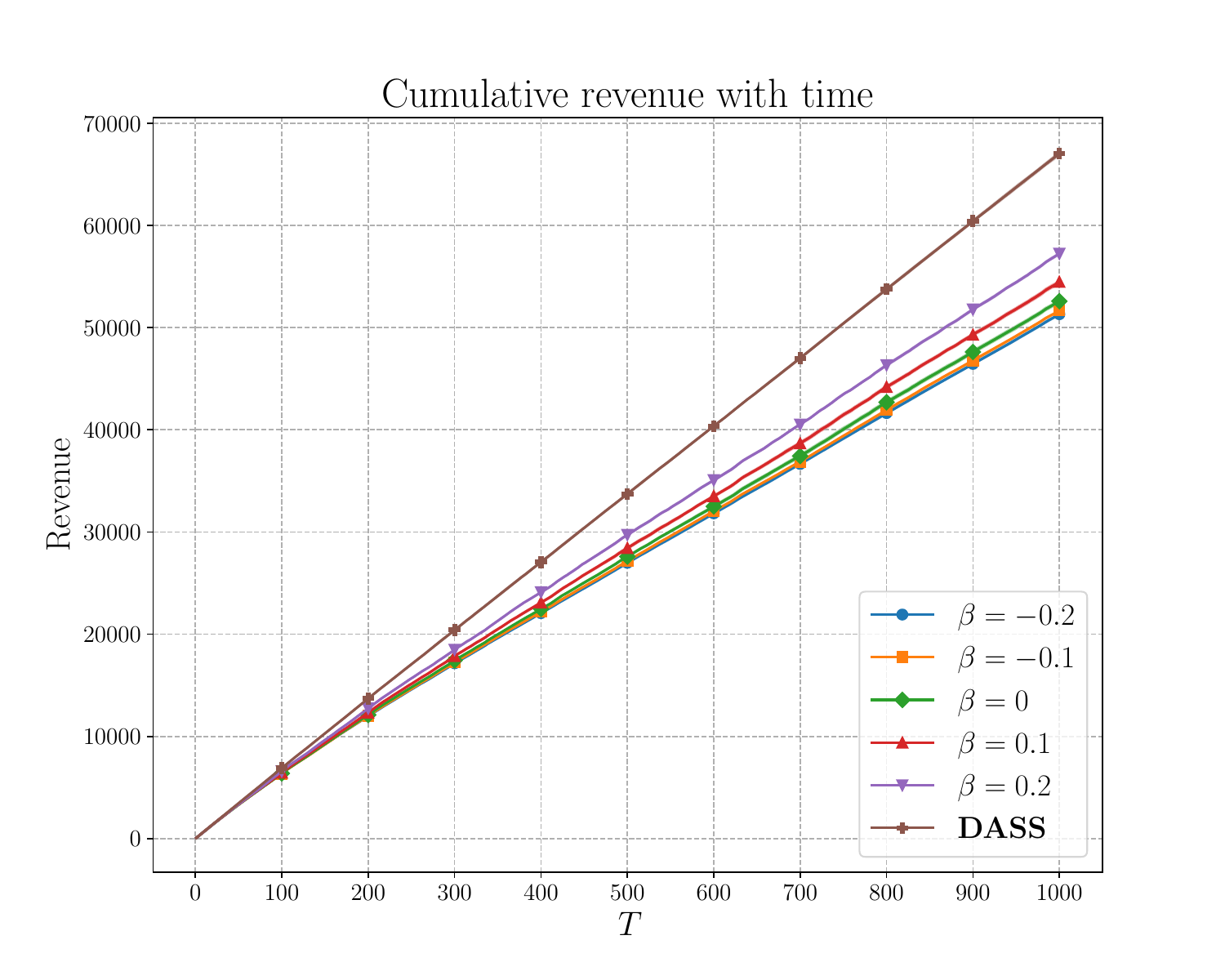}}
    \hfill
    \subfigure[Total revenue with different $v$.]{\includegraphics[width=0.48\linewidth]{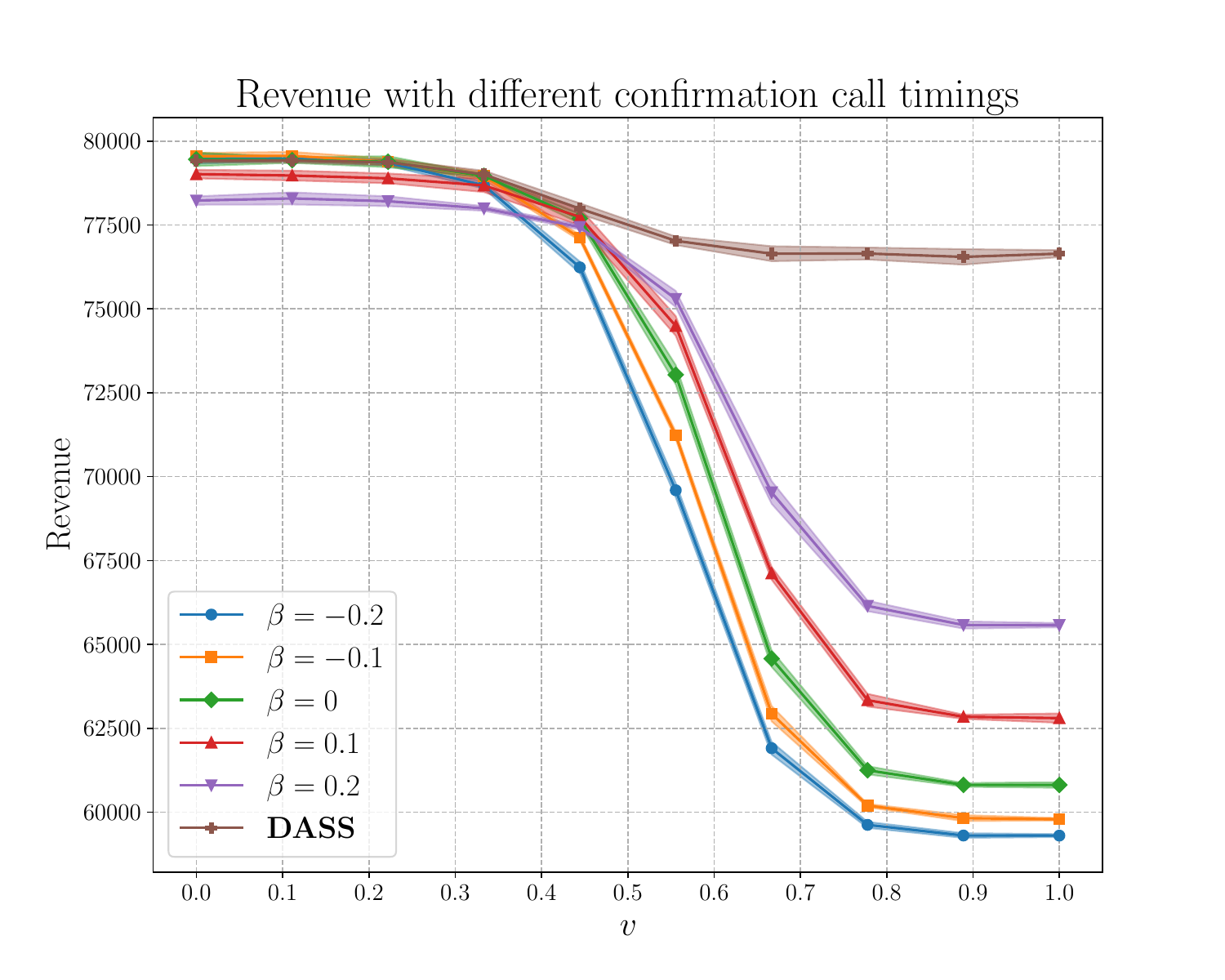}}
    \caption{Cumulative regret with different confirmation call timings. We generate the customer flow with the parameters estimated from the real data described above, and choose the parameters $(\alpha,\iota)$ in our algorithm that minimize the regret. We simulate $5$ independent customer flows within $T=1000$ days to compute the expected regret and the standard error of each algorithm.}\label{fig::real}
\end{figure}

\section{Conclusion}
We study online reusable resource allocation with advance reservations, overbooking, and confirmation calls. Our two-stage model incorporates a busy season assumption that mitigates uncertainties from cancellations and overlapping occupancy durations. Without this assumption, all online policies incur $\Omega(T)$ regret—this explains why policies achieving constant regret for non-reusable resources fail in the reusable setting.

Our Decoupled Adaptive Safety Stock (\textbf{DASS}) policy uses adaptive safety stocks to decouple inter-day correlations, mitigating overbooking risk while improving utilization. Under DASS, optimal Stage II control can be approximated by independent single-day policies, demonstrating the power of decoupling. Regret decays exponentially with the interval between confirmation and service end, achieving near-optimal constant regret even with late confirmations. Numerical experiments on synthetic and real hotel data validate DASS performance, achieving 15\% revenue improvement over the best heuristic on real data, and align with our theoretical analysis.

Extensions include instance-dependent bounds based on confirmation timing and walk-in rates, and applications to multi-resource settings with interdependent occupancy durations.


\clearpage

\section*{Acknowledgments}
We would like to thank Feng Zhu for helpful and insightful discussions.

\bibliographystyle{apalike}
\bibliography{ref}

\newpage
\setcounter{page}{1}
\begin{APPENDICES}
\small
\fontsize{9}{10}
\clearpage
\section{Proof of Main Results}\label{appendix:proof}

\subsection{Precise Conditions for Busy Season Assumption}\label{appendix::busy_season}
Assumption \ref{assump:busy_season} requires the following conditions for any $k\in[T]$:

\noindent\textbf{Condition (i):} The reservation arrival rate satisfies
\begin{align}\label{eq:require_booking_appendix}
    \lambda^{(1,k)}\ge  \frac{\delta C}{q^{(1,k)}}+\frac{4\iota}{3} +\sqrt{\frac{8\iota^2}{3}+\frac{2\delta C}{q^{(1,k)}}\cdot\iota}.
\end{align}

\noindent\textbf{Condition (ii):} The walk-in arrival rate satisfies
\begin{align}\label{equ::require_lambda_appendix}
    \lambda^{(2,k)}\ge 5\iota+1+4\sqrt{3\delta C\cdot \iota }+\sqrt{10\iota^2 +2\iota+8\iota\sqrt{3\delta C\cdot\iota}}.
\end{align}

These conditions ensure that reservation rates slightly exceed expected departure rates (condition (i)), and walk-in rates are sufficient to offset uncertainties from cancellations and no-shows with confidence level $\exp(-\iota)$ over horizon $T$ (condition (ii)).

\subsection{Formulas for DASS Policy}\label{appendix::DASS_formulas}

This subsection provides the precise formulas for the booking threshold $\hat{B}^{(k)}_{t}$ and estimated capacity $\hat{C}^{(k)}$ used in Algorithm \ref{alg:DASS_stage_I}.

\paragraph{Booking Threshold.} The estimated number of confirmed bookings plus safety stock is given by:
\begin{align}\label{eq:hat_B_t_full}
    \hat{B}^{(k)}_{t} = p^{(k)}(t){B}^{(k)}_{t}+\frac{\iota(1-p^{(k)}_{t})}{3}+\sqrt{\left(\frac{\iota(1-p^{(k)}_{t})}{3}\right)^2 +2\iota{{B}^{(k)}_{t}p^{(k)}(t)(1-p^{(k)}(t))}}.
\end{align}
Here $p^{(k)}(t)$ is the confirmation rate at time $t$ on day $k$, $B^{(k)}_{t}$ is the current number of bookings, and $\iota=O(\log(CT))$ is the confidence parameter.

\paragraph{Estimated Capacity.} The capacity available for new bookings depends on the occupancy duration distribution:

\noindent\textbf{(i) Geometric distribution:} When occupancy duration follows a geometric distribution with parameter $q$, the estimated capacity is
\begin{align}\label{eq:hat_C_geometric}
    \hat{C}^{(k)} = C -\frac{1-q}{q}\left[\lambda^{(1,k)}-\frac{4\iota}{3}-\sqrt{\frac{8\iota^2}{3}+\frac{2\iota(1-q)}{q}\lambda^{(1,k)}}\right].
\end{align}

\noindent\textbf{(ii) Constant duration:} When occupancy duration is constant $d$, the estimated capacity is
\begin{align}\label{eq:hat_C_constant}
    \hat{C}^{(k)} = C -d\left[\lambda^{(1,k)}-\frac{4\iota}{3}-\sqrt{\frac{8\iota^2}{3}+2\iota\lambda^{(1,k)}}\right].
\end{align}

These formulas ensure that with high probability (at least $1-\exp(-\iota)$), the accepted bookings will not exceed the available capacity after accounting for expected occupancy from previous days.

\paragraph{Expected Show-ups in Stage II.} For Algorithm \ref{alg:DASS_stage_II}, the expected show-ups $\hat{N}^{(k)}_{u}$ at time $u$ on Day $k$ is computed as:
\begin{align}\label{equ::StageII standard full}
\hat{N}^{(k)}_{u}=\begin{cases}
    \underbrace{B^{(k)}_{1,u}+q^{(1.k)}\cdot(B^{(k)}-B^{(k)}_{1,u}-B^{(k)}_{2,u})}_{\text{Expected  Type I arrivals}}+\underbrace{W_{1,u}^{(k)}+\alpha\cdot \int_{u}^{1}\lambda^{(2,k)}(s) ds}_{\text{Weighted expected Type II arrivals}},~~0\le u<v,\\
    B^{(k)}_{1,v}+B^{(k)}_{3,v}+W_{1,u}^{(k)},~~ u \ge v,
\end{cases}
\end{align}
where $B^{(k)}_{1,u}$ is the number of Type I customers who checked in during $[0,u]$, $B^{(k)}_{2,u}$ is the number who cancelled, $W^{(k)}_{1,u}$ is the number of accepted walk-ins, $B^{(k)}_{3,v}$ is the confirmed number of Type I customers after confirmation call at time $v$, and $0<\alpha<1$ is a tunable parameter.

\subsection{Proof of Proposition \ref{prop:lower_bound}}
Consider the instance where $C=1$, $\lambda^{(1,k)}=1,\lambda^{(2,k)}=\sqrt{\iota},p^{(k)}=1/2$,  $\loss^{(k)}=r^{(k)}=1$ and $\delta =1$, i.e. the occupancy rate $d=1$ almost surely. In this case, when more than one unit of bookings $B^{(k)}$ are accepted in Stage I, there is a constant $c_1>0$ such that $\P(\text{Inccur overbooking loss}|B^{(k)}>1) > c_1$. When $B^{(k)}=1,$ by \eqref{equ::poisson tail lower}, there is a constant $c_2>0$ such that the following events happen: the total number of arrivals in Stage I $\Lambda^{(1,k)}(0,k) > 1$ and one booking other than the accepted one is not cancelled in Stage II while the accepted one is cancelled in Stage II. When these events happen, the hindsight optimal policy will not have loss and run with full capacity. Moreover there exists constant $c_3$ such that $\P(\Lambda^{(2,k)}[0,1]=0)\ge \exp(-c_3\iota)$ by \eqref{equ::poisson proportion low} and \eqref{equ::poisson proportion upper}. Then we have $\P(\text{Incur idling loss}|B^{(k)}=1)\ge c_2\exp(-\iota).$  When $B^{(k)}=0,$ we similarly have $\P(\text{Incur idling loss}|B^{(k)}=1)\ge c_2\exp(-\iota)$. WLOG we assume $\exp(-c_3\iota)\le \min\{c_1,c_2\}$ as $T\to\infty,$ we then have 
\begin{align*}
    \regret^T(\pi_1,\pi_2)\ge \sum_{t=1}^T\min\{\loss^{(k)},r^{(k)}\}\exp(-c_3\iota) .
\end{align*}

\subsection{Proof of Theorem \ref{prop::StageI regret}}\label{appendix::stage I}
Before proving Theorem \ref{prop::StageI regret}, we first state some well-known concentration inequalities.

\begin{itemize}
    \item Bernoulli trials: For $W \sim B(n,p)$, By one sided Bernstein inequality, we have
    \begin{align}
        \Pr\left[W-np\ge \frac{1-p}{3}\cdot \iota + \sqrt{\left(\frac{1-p}{3}\cdot \iota\right)^2+2 \iota np(1-p)}\right]&\le \exp(-\iota),\label{equ::bounoulli tail}\\
        \Pr\left[W-np\le -\frac{p}{3}\cdot \iota -\sqrt{\left(\frac{p}{3}\cdot \iota\right)^2+2\iota np(1-p)}\right]&\le \exp(-\iota).\label{equ::bounoulli tail lower}
    \end{align}
    \item Poisson distribution: For $V \sim \text{Poisson}(\lambda)$ and $0<\alpha<1$, we have
    \begin{align}
        \Pr\left[V-\lambda \ge  \iota + \sqrt{\iota^2 +2 \lambda \cdot \iota}\right] &\le \exp(-\iota) ,\\
        \Pr\left[V-\lambda \le -\iota - \sqrt{\iota^2 +2 \lambda \cdot \iota}\right] &\le \exp(-\iota).\label{equ::poisson tail lower}\\
        \Pr\left[V \le (1-\alpha)\lambda\right] &\le \exp(-\lambda(1-\alpha)\log(1-\alpha)+\alpha), \label{equ::poisson proportion low}\\
        \Pr\left[V \ge (1+\alpha)\lambda\right] &\le \exp(-\lambda(1-\alpha)\log(1-\alpha)+\alpha). \label{equ::poisson proportion upper}
    \end{align}
\end{itemize}

Then we begin our proof of Proposition \ref{prop::StageI regret}. We analyse the overbooking loss and the idling loss in Appendices \ref{appendix::overbooking} and \ref{appendix::idling}, respectively, and combine them in Appendix \ref{appendix::stageI combine} which concludes the proof.
\subsubsection{Analyse the overbooking loss}\label{appendix::overbooking}
To avoid overbooking, we bound the number of confirmed bookings using concentration inequalities. The following analysis shows that DASS-I maintains this bound with high probability, preventing capacity violations. First, we prove that our algorithm ensures no overbooking loss for a hindsight optimal algorithm on Stage II with high probability. Denote the time  that our policy accepts the \textit{last} type I customer to be $t\le k$. By definition, our policy ensures that:
\begin{align}
    \hat{B}^{(k)}_{t} &={p^{(k)}(t){B}^{(k)}_{t}}+{\frac{\iota(1-p^{(k)}_{t})}{3}+\sqrt{\left(\frac{\iota(1-p^{(k)}_{t})}{3}\right)^2 +2\iota{{B}^{(k)}_{t}p^{(k)}(t)(1-p^{(k)}(t))}{}}}\le \hat{C}^{(k)}. \label{equ::Bk bound}
\end{align}
Given the number of accepted bookings at time $t$, i.e., ${B}^{(k)}_{t}$, we know the ultimate number of accepted bookings ${B}^{(k)} \sim B({B}^{(k)}_{t},p^{(k)}(t))$. Thus, by \eqref{equ::bounoulli tail}, we know $\Pr[{B}^{(k)}\le \hat{C}^{(k)}]\ge 1- \exp(-\iota)$. Moreover, in the second stage, given the number of non-cancelled reservations ${B}^{(k)}$, we know the final number of checked-in type I customers (denoted as $n_{1}^{(k)}$) follows another Bernoulli distribution $n_{1}^{(k)}\sim B({B}^{(k)},q^{(1,k)})$. Reall that $\hat{C}^{(k)}$ satisfies
\begin{align*}
     &\quad q^{(1,k)}\hat{C}^{(k)}+\frac{\iota(1-q^{(1,k)})}{3}+\sqrt{\left(\frac{\iota(1-q^{(1,k)})}{3}\right)^2+2{\iota \hat{C}^{(k)}q^{(1,k)}(1-q^{(1,k)})}{}} \\
    &=\underline{C}^{(k)}:=\left\{\begin{matrix}
        (1-q)C-\frac{\iota(1-q)}{3}- \sqrt{\left(\frac{1-q}{3}\right)^2+ {2\iota C q(1-q)}{}}&~~\text{(Geometric)}\\
        \frac{C}{d} &~~\text{(Constant)}
    \end{matrix}\right.
\end{align*}
Thus, again by \eqref{equ::bounoulli tail} and the high probability upper bound on $B^{(k)}$ \eqref{equ::Bk bound}, with probability at least $1-2\exp(-\iota)$, it holds that
\begin{align*}
    n_{1}^{(k)}\le \left\{\begin{matrix}
        (1-q)C-\frac{\iota(1-q)}{3}- \sqrt{\left(\frac{1-q}{3}\right)^2+ {2\iota C q(1-q)}{}}&~~\text{(Geometric)}\\
        \frac{C}{d} &~~\text{(Constant)}
    \end{matrix}\right.
\end{align*}
For the geometric distributed duration case, suppose there are $n^{(k-1)}$ checked-in customers on Day $k-1$, we know the number of vacant rooms on Day $k$ for the newly checked-in customers $C^{(k)}$ satisfies $C^{(k)}-C+n^{(k-1)}\sim B(n^{(k-1)},q)$. Thus, with probability $1-\exp(-\iota)$, we have at least
\begin{align*}
     C^{(k)}&\ge C-n^{(k-1)}+(1-q)n^{(k-1)}-\frac{\iota(1-q)}{3}- \sqrt{\left(\frac{1-q}{3}\right)^2+ {2\iota n^{(k-1)} q(1-q)}{}}\\
    &\ge   (1-q)C-\frac{\iota(1-q)}{3}- \sqrt{\left(\frac{1-q}{3}\right)^2+ {2\iota C q(1-q)}{}}
\end{align*}
vacant rooms. The second inequality holds because the term is monotone decreasing with $n^{(k-1)}$. Altogether, with probability at least $1-3\exp(-\iota)$, we have $n_{1}^{(k)}\le \hat{C}^{(k)}\le C^{(k)}$, i.e., the final number of checked-in type I customers does not exceeds the number of vacant rooms. We denote the joint event over all the $T$ days as $E^{\rm geo}_{\rm no-over}$ and $E^{\rm const}_{\rm no-over}$, respectively for the two duration cases, either of which happens with probability at least $1-3T\exp(-\iota)$.

Thus, there exists an offline policy $\pi_2^{\star,{\rm geo}}$ incurring no overbooking loss for a hindsight optimal policy on Stage II under $E^{\rm geo}_{\rm no-over}$. To be specific, $\pi_2^{\star,{\rm geo}}$ maximizes the revenue on Day $k$ by reserving the exact number of rooms for final arrival of type II customers and leaving the remaining rooms for walk-in customers.

For the constant duration case, under the joint event that $n^{(k)}_1 \le \frac{C}{d}$ for any $k\in[T]$, we know the hindsight optimal policy on Stage II always prevents overbooking by controlling the number of accepted walk-in customers on each day to be no more than the room capacity that is left for the walk-in customers. Thus, there also exists an offline policy $\pi_2^{\star,{\rm const}}$ incurring no overbooking loss for a hindsight optimal policy on Stage II under $E^{\rm const}_{\rm no-over}$. To adapt to the threshold we set in Stage I, we set $\pi_2^{\star,{\rm const}}$ to allocate at most $C/d$ rooms to newly checked-in customers (including both type I and type II) and maximize the revenue on Day $k$ by reserving the exact number of rooms for final arrival of type II customers and leaving the remaining rooms for walk-in customers under the room capacity of $C/d$. 

Thus, in both cases, with probability at least $1-3T\exp(-\iota)$, our Stage I policy leads to zero overbooking loss with the offline stage II policy $\pi_2^{\star,{\rm geo}}$ or $\pi_2^{\star,{\rm const}}$
throughout the entire decision process.

\subsubsection{Analyse the idling loss}\label{appendix::idling}
Next, we prove that with sufficient walk-in customers, with high probability our stage I policy combined with the offline stage II policy also achieves no more idling loss compared with the optimal policy. For the room allocation on day $k$, if  $\hat{B}^{(k)}_t$ has not reached the threshold $\hat{C}^{(k)}$ throughout stage I (we denote the event as $E_{k,{\rm low}}$). We know our policy keeps accepting type I customers without rejecting any one. Thus, $E_{k,{\rm low}}$ implies that we have not rejected any reservations and the threshold is not reached at the end of Stage I, i.e., 
\begin{align*}
    \hat{B}^{(k)}_{k} 
    &=B_t^{(k)}+\frac{2\iota}{3}=\sum_{i=1}^{\Lambda^{(1,k)}(k-k_0,k)}Y_1^{(0,k)}
    \le \hat{C}^{(k)}.
\end{align*}
Here $Y_i^{(0,k)}\sim {\rm Bernoulli}(p^{(k)}(X_i^{(0,k)}))$ for $i=1,2,\dots,\Lambda^{(1,k)}(k-k_0,k)$. Note that 
\begin{align*}
    \sum_{i=1}^{\Lambda^{(1,k)}(k-k_0,k)}Y_1^{(0,k)}\sim {\rm Poisson}\left(\int_{k-k_0}^{k}\lambda^{(1,k)}(t)1-p^{(k)}(t) dt\right)={\rm Poisson}\left(\lambda^{(1,k)}\right).
\end{align*}
Thus, by the lower tail bound of Poison distribution \eqref{equ::poisson tail lower}, we have with probability at least $1-\exp(-\iota)$,
\begin{align*}
    B_t^{(k)}\ge \lambda^{(1,k)}-\iota - \sqrt{\iota^2 +2 \lambda^{(1,k)} \cdot \iota}.
\end{align*}
Consequently, we have
\begin{align*}
    \hat{C}^{(k)}\ge  B_t^{(k)}+\frac{2\iota}{3}\ge \lambda^{(1,k)}-\frac{\iota}{3}-\sqrt{\iota^2 +2 \lambda^{(1,k)} \cdot \iota},
\end{align*}
which imples that
\begin{align*}
    \lambda^{(1,k)}\le  \hat{C}^{(k)}+\frac{4\iota}{3} +\sqrt{\frac{8\iota^2}{3}+2\hat{C}^{(k)}\cdot\iota}.
\end{align*}
Since we have by definition that
\begin{align}
&\quad q^{(1,k)}\hat{C}^{(k)}+\frac{\iota(1-q^{(1,k)})}{3}+\sqrt{\left(\frac{\iota(1-q^{(1,k)})}{3}\right)^2+2{\iota \hat{C}^{(k)}q^{(1,k)}(1-q^{(1,k)})}{}} \nonumber\\
    &=\left\{\begin{matrix}
        (1-q)C-\frac{\iota(1-q)}{3}- \sqrt{\left(\frac{\iota(1-q)}{3}\right)^2+ {2\iota C q(1-q)}{}}&~~\textbf{(Geometric)}\\
        \frac{C}{d} &~~\textbf{(Constant)}
    \end{matrix}\right.\nonumber\\
    &\le \delta C \nonumber
\end{align}
for any $k\in[T]$, it holds that
\begin{align*}
    \lambda^{(1,k)}\le \hat{C}^{(k)}+\frac{4\iota}{3} +\sqrt{\frac{8\iota^2}{3}+2\hat{C}^{(k)}\cdot\iota}&\le \frac{\delta C}{q^{(1,k)}}+\frac{4\iota}{3} +\sqrt{\frac{8\iota^2}{3}+\frac{2\delta C}{q^{(1,k)}}\cdot\iota}
\end{align*}
which contradicts our assumption on $\lambda^{(1,k)}$ in \eqref{eq:require_booking_appendix}. Thus, we have $\Pr[E_{k,{\rm low}}]\le \exp(-\iota)$ for any $k\in[T]$. We denote  $E_{\rm high}$ to be the event that excludes all the $E_{k,{\rm low}}$ for $k\in[T]$.

We now analyze the complementary case where the threshold approaches capacity, i.e., we exclude the event $E_{k,{\rm low}}$ and work under $E_{\rm high}$. Using lower tail bounds on booking confirmations, we show that sufficient confirmed reservations prevent idling with high probability. In this case, we have $ \hat{C}^{(k)} -1 \le \hat{B}^{(k)}_t \le \hat{C}^{(k)}$ for some $t\in[0,1]$, and we do not receive any additional type I customer after time $t$. Thus, by the lower bound of Bernoulli distribution \eqref{equ::bounoulli tail lower} , we similarly have with probability $1-\exp(-\iota)$,
\begin{align*}
   B^{(k)}&\ge {p^{(k)}(t){B}^{(k)}_{t}}-{\frac{\iota p^{(k)}_{t}}{3}-\sqrt{\left(\frac{\iota p^{(k)}_{t}}{3}\right)^2 +2\iota{{B}^{(k)}_{t}p^{(k)}(t)(1-p^{(k)}(t))}{}}}\\
   &=\hat{B}^{(k)}_{t}-{\frac{\iota(1-p^{(k)}_{t})}{3}-\sqrt{\left(\frac{\iota(1-p^{(k)}_{t})}{3}\right)^2 +2\iota{{B}^{(k)}_{t}p^{(k)}(t)(1-p^{(k)}(t))}{}}}\\
   &\quad-{\frac{\iota p^{(k)}_{t}}{3}-\sqrt{\left(\frac{\iota p^{(k)}_{t}}{3}\right)^2 +2\iota{{B}^{(k)}_{t}p^{(k)}(t)(1-p^{(k)}(t))}{}}}\\
   &\ge \hat{B}^{(k)}_{t}-\frac{\iota}{3}-2\sqrt{\left(\frac{\iota}{3}\right)^2 +2\iota{{B}^{(k)}_{t}p^{(k)}(t)(1-p^{(k)}(t))}{}}\\
   &\ge \hat{C}^{(k)}-1-\frac{\iota}{3}-2\sqrt{\left(\frac{\iota}{3}\right)^2 +2\iota{\hat{C}^{(k)}p^{(k)}(t)(1-p^{(k)}(t))}{}}.
\end{align*}
In the last inequality, we invoke $\hat{B}^{(k)}_{t}\ge \hat{C}^{(k)}-1$ and ${B}^{(k)}_{t}\le \hat{B}^{(k)}_{t}\le \hat{C}^{(k)}$. Thus, again by \eqref{equ::bounoulli tail lower}, with probability $1-2\exp(-\iota)$, we have the final number of checked-in type I customers lower bounded by
\begin{align*}
    n_{1}^{(k)}&\ge q^{(1,k)} B^{(k)}-\frac{\iota q^{(1,k)}}{3}-\sqrt{\left(\frac{\iota q^{(1,k)}}{3}\right)^2+2{\iota  B^{(k)}q^{(1,k)}(1-q^{(1,k)})}{}}\\
    &\ge q^{(1,k)}\cdot \left(\hat{C}^{(k)}-1-\frac{\iota}{3}-2\sqrt{\left(\frac{\iota}{3}\right)^2 +2\iota{\hat{C}^{(k)}p^{(k)}(t)(1-p^{(k)}(t))}{}}\right)\\
    &\quad -\frac{\iota q^{(1,k)}}{3}-\sqrt{\left(\frac{\iota q^{(1,k)}}{3}\right)^2+2{\iota  B^{(k)}q^{(1,k)}(1-q^{(1,k)})}{}}\\
    & \ge q^{(1,k)}\hat{C}^{(k)}+\frac{\iota(1-q^{(1,k)})}{3}+\sqrt{\left(\frac{\iota(1-q^{(1,k)})}{3}\right)^2+2{\iota \hat{C}^{(k)}q^{(1,k)}(1-q^{(1,k)})}{}}\\
    &\quad -q^{(1,k)}\cdot \left(1+\frac{\iota}{3}+2\sqrt{\left(\frac{\iota}{3}\right)^2 +2\iota{\hat{C}^{(k)}p^{(k)}(t)(1-p^{(k)}(t))}{}}\right)\\
    &\quad - \frac{\iota}{3} -2\sqrt{\left(\frac{\iota}{3}\right)^2+2{\iota \hat{C}^{(k)}q^{(1,k)}(1-q^{(1,k)})}{}}\\
    &=\underline{C}^{(k)}-\frac{(1+q^{(1,k)})\iota}{3}-q^{(1,k)}-2q^{(1,k)}\sqrt{\left(\frac{\iota}{3}\right)^2 +2\iota{\hat{C}^{(k)}p^{(k)}(t)(1-p^{(k)}(t))}{}}-2\sqrt{\left(\frac{\iota}{3}\right)^2+2{\iota \hat{C}^{(k)}q^{(1,k)}(1-q^{(1,k)})}{}}.
\end{align*}
For the geometric distributed duration case, since the number of vacant rooms on Day $k$, i.e., $C^{(k)}$ satisfies $C^{(k)}-C+n^{(k-1)}\sim B(n^{(k-1)},q)$. Thus, given that $n^{(k-1)}=C$, i.e., the hotel was at full capacity on Day $k-1$, with probability $1-\exp(-\iota)$, we have at most
\begin{align*}
     C^{(k)}&\le C-n^{(k-1)}+(1-q)n^{(k-1)}+\frac{\iota q}{3}+ \sqrt{\left(\frac{q}{3}\right)^2+ {2\iota n^{(k-1)} q(1-q)}{}}\\
    &=   (1-q)C+\frac{\iota q}{3}+ \sqrt{\left(\frac{\iota q}{3}\right)^2+ {2\iota C q(1-q)}{}}\\
    &\le \underline{C}^{(k)} +\frac{\iota}{3}+2\sqrt{\left(\frac{\iota}{3}\right)^2+ {2\iota C q(1-q)}{}}
\end{align*}
vacant rooms. Here we invoke $\underline{C}^{(k)}=(1-q)C-\frac{\iota(1-q)}{3}- \sqrt{\left(\frac{1-q}{3}\right)^2+ {2\iota C q(1-q)}{}}$ in the last inequality. Thus, after offering rooms to checked-in type I customers, with probability at least $1-4\exp(-\iota)$, there are at most
\begin{align*}
    C^{(k)}-n_1^{(k)}&\le \frac{(1+q^{(1,k)})\iota}{3}+q^{(1,k)}+2q^{(1,k)}\sqrt{\left(\frac{\iota}{3}\right)^2 +2\iota{\hat{C}^{(k)}p^{(k)}(t)(1-p^{(k)}(t))}{}}\\&\quad+2\sqrt{\left(\frac{\iota}{3}\right)^2+2{\iota \hat{C}^{(k)}q^{(1,k)}(1-q^{(1,k)})}{}}+ \frac{\iota}{3}+2\sqrt{\left(\frac{\iota}{3}\right)^2+ {2\iota C q(1-q)}{}}\\
    &\le \frac{(8+q^{(1,k)})\iota}{3}+q^{(1,k)}+2\sqrt{3}\cdot\sqrt{2\iota C(1-q)\left(q^{(1,k)}p^{(k)}(t)(1-p^{(k)}(t))+(1-q^{(1,k)})+q\right)}\\
    &\le 3\iota+1+4\sqrt{3\iota C(1-q)}
\end{align*}
vacant rooms left for walk-in customers. Here we invoke $\hat{C}^{(k)}q^{(1,k)}\le (1-q)C$ in the second inequality and $q \le 1$ in the last one. Recall that by the lower bound of Poisson distribution \eqref{equ::poisson tail lower}, with probability at least $1-\exp(-\iota)$, we have at least $ \lambda^{(2,k)}-\iota - \sqrt{\iota^2 +2 \lambda^{(2,k)} \cdot \iota}$ walk-in customers, so the idling loss for the Stage-II hindsight optimal algorithm is at most
\begin{align}\label{equ::Requirements on lambda}
   r^{(k)}\cdot \left( \lambda^{(2,k)}-\iota - \sqrt{\iota^2 +2 \lambda^{(2,k)} \cdot \iota}- 3\iota-1-4\sqrt{3\iota C(1-q)}\right)\vee 0.
\end{align}
With sufficiently large $\lambda^{(2,k)}$, we have with probability at least $1-4\exp(-\iota)$, the hindsight optimal algorithm on stage II can fill all the vacant rooms by accepting enough customers, achieving zero idling loss. This requirement on $ \lambda^{(2,k)}$ is equivalent to letting
\begin{align*}
    &\lambda^{(2,k)}\ge 5\iota+1+4\sqrt{3\iota C(1-q)}+\sqrt{10\iota^2 +2\iota+8\iota\sqrt{3\iota C(1-q)}}.
\end{align*}
We denote this event as $E^{\rm geo}_{k,{\rm no-idle}}$. Thus, under the event $E_{\rm high}$ and  the joint event of $E^{\rm geo}_{{\rm no-idle}}:=\cup_{k\in[T]} E^{\rm geo}_{k,{\rm no-idle}}$ which happens with probability at least $1-4T\exp(-\iota)$, as long as the hotel is at full capacity on Day $0$, we can ensure filling all the vacant rooms and achieve zero idling loss throughout the $T$ days by the offline stage II policy $\pi_2^{\star,{\rm geo}}$.

For the constant duration case, following the Stage II policy $\pi_2^{\star,{\rm const}}$, let's strictly allocate $C^{(k)}=C/d$ to the newly checked-in customers on Day $k$. For the first $d$ days, we may suffer from vacant rooms since for any day $k\le d$, we have not ensured accommodating $C/d$ customers for all the past $d$ days. To be specific, there are at most $C(d-k+1)/d$ vacant rooms, so leading to a idling loss of at most $r^{(k)}\cdot C(d-k)/d$.
On the other hand, for $k>d$, there are exactly $C/d$ customers that have been stayed in the hotel since Day $k-i$ for any $i=1,2,3,\dots,d$. Thus, the departure of $C/d$ customers that checked in on Day $k-d$ leaves exactly $C/d$ vacant rooms. Then we know there are at most
\begin{align*}
      C^{(k)}-n_1^{(k)}&\le \frac{(1+q^{(1,k)})\iota}{3}+q^{(1,k)}+2q^{(1,k)}\sqrt{\left(\frac{\iota}{3}\right)^2 +2\iota{\hat{C}^{(k)}p^{(k)}(t)(1-p^{(k)}(t))}{}}\\&\quad+2\sqrt{\left(\frac{\iota}{3}\right)^2+2{\iota \hat{C}^{(k)}q^{(1,k)}(1-q^{(1,k)})}{}}\\
      &\le 3\iota+1+2\sqrt{2}\cdot\sqrt{ \frac{2\iota C}{d}\cdot \left(q^{(1,k)}p^{(k)}(t)(1-p^{(k)}(t))+1-q^{(1,k)} \right)}\\
      &\le 3\iota +1 +4\sqrt{\frac{\iota C}{d}}.
\end{align*}
vacant rooms left for walk-in customers. Thus, the idling loss is at most 
\begin{align*}
    r^{(k)}\cdot \left( \lambda^{(2,k)}-\iota - \sqrt{\iota^2 +2 \lambda^{(2,k)} \cdot \iota}-  3\iota -1 -4\sqrt{\frac{\iota C}{d}} \right) \vee 0,
\end{align*}
which approaches zero when
\begin{align*}
\lambda^{(2,k)}\ge 5\iota+1+4\sqrt{\frac{\iota C}{d}}+\sqrt{10\iota^2+2\iota+8\iota\sqrt{\frac{\iota C}{d}}}.
\end{align*}
We denote this event as $E^{\rm const}_{k,{\rm no-idle}}$. Thus, conditional on the event that excludes $\{E_{k,{\rm low}}\}_{k\in[T]}$, under the joint event of $E^{\rm geo}_{{\rm no-idle}}:=\cup_{k\in[T]} E^{\rm geo}_{k,{\rm no-idle}}$ which happens with probability at least $1-4T\exp(-\iota)$, as long as the hotel is at full capacity on Day $0$, we can ensure filling all the vacant rooms and achieve zero idling loss throughout the $T$ days by the stage II.

\subsubsection{Combine the losses}\label{appendix::stageI combine}
By combining the analysis of overbooking loss and idling loss, for the geometric duration case, we have expected regret in Stage I bounded by
    \begin{align*}
\E\left[\sum_{k=1}^{T}\mathcal{R}_1^{(k)}(\hat{\pi}_1)\right]  &\le \sum_{k=1}^T\E_{\Fi^{(T)},\Fii^{(T)}}\brk{{\Lossii^{(k)}(\pi^{\star}_2,\{B^{(k')}(\hat\pi_1))\}_{k'=1}^{k})-\Lossii^{(k)}(\pi^{\star}_2,\{B^{(k')}(\pi^{\star}_1))\}_{k'=1}^{k})}}\\
&\le\sum_{k=1}^T\E_{\Fi^{(T)},\Fii^{(T)}}\brk{{\Lossii^{(k)}(\pi^{\star,{\rm geo}}_2,\{B^{(k')}(\hat\pi_1))\}_{k'=1}^{k})-\Lossii^{(k)}(\pi^{\star}_2,\{B^{(k')}(\pi^{\star}_1))\}_{k'=1}^{k})}}\\
&\le \sum_{k=1}^T\E_{\Fi^{(T)},\Fii^{(T)}}\brk{{\Lossii^{(k)}(\pi^{\star,{\rm geo}}_2,\{B^{(k')}(\hat\pi_1))\}_{k'=1}^{k})}}\\
     &\le
     (1-  \Pr[E^{\rm geo}_{{\rm no-idle}}] -\Pr[E_{{\rm high}}])\cdot C\max_{k}r^{(k)} +(1-\Pr[E^{\rm geo}_{\rm no-over}])\cdot C\max_{k}\ell^{(k)}\\
     &\quad+\sum_{k=1}^T\E_{\Fi^{(T)},\Fii^{(T)}}\brk{{\Lossii^{(k)}(\pi^{\star,{\rm geo}}_2,\{B^{(k')}(\hat\pi_1))\}_{k'=1}^{k})}\mid E^{\rm geo}_{{\rm no-idle}},E_{{\rm high}},E^{\rm geo}_{\rm no-over}}\\
     &\le 5CT\exp(-\iota)\cdot \max_{k}r^{(k)}+3CT\exp(-\iota)\cdot \max_{k}\ell^{(k)}.
    \end{align*}
In the last inequality, we invoke the fact that we suffer from at most $C\max_{k}\ell^{(k)}$ overbooking loss and $C\max_{k}r^{(k)}$ idling loss on each day, and under the joint event of $E^{\rm geo}_{{\rm no-idle}},E_{{\rm high}},E^{\rm geo}_{\rm no-over}$, we achieve zero loss on each day.

For the constant duration case, we similarly have
\begin{align*}   \E\left[\sum_{k=1}^{T}\mathcal{R}_1^{(k)}(\hat{\pi}_1)\right]  &\le\sum_{k=1}^T\E_{\Fi^{(T)},\Fii^{(T)}}\brk{{\Lossii^{(k)}(\pi^{\star,{\rm const}}_2,\{B^{(k')}(\hat\pi_1))\}_{k'=1}^{k})}}\\
 &\le
     (1-  \Pr[E^{\rm const}_{{\rm no-idle}}] -\Pr[E_{{\rm high}}])\cdot \max_{k}r^{(k)} +(1-\Pr[E^{\rm const}_{\rm no-over}])\cdot \max_{k}\ell^{(k)}\\
     &\quad+\sum_{k=1}^T\E_{\Fi^{(T)},\Fii^{(T)}}\brk{{\Lossii^{(k)}(\pi^{\star,{\rm const}}_2,\{B^{(k')}(\hat\pi_1))\}_{k'=1}^{k})}\mid E^{\rm const}_{{\rm no-idle}},E_{{\rm high}},E^{\rm const}_{\rm no-over}}\\
     &\le 5CT\exp(-\iota)\cdot \max_{k}r^{(k)}+3CT\exp(-\iota)\cdot \max_{k}\ell^{(k)}+\sum_{k=1}^{d}\frac{C(d-k)}{d}\cdot r^{(k)}\\
     &\quad+\sum_{k=d+1}^T\E_{\Fi^{(T)},\Fii^{(T)}}\brk{{\Lossii^{(k)}(\pi^{\star,{\rm const}}_2,\{B^{(k')}(\hat\pi_1))\}_{k'=1}^{k})}\mid E^{\rm const}_{{\rm no-idle}},E_{{\rm high}},E^{\rm const}_{\rm no-over}}\\
     &=5CT\exp(-\iota)\cdot \max_{k}r^{(k)}+3CT\exp(-\iota)\cdot \max_{k}\ell^{(k)}+\sum_{k=1}^{d}\frac{C(d-k)}{d}\cdot r^{(k)}.
\end{align*}
The proof is complete.

\subsection{Proof of Lemma \ref{lemma::decoupling}}\label{appendix::decoupling}
By the proof of Proposition \ref{prop::StageI regret}, we know $\hat{\pi}^{\star}_2$ is exactly equivalent to the offline policy $\pi_2^{\star,{\rm geo}}$ or $\pi_2^{\star,{\rm const}}$. Thus, for the geometric duration case, under the joint event of $E^{\rm geo}_{{\rm no-idle}},E_{{\rm high}},E^{\rm geo}_{\rm no-over}$, we have $\Lossii^{(k)}(\pi^{\star,{\rm geo}}_2,\{B^{(k')}(\hat\pi_1))\}_{k'=1}^{k})=0$ for $k\in[T]$. 
For the constant duration case, as analysed in the proof of Proposition \ref{prop::StageI regret}, under the joint event of $E^{\rm const}_{{\rm no-idle}},E_{{\rm high}},E^{\rm const}_{\rm no-over}$, for $k\le d$ which happens with probability at least $1-8T\exp(-\iota)$, we face a revenue loss of at most $(d-k)/C\cdot r^{(k)}$. For $k>d$, we eliminate the loss. Thus, we have 
\begin{align*}
   \sum_{k=1}^{T}\Lossii^{(k)}(\pi^{\star,{\rm geo}}_2,\{B^{(k')}(\hat\pi_1))\}_{k'=1}^{k})\le \sum_{k=1}^{d}\frac{C(d-k)}{d}\cdot r^{(k)}.
\end{align*}
In both cases, we achieve a constant loss by the mixed policy of $(\hat\pi_1,\hat{\pi_2}^{\star})$, thus a constant difference of loss between $(\hat\pi_1,\hat{\pi_2}^{\star})$ and $(\hat\pi_1,{\pi_2}^{\star})$, which concludes our proof.

\subsection{Proof of Theorem \ref{prop:HO_gap}}\label{appendix::Stage II}

\paragraph{Complete bound in Theorem \ref{prop:HO_gap}.} The exact bound for the Stage II loss is:
\begin{align}\label{eq:stage_II_full_bound}
    \begin{aligned}
        &\quad\E_{\Fii^{(k)}}\left[\Lossii^{(k)}(\hat{\pi}_2,\{B^{(k')}(\hat\pi_1))\}_{k'=1}^{k})-\Lossii^{(k)}(\hat{\pi}^{\star}_2,\{B^{(k')}(\hat\pi_1))\}_{k'=1}^{k})\right]\\&\le \lambda^{(2,k)}_{[0,v]}\cdot \left((\ell_{\rm over}^{(k)}+r^{(k)})\cdot\exp\left(\frac{-\frac{\left(\alpha\lambda^{(2,k)}_{[v,1]}\right)^2}{2}}{{B^{(k)}}q^{(1,k)}(1-q^{(1,k)})
        +\frac{\alpha\lambda^{(2,k)}_{[v,1]}}{3}}\right)+ r^{(k)}\cdot\exp\left(-\lambda^{(2,k)}_{[v,1]}\cdot\left(2\alpha\log(2\alpha)+1-2\alpha \right)\right)\right).
  \end{aligned}
\end{align}

\paragraph{Homogeneous case.} When both the walk-in customer flow and the check-in flow of Type I customers are homogeneous, i.e., $\lambda^{(2,k)}(s)=\lambda^{(2,k)}$ and $\gamma^{(k)}(s)=1$, the loss difference is bounded by
 \begin{align}\label{eq:stage_II_homogeneous}
         v\lambda^{(2,k)} \cdot \Bigg((\ell_{\rm over}^{(k)}+r^{(k)})\cdot\exp\left(\frac{-\frac{\alpha^2(1-v)\lambda^{(2,k)}}{2}}{\frac{{\delta C}(1-q^{(1,k)})}{\lambda^{(2,k)}}
        +\frac{\alpha}{3}}\right)+r^{(k)}\cdot\exp\left(-(1-v)\lambda^{(2,k)}\cdot\left(2\alpha\log(2\alpha)+1-2\alpha \right)\right)\Bigg).
    \end{align}
This error decays exponentially with $1-v$. The exact condition for achieving $O(1)$ regret is
\begin{align}\label{eq:stage_II_call_timing_exact}
    (1-v)\lambda^{(2,k)}\ge \max\left(\frac{\iota+\sqrt{\iota^2+18(1-v)\delta C\cdot \iota}}{3\alpha},\frac{\iota}{2\alpha\log(2\alpha)+1-2\alpha}\right).
\end{align}

\paragraph{Proof outline.} The proof of Theorem \ref{prop:HO_gap} proceeds as follows. In Appendices \ref{appendix::decision decoupling} and \ref{appendix::Stage 2 policy def}, we give a rigorous definition on the optimal single day policy and reformulate our stage II policy as a mixture of a fully online policy and the single-day offline optimal policy, respectively. In Appendix \ref{appendix::regret decompose}, we decompose the stage II regret into a summation of increment losses, each of which corresponds to  a wrong decision made on the walk-in customer. In Appendix \ref{appendix::wrong decision prob}, we further analyze the probability that a wrong decision is made by our Stage II policy. Finally, in Appendix \ref{appendix::combine}, by combining all the increment losses together and taking expectation over all the sample paths to conclude our proof.
\subsubsection{Decision decoupling}\label{appendix::decision decoupling}
In Lemma \ref{lemma::decoupling}, we have shown that under our Stage I algorithm, with high probability, the Stage
 II hindsight optimal policies can be decomposed into the concatenation of single-day policies
 \begin{align*}
     \pi_2^{\star}=\pi_{2}^{(1)\star} \times \pi_{2}^{(2)\star} \times \cdots \times \pi_{2}^{(T)\star},
 \end{align*}
 where the optimal single-day policy $\pi_{2,k}^{\star}$ makes Stage II decision on day $k$ that maximizes the revenue on Day $k$ with access to full customer information, regardless of its potential effect on the future days caused by the existence of duration. 
Given a fixed day $k$, we are given the total number of reserved customers and the number of vacant rooms allocated at the beginning of the check-in process on day $k$ are denoted $B^{(k)}$ and $\Tilde{C}^{(k)}$. For any time $u \in [0,1]$, we further denote the number of checked-in reserved customers, canceled reserved customers, and checked-in walk-in customers at time $u$ by $B^{(k)}_{1,u},B^{(k)}_{2,u}$ and $W^{(k)}_{1,u}$. For the future arrival of the remaining $B^{(k)}-B^{(k)}_{1,u}-B^{(k)}_{2,u}$ reserved customers, there are $B^{(k)}_{3,u}$ customers choosing to check in. With all the information known, we can rigorously define the optimal single-day policy $\pi_2^{
(k)\star}$ as follows.
    \begin{itemize}
        \item If $B^{
(k)}_{1,u}+W^{
(k)}_{1,u}+B^{
(k)}_{3,u} \le \Tilde{C}^{
(k)}$, which means that there will be no overbooking given the current number of accepted walk-in customers on day $k$,
$\pi_2^{
(k)\star}$ reserves $B^{
(k)}_{3,u}$ for the coming customers with reservations and allocate the remaining $\Tilde{C}^{
(k)}-B^{
(k)}_{1,u}-W^{
(k)}_{1,u}-B^{
(k)}_{3,u}$ rooms to potential walk-in customers.
        \item If $B^{(k)}_{1,u}+W_{1,u}^{(k)}+B_{3,u}^{(k)}> \Tilde{C}^{(k)}$, which means overbooking will happend, $\pi_2^{
(k)\star}$ stops receiving walk-in customers and reserves all the remaining rooms for the coming customers with reservations. At last, the hotel will face $B^{(k)}_{1,u}+W_{1,u}^{(k)}+B_{3,u}^{(k)}-\Tilde{C}^{(k)}$ overbooking reservations.
    \end{itemize}
It is easy to verify that this algorithm is \textbf{single-day offline optimal}, i.e., a solution to the optimization problem described in Lemma \ref{lemma::decoupling}. Full arrival and booking information eliminates all the randomness and the Stage II regret. 

 \subsubsection{Reformulate our policy}\label{appendix::Stage 2 policy def}
For the real setting, since we have no hindsight information $B_{3,u}^{(k)}$ before time $u\le v$, but only the online information $B^{(k)}$, $\Tilde{C}^{(k)}$, $B^{(k)}_{1,u},B^{(k)}_{2,u}$ and $W_{1,u}^{(k)}$, we need to design an online policy to make decisions. Recall that we have designed an online policy in Section \ref{sec:Stage II} that uses the conditional expectations of total shown-ups as the standard to decide whether to accept a new walk-in customer. To be specific, given the past information $B^{(k)}$, $\Tilde{C}^{(k)}$, $B^{(k)}_{1,u},B^{(k)}_{2,u}$ and $W_{1,u}^{(k)}$, we can compute the expected  shown ups as 
\begin{align*}
    \hat{N}^{(k)}_{u}=\E\left[B_{3,u}^{(k)}|B^{(k)},B^{(k)}_{1,u},B^{(k)}_{2,u}\right]+B^{(k)}_{1,u}+W_{1,u}^{(k)}+\alpha \E\left[W^{(k)}_{2,u}|B^{(k)},B^{(k)}_{1,u},B^{(k)}_{2,u}\right].
\end{align*}
Here $W^{(k)}_{2,u}$ is the number of arriving walk-in customers after time $u$, and $0\le \alpha \le 1$ is parameter to be chosen later. Only if $\hat{N}^{(k)}_u < \Tilde{C}^{(k)} $, we accept the walk-in customer showing up at time $u$. We denote the policy that only utilizes online information throughout the decision process as $\Tilde{\pi}_2^{(k)}$ (i.e., without confirmation call). Then, with a confirmation calling time at $v$, our policy $\hat{\pi}_2=\hat{\pi}_2^{(1)}\times\cdots\times\hat{\pi}_2^{(T)}$ can also be seen as a mixture of the online algorithm and the offline optimal one. We define the mixture policy as follows.

\begin{definition}
    \label{def:mix}
    For $0\le u\le 1$, we define $\mix^{(k)}_{2,u}$ as the policy that applies the online policy $\Tilde{\pi}_2^{(k)}$ described above in time $[0,u)$ and applying the single-day offline optimal policy $\pi^{(k)\star}_{2}$ to the remaining time period of $[u,1]$. Specially, define $\mix^{(k)}_{2,0}=\pi^{(k)\star}_{2}$ as the policy that applies the single-day optimal policy throughout the decision process and $\mix^{(k)}_{2,1}=\Tilde{\pi}_2^{(k)}$ as the policy that applies $\Tilde{\pi}_2^{(k)}$ throughout the process.
\end{definition}
 For our setting, if we gain the entire information, i.e., the confirmation call at time $v$, we actually conduct the mixture policy of $\hat{\pi}_2^{(k)}=\mix^{(k)}_{2,v}$. Note that here $\mix^{(k)}_{2,v}$ is dependent on $\hat\pi_2^{(k)}$. We do not explicitly write such dependence for sake of notation simplicity. 
 
  \subsubsection{Decompose the regret}\label{appendix::regret decompose}
 For notation simplicity, we abbreviate the revenue of any Stage II policy $\pi^{(k)}_2$ on Day $k$  under a realization of $\Fii^{(k)}$ as
 \begin{align*}
     \pow^{(k)}(\pi^{(k)}_2|\Fii^{(k)}):=\Lossii^{(k)}(\pi^{(k)}_2,\{B^{(k')}(\hat\pi_1))\}_{k'=1}^{k}).
 \end{align*}
 By definition, given any realization of  the customer process  $\Fii^{(k)}$, $\pow^{(k)}(\mix_{2,u}^{(k)}\mid \Fii^{(k)})$ is right-continuous staircase function that may only increase its value at point $u\in[0,1]$ when a walk-in customer comes at time $u$ and $\hat\pi_2^{(k)}$ makes a \textit{wrong decision} on whether to accept the customer compared with the single-day optimal policy, which leads to an increase in the loss, i.e., $\pow^{(k)}(\mix^{(k)}_{2,u}|\Fii^{(k)}) > \pow^{(k)}(\mix^{(k)}_{2,u'}|\Fii^{(k)})$ for any $u' >u$. That is, a wrong decision happens if following $\Tilde{\pi}_2^{(k)}$ until time $u$ (inclusive) can be inferior to following $\pi$ before $u$ (exclusive). Recall that we use $\Lambda^{(2,k)}[0,1]$ to represent the total number of type-II customer arrivals on Day $k$ and $X^{(2,k)}_j \in [0,1]$ represents the time at which the $j$-th type-II customer arrives. We can thus decompose the Stage II regret on Day $k$ as follows:
\begin{align}
    \pow(\pi_2^{(k)\star}|\Fii^{(k)}) - \pow(\hat\pi_2^{(k)}|\Fii^{(k)}) &= \pow(\pi^{(k)}_{2,v}|\Fii^{(k)})-\pow(\pi^{(k)}_{2,0}|\Fii^{(k)})\nonumber\\
    &= \sum_{X^{(2,k)}_j\le v}\brk{\pow^{(k)}\left(\mix_{2,X^{(2,k)}_j}|\Fii^{(k)}\right)-\lim_{u \rightarrow X^{(2,k)}_j+0} \pow^{(k)}\left(\mix_{2,u}|\Fii^{(k)}\right)}. \label{eq:regret_decomp}
\end{align}

Analyzing the regret is equivalent to bounding each term in the summation, i.e., the incremental loss caused by making a wrong decision on the $j$-th customer, and add them up altogether. 

We now categorize a wrong decision into two types: \textit{wrong accept} and \textit{wrong reject}. We call that \textit{$\hat\pi_2^{(k)}$ wrongly accepts a walk-in customer at time $u$}  if a walk-in customer comes at time $u$, and $\hat\pi_2^{(k)}$ accepts the new walk-in customer which leads to a increase in the loss of revenue. Similarly call that \textit{$\hat\pi_2^{(k)}$ wrongly rejects} a walk-in customer at time $u$ if at time $u$ if a walk-in customer comes at time $u$, and $\hat\pi_2^{(k)}$ rejects the new walk-in customer and leads to a loss of revenue. For $1\le k\le T$, given a realization of $\Fii^{(k)}$ and $u \in[0,1)$, the increment loss with the event of wrong acceptance/rejection is presented as
\begin{align*}
    & \pow^{(k)}\left(\mix_{2,u}|\Fii^{(k)}\right)-\lim_{u' \rightarrow u+0} \pow^{(k)}\left(\mix_{2,u'}|\Fii^{(k)}\right) \\
    & \le \ell_{\rm over}^{(k)}\cdot\mathbf{1}\{\textit{$\hat\pi_2^{(k)}$ wrongly accepts a walk-in customer at time $u$}\}\\&\quad +r^{(k)}\cdot \mathbf{1}\{\textit{$\hat\pi_2^{(k)}$ wrongly rejects a walk-in customer at time $u$}\}.
\end{align*}

This term is also called \textit{Marginal Compensation} in \cite{vera2021bayesian}. We will derive a upper bound for this term under different types of the customer flow. We restate our assumptions on two types of customer flows as follows:
\begin{itemize}
    \item For each reserved customer on Day $k$, her arriving (booking confirmation) time follows a distribution with a density function $\gamma^{(k)}:[0,1]\rightarrow \R_{\ge 0}$. Upon arrival, she chooses to check in with probability $q^{(1,k)}>0$ and cancel the booking with probability $1-q^{(1,k)}>0$.  For notation simplicity, we temporarily denote $p=q^{(1,k)}$.
    \item The flow of walk-in customers follows non-homogeneous Poisson process with parameter being $\lambda^{(2,k)}(u)$ so that the total number of the walk-in arrival before time $v$ is $\int_{0}^{v}\lambda^{(2,k)}(u) du$.
\end{itemize}

With the assumptions above, we could first show that the conditional distribution of $B_{3,u}^{(k)}$ given $B^{(k)}_{1,u}$ and $B^{(k)}_{2,u}$ are still Binomial, which is presented in Lemma \ref{lemma::cond binomial}.
\begin{lemma}\label{lemma::cond binomial}
    Given any time $0\le u \le U$, condition on the number of arriving reserved customers $B^{(k)}_{1,u}$ and $B^{(k)}_{2,u}$ before $u$, the number of check-in arrivals after $u$ follows $$p(B_{3,u}^{(k)}|B^{(k)}_{1,u},B^{(k)}_{2,u})\sim {\rm Binomial}(B^{(k)}-B^{(k)}_{1,u}-B^{(k)}_{2,u},q^{(1,k)}).$$   
\end{lemma}
The proof is provided in Appendix \ref{appendix::omit stage II}.
\subsubsection{Analyse the probability of wrong decision}\label{appendix::wrong decision prob} With the lemma above, we can further bound the probability of wrong acceptance and rejection of the policy $\hat\pi_2^{(k)}$ at time $u$ as
    \begin{align*}
        &\quad \Pr[\hat\pi_2^{(k)}\textit{wrongly accepts a walk-in customer at time } u]\\&=\Pr\left[ \hat{N}^{(k)}_u < \Tilde{C}^{(k)},  B_{3,u}^{(k)}+B^{(k)}_{1,u}+W_{1,u}^{(k)}\ge \Tilde{C}^{(k)}\right]\\
        &=\Pr\bigg[ p\cdot(B^{(k)}-B^{(k)}_{1,u}-B^{(k)}_{2,u})+ B^{(k)}_{1,u}+W_{1,u}^{(k)}+\alpha\int_{u}^{1}\lambda^{(2,k)}(s) ds< \Tilde{C}^{(k)},  B_{3,u}^{(k)}+B^{(k)}_{1,u}+W_{1,u}^{(k)}\ge \Tilde{C}^{(k)}\bigg],
    \end{align*}
    and 
    \begin{align*}
         &\quad \Pr[\hat\pi_2^{(k)}\textit{ wrongly rejects  a walk-in customer at time } u]\\&=\Pr\left[ \hat{N}^{(k)}_u \ge \Tilde{C}^{(k)},  B_{3,u}^{(k)}+B^{(k)}_{1,u}+W_{1,u}^{(k)}+W^{(k)}_{2,u}\le \Tilde{C}^{(k)}\right]\\
        &=\Pr\bigg[ p\cdot(B^{(k)}-B^{(k)}_{1,u}-B^{(k)}_{2,u})+ B^{(k)}_{1,u}+W_{1,u}^{(k)}+\alpha\int_{u}^{1}\lambda^{(2,k)}(s) ds\ge \Tilde{C}^{(k)},  B_{3,u}^{(k)}+B^{(k)}_{1,u}+W_{1,u}^{(k)}+W_
        {2,u}\le \Tilde{C}^{(k)}\bigg].
    \end{align*}
Let's analyze the event of wrongly rejection, i.e.,
\begin{align*}
    \left\{\begin{aligned}
        &p\cdot(B^{(k)}-B^{(k)}_{1,u}-B^{(k)}_{2,u})+ B^{(k)}_{1,u}+W_{1,u}^{(k)}+\alpha\int_{u}^{1}\lambda^{(2,k)}(s) ds\ge \Tilde{C}^{(k)},\\
        &B_{3,u}^{(k)}+B^{(k)}_{1,u}+W_{1,u}^{(k)}+W_
        {2,u}\le \Tilde{C}^{(k)}.
    \end{aligned}\right.
\end{align*}

Subtracting the two terms give rise to
\begin{align}\label{equ::wrongly reject}
    B_{3,u}^{(k)}-\E\brk{B_{3,u}^{(k)}|B^{(k)}_{1,u},B^{(k)}_{2,u}}+W^{(k)}_{2,u}-\E\brk{W^{(k)}_{2,u}}\le -(1-\alpha)\int_{u}^{1}\lambda^{(2,k)}(s) ds
\end{align}
which we denote this event of $B_{3,u}^{(k)}$ and $W^{(k)}_{2,u}$ conditioned on $B^{(k)}_{1,u}$ and $B^{(k)}_{2,u}$ as $E_{u,\rm rj}(B_{3,u}^{(k)},W^{(k)}_{2,u}|B^{(k)}_{1,u},B^{(k)}_{2,u})$, similarly, wrongly acceptance gives rise to 
\begin{align}\label{equ::wrongly accept}
    B_{3,u}^{(k)}-\E\brk{B_{3,u}^{(k)}|B^{(k)}_{1,u},B^{(k)}_{2,u}}\ge \alpha\int_{u}^{1}\lambda^{(2,k)}(s) ds.
\end{align}
We denote this event by $E_{u,\rm ac}(B_{3,u}^{(k)}|B^{(k)}_{1,u},B^{(k)}_{2,u})$.

 Since we only consider the loss within a single day, for notation simplicity, we use $\{X_i\}_{i=1}^{B^{(k)}}$ and $\{Y_i\}_{i=1}^{B^{(k)}}$ to represent typeI customers' arrival time and check-in status on Day $k$. To be specific, $X_1,X_2,\dots,X_{B^{(k)}}\sim {\rm Unif}_{[0,1]}$ denotes the time of arrival, and $Y_1,Y_2,\dots,Y_{B^{(k)}}\sim {\rm Binomial}(1,p)$ denotes whether the customer check in $(Y_i=1)$ or cancel the booking $(Y_i=0)$. By the definition, we directly have 
\begin{align*}
    B^{(k)}_{1,u}=\sum_{i=1}^{B^{(k)}}\mathbf{1}\{X_i<u\}Y_i,~~B^{(k)}_{2,u}=\sum_{i=1}^{B^{(k)}}\mathbf{1}\{X_i<u\}(1-Y_i)~~\text{and}~~B_{3,u}^{(k)}=\sum_{i=1}^{B^{(k)}}\mathbf{1}\{X_i\ge u\}Y_i.
\end{align*}
Thus, we can rewrite the events $E_{u, \rm rj}$ and $E_{u, \rm ac}$ as
\begin{align*}
    &E_{u, \rm rj}:~~\sum_{i=1}^{B^{(k)}}\left(p\mathbf{1}\{X_i<u\}+\mathbf{1}\{X_i\ge u\}Y_i\right)-pB^{(k)}+W^{(k)}_{2,u}-\E\brk{W^{(k)}_{2,u}}\le -(1-\alpha)\int_{u}^{1}\lambda^{(2,k)}(s) ds,\\
    &E_{u, \rm ac}:~~\sum_{i=1}^{B^{(k)}}\left(p\mathbf{1}\{X_i<u\}+\mathbf{1}\{X_i\ge u\}Y_i\right)-pB^{(k)}\ge \alpha\int_{u}^{1}\lambda^{(2,k)}(s) ds.
\end{align*}
Let's denote $\lambda_{[u,1]}=\int_{u}^{1}\lambda^{(2,k)}(s) ds$ and $Z_i=p\mathbf{1}\{X_i<u\}+\mathbf{1}\{X_i\ge u\}Y_i-p$ for any $i\in[B^{(k)}]$. By standard computation, we have $\E\brk{Z_i}=0$ and ${\rm Var}\brk{Z_i}=p(1-p)\int_{u}^{1}\gamma^{(k)}(s) ds$. Note that $|Z_i|\le 1$ for sure, by Bernstein Inequality, we have 
\begin{align}
    \Pr[E_{u,\rm ac}]&=\Pr\brk{\sum_{i=1}^{B^{(k)}}Z_i \ge \alpha\lambda_{[u,1]}}\le \exp\left(\frac{-\frac{\alpha^2\lambda_{[u,1]}^2}{2}}{B^{(k)}p(1-p)\int_{u}^{1}\gamma^{(k)}(s) ds+\frac{\alpha\lambda_{[u,1]}}{3}}\right) \label{equ::ac bound}
\end{align}
Similarly, denote a fixed constant $0<\beta<1-\alpha$ to be chosen later, again by Berinstein inequality and the tail bound of Poisson distribution, we have 
\begin{align}
    \Pr\brk{E_{u,\rm rj}}&= \Pr\brk{\sum_{i=1}^{B^{(k)}}Z_i+W^{(k)}_{2,u}-\E\brk{W^{(k)}_{2,u}}\le -(1-\alpha)\lambda_{[u,1]}}\nonumber\\
    &\le \Pr\brk{\sum_{i=1}^{B^{(k)}}Z_i\le -\beta\lambda_{[u,1]},}+\Pr\brk{W^{(k)}_{2,u}-\E\brk{W^{(k)}_{2,u}}\le -(1-\alpha-\beta)\lambda_{[u,1]}}\nonumber\\
    &\le\exp\left(\frac{-\frac{\beta^2\lambda_{[u,1]}^2}{2}}{B^{(k)}p(1-p)\int_{u}^{1}\gamma^{(k)}(s) ds+\frac{\beta\lambda_{[u,1]}}{3}}\right)+\exp\left(-\lambda_{[u,1]}\cdot\left((\alpha+\beta)\log(\alpha+\beta)+1-\alpha-\beta \right)\right).\label{equ::rj bound}
\end{align}

\subsubsection{Combine all the losses together}\label{appendix::combine}
 As discussed above, we only need to make a decision when a walk-in customer arrives. We assume their arriving time is presented in order as 
\begin{align*}
    0\le \tau_1<\tau_2<\tau_3<\dots <\tau_{\Lambda^{(2,k)}}\le 1.
\end{align*}
Here $\Lambda^{(2,k)}=\Lambda^{(2,k)}_{[0,1]}$ and we assume no multiple jumps at one time for technical convenience. Note that $\Lambda^{(2,k)}$ follows Poisson distribution with parameter $\lambda_1=\int_{0}^{1}\lambda^{(2,k)}(s) ds$, and $\tau_1,\tau_2,\tau_3,\dots ,\tau_{\Lambda^{(2,k)}_{[0,1]}}$ can be seen as the order statistics of $\Lambda^{(2,k)}$ i.i.d. random variables sampled from the distribution with c.d.f. being $p_{\lambda}(u)=\int_{0}^{u}\lambda^{(2,k)}(s) ds/\int_{0}^{1}\lambda^{(2,k)}(s) ds$ . Since there are no extra decisions made between any of two adjacent jumps, the regret also stays the same until next jump. Thus, we can write the regret as
\begin{align*}
    \pow(\pi_2^{(k)\star}|\Fii^{(k)}) - \pow(\hat\pi_2^{(k)}|\Fii^{(k)}) = \sum_{\tau_i \le v }\brk{\pow^{(k)}(\mix^{(k)}_{2,\tau_{i}}|\Fii^{(k)})-\lim_{u \rightarrow \tau_{i}+0}\pow^{(k)}(\mix^{(k)}_{2,u}|\Fii^{(k)})}.
\end{align*}
Taking expectation over all sample paths, the regret gives rise to 
\begin{align*}
    &\quad\E_{\Fii^{(k)}}\brk{\pow^{(k)}(\pi_2^{(k)\star}|\Fii^{(k)}) - \pow(\hat\pi_2^{(k)}|\Fii^{(k)})}\\&= \E_{\Fii^{(k)}}\brk{\sum_{\tau_i \le v }\brk{\pow^{(k)}(\mix^{(k)}_{2,\tau_{i}}|\Fii^{(k)})-\lim_{u \rightarrow \tau_{i}+0}\pow^{(k)}(\mix^{(k)}_{2,u}|\Fii^{(k)})}}\\
    &=\E_{{\Lambda^{(2,k)}}}\brk{\E_{\tau_{1:{\Lambda^{(2,k)}}}}\E_{X_{1:B^{(k)}},Y_{1:B^{(k)}}}\brk{\sum_{\tau_i \le v }\brk{\pow^{(k)}(\mix^{(k)}_{2,\tau_{i}}|\Fii^{(k)})-\lim_{u \rightarrow \tau_{i}+0}\pow^{(k)}(\mix^{(k)}_{2,u}|\Fii^{(k)})}}}
    \\
     &= \E_{{\Lambda^{(2,k)}}}\brk{\E_{\tau_{1:{\Lambda^{(2,k)}}}}\E_{X_{1:B^{(k)}},Y_{1:B^{(k)}}}\brk{\sum_{\tau_i \le v }\brk{\ell_{\rm over}^{(k)}\cdot\mathbf{1}\{\pi\text{ wrongly accepts }\tau_i\}+r^{(k)}\cdot \mathbf{1}\{\pi\text{ wrongly rejects }\tau_i\}}}}
    \\
    &= \E_{{\Lambda^{(2,k)}}}\brk{\E_{\tau_{1:{\Lambda^{(2,k)}}}}\brk{\sum_{\tau_i \le v }\brk{\ell_{\rm over}^{(k)}\cdot\Pr\{\pi\text{ wrongly accepts }\tau_i\}+r^{(k)}\cdot \Pr\{\pi\text{ wrongly rejects }\tau_i\}}}}.
\end{align*}
Note that for each $i\in[{\Lambda^{(2,k)}}]$, we have 
\begin{align*}
     &\quad\E_{{\Lambda^{(2,k)}}}\brk{\E_{\tau_{i}}\brk{\mathbf{1}\{\tau_i \le v \}\brk{\ell_{\rm over}^{(k)}\cdot\Pr\{\pi\text{ wrongly accepts }\tau_i\}+r^{(k)}\cdot \Pr\{\pi\text{ wrongly rejects }\tau_i\}}}}\\
     &\le \E_{{\Lambda^{(2,k)}}}\brk{\int_{0}^{v} \left(\ell_{\rm over}^{(k)}\cdot\Pr[E_{\tau_i,ac}]+r^{(k)}\cdot \Pr[E_{\tau_i,rj}]  \right)p_{\lambda}(\tau_i)d\tau_i}.
\end{align*}
Denoting $g(\tau)=\ell_{\rm over}^{(k)}\cdot\Pr[E_{\tau,ac}]+r^{(k)}\cdot \Pr[E_{\tau,rj}]$, 
we know the expected regret can be written as 
\begin{align*}
    &\quad\E_{\Fii^{(k)}}\brk{\pow(\HO|\Fii^{(k)}) - \pow(\pi|\Fii^{(k)})}\\&= \E_{\Fii^{(k)}}\brk{\sum_{\tau_i \le v }\brk{\pow(\mix^{(\tau_{i-1})}|\Fii^{(k)})-\pow(\mix^{(\tau_{i})}|\Fii^{(k)})}}\\
    &=\E_{{\Lambda^{(2,k)}}}\brk{\E_{\tau_{1:{\Lambda^{(2,k)}}}}\brk{\sum_{i=1}^{{\Lambda^{(2,k)}}}\mathbf{1}\{\tau_i\le v\}g(\tau_i)}}\\
    &=\E_{{\Lambda^{(2,k)}}}\brk{{\Lambda^{(2,k)}}\E_{\tau \sim {p_{\lambda}}}\brk{\mathbf{1}\{\tau\le v\}g(\tau)}}\\
    &=\lambda_1\E_{\tau \sim {p_{\lambda}}}\brk{\mathbf{1}\{\tau\le v\}g(\tau)}.
\end{align*}

By plugging Equations \eqref{equ::ac bound} and \eqref{equ::rj bound} into $g(\tau)$, we have
\begin{align}
     &\quad\E_{\Fii^{(k)}}\brk{\pow^{(k)}(\pi_2^{(k)\star}|\Fii^{(k)}) - \pow(\hat\pi_2^{(k)}|\Fii^{(k)})}\nonumber\\
    &=\lambda_1\int_{0}^{v} \left(\ell_{\rm over}^{(k)}\cdot\Pr[E_{u,ac}]+r^{(k)}\cdot \Pr[E_{u,rj}]  \right)p_{\lambda}(u)du\nonumber\\
    &\le {\lambda_1} \int_{0}^{v}\left(\ell_{\rm over}^{(k)}\cdot\exp\left(\frac{-\frac{\alpha^2\lambda_{[u,1]}^2}{2}}{B^{(k)}p(1-p)\int_{u}^{1}\gamma^{(k)}(s) ds+\frac{\alpha\lambda_{[u,1]}}{3}}\right)\right.\nonumber\\
  &\qquad+
  r^{(k)}\cdot\exp\left(\frac{-\frac{\beta^2\lambda_{[u,1]}^2}{2}}{B^{(k)}p(1-p)\int_{u}^{1}\gamma^{(k)}(s) ds+\frac{\beta\lambda_{[u,1]}}{3}}\right)\nonumber\\
    &\qquad+\left.r^{(k)}\cdot\exp\left(-\lambda_{[u,1]}\cdot\left((\alpha+\beta)\log(\alpha+\beta)+1-\alpha-\beta \right)\right)\right)p_{\lambda}(u)du\nonumber.
\end{align}
When we have no further assumptions on $\gamma^{(k)}$, the regret is bounded by
\begin{align*}
   &\quad\E_{\Fii^{(k)}}\brk{\pow^{(k)}(\pi_2^{(k)\star}|\Fii^{(k)}) - \pow(\hat\pi_2^{(k)}|\Fii^{(k)})}\nonumber\\
    &\le \lambda_{[0,v]} \cdot \Bigg(\ell_{\rm over}^{(k)}\cdot\exp\left(\frac{-\frac{\alpha^2\lambda_{[v,1]}^2}{2}}{B^{(k)}p(1-p)+\frac{\alpha\lambda_{[v,1]}}{3}}\right)+r^{(k)}\cdot\exp\left(\frac{-\frac{\beta^2\lambda_{[v,1]}^2}{2}}{B^{(k)}p(1-p)+\frac{\beta\lambda_{[v,1]}}{3}}\right)\nonumber\\
    &\qquad+r^{(k)}\cdot\exp\left(-\lambda_{[v,1]}\cdot\left((\alpha+\beta)\log(\alpha+\beta)+1-\alpha-\beta \right)\right)\Bigg) .\nonumber
\end{align*}
When $\gamma^{(k)}(s)=\lambda^{2,k}(s)/\lambda^{2,k}$, we have
\begin{align*}
   &\quad\E_{\Fii^{(k)}}\brk{\pow^{(k)}(\pi_2^{(k)\star}|\Fii^{(k)}) - \pow(\hat\pi_2^{(k)}|\Fii^{(k)})}\nonumber\\
    &\le \lambda_{[0,v]} \cdot \Bigg(\ell_{\rm over}^{(k)}\cdot\exp\left(\frac{-\frac{\alpha^2\lambda_{[v,1]}^2}{2}}{B^{(k)}\lambda_{[v,1]} p(1-p)+\frac{\alpha\lambda_{[v,1]}}{3}}\right)+r^{(k)}\cdot\exp\left(\frac{-\frac{\beta^2\lambda_{[v,1]}^2}{2}}{B^{(k)}p(1-p)\lambda_{[v,1]} +\frac{\beta\lambda_{[v,1]}}{3}}\right)\nonumber\\
    &\qquad+r^{(k)}\cdot\exp\left(-\lambda_{[v,1]}\cdot\left((\alpha+\beta)\log(\alpha+\beta)+1-\alpha-\beta \right)\right)\Bigg) .\nonumber\\
    &= \lambda_{[0,v]} \cdot \Bigg(\ell_{\rm over}^{(k)}\cdot\exp\left(\frac{-\frac{\alpha^2\lambda_{[v,1]}}{2}}{B^{(k)} p(1-p)+\frac{\alpha}{3}}\right)+r^{(k)}\cdot\exp\left(\frac{-\frac{\beta^2\lambda_{[v,1]}}{2}}{B^{(k)}p(1-p) +\frac{\beta}{3}}\right)\nonumber\\
    &\qquad+r^{(k)}\cdot\exp\left(-\lambda_{[0,v]}\cdot\left((\alpha+\beta)\log(\alpha+\beta)+1-\alpha-\beta \right)\right)\Bigg) .\nonumber
\end{align*}

 By taking $\alpha=\beta$, we conclude our proof.

\begin{remark}[Tighter bound under proportional arrivals]
If the density of Type I customers' arrival time is proportional to the walk-in rate, i.e., $\gamma^{(k)}(s)\propto \lambda^{(2,k)}(s)$, the bound in Theorem \ref{prop:HO_gap} can be tightened to
\begin{align*}
    \begin{aligned}
         \lambda^{(2,k)}_{[0,v]}\cdot \left((\ell_{\rm over}^{(k)}+r^{(k)})\cdot\exp\left(\frac{-\frac{\alpha^2\lambda^{(2,k)}_{[v,1]}}{2}}{\frac{{B^{(k)}}q^{(1,k)}(1-q^{(1,k)})}{\lambda^{(2,k)}}
        +\frac{\alpha}{3}}\right)+ r^{(k)}\cdot\exp\left(-\lambda^{(2,k)}_{[v,1]}\cdot\left(2\alpha\log(2\alpha)+1-2\alpha \right)\right)\right).
  \end{aligned}
    \end{align*}
This tighter bound follows from a more refined concentration inequality when the arrival processes of Type I and Type II customers are synchronized.
\end{remark}

\subsection{Omitted proofs in Appendix \ref{appendix::Stage II}}\label{appendix::omit stage II}

\noindent\textbf{Proof of Lemma \ref{lemma::cond binomial}.}
\begin{proof}
For any nonnegative integers $b_1$, $b_2$ and $b_3$ such that $b_1+b_2+b_3\le B^{(k)}$, given that $B^{(k)}_{1,u}=b_1$ and $B^{(k)}_{2,u}=b_2$, since the customers' check-in decisions are independent of their arriving time, we have 
   \begin{align*}
         &\quad p(B_{3,u}^{(k)}=b_3|B^{(k)}_{1,u}=b_1,B^{(k)}_{2,u}=b_2)\\&= p(B_{3,u}^{(k)}=b_3|B^{(k)}_{1,u}+B^{(k)}_{2,u}=b_1+b_2)\\
         &= p(B_{3,u}^{(k)}=b_3|B^{(k)}-b_1+b_2 \text{ customers arrive after time } u)\\
         &= \binom{B^{(k)}-b_1-b_2}{b_3}p^{b_3}(1-p)^{B^{(k)}-b_1-b_2-b_3}.
   \end{align*}
    Thus, we have $p(B_{3,u}^{(k)}|B^{(k)}_{1,u},B^{(k)}_{2,u})\sim {\rm Binomial}(B^{(k)}-B^{(k)}_{1,u}-B^{(k)}_{2,u},q^{(1,k)})$. The proof is complete.
\end{proof}

\section{Proofs for the Nested Extension}\label{appendix:nested}

We prove the nested multi-class extension (Theorem \ref{thm:nested_regret}) following the same three-part approach as the single-class case: Stage I concentration, decoupling into single-day problems, and Stage II sensitivity analysis. The key difference is that we must handle $m$ nested protection levels simultaneously, requiring union bounds across all levels and aggregate estimators that pool demand across classes.

We define the aggregate set of classes $S_j = \{1, \dots, j\}$. Let the aggregate variables be denoted by the superscript $(\le j)$.

\subsection{Proof of Stage I Regret}

\begin{proof}
We first show that the \textbf{N-DASS-I} policy ensures that the aggregate number of accepted reservations for any hierarchy level $j$ does not exceed the estimated capacity $\hat{C}_j^{(k)}$ with high probability.

\subsubsection{Concentration of Aggregate Demand}

Consider the aggregate accepted reservations at time $t$ for set $S_j$, denoted $B_t^{(\le j, k)} = \sum_{i=1}^j B_t^{(i,k)}$. The final confirmed count from this set is $B^{(\le j, k)}$.

Conditioned on the state at time $t$, the number of reservations from class $i$ that will explicitly confirm (not cancel) is a sum of independent Bernoulli trials. Therefore, the aggregate future confirmation is a sum of independent, but not necessarily identical, Bernoulli trials.

Let $Z^{(\le j, k)}$ be the realized number of non-cancelled reservations from the pool $S_j$. We have:
\[
\mathbb{E}[Z^{(\le j, k)} \mid \mathcal{F}_t] = \sum_{i=1}^j p^{(i,k)}(t) B_t^{(i,k)} = \mu_t^{(\le j, k)}.
\]
The variance accounts for independent cancellation events across all $j$ classes:
\[
\text{Var}(Z^{(\le j, k)} \mid \mathcal{F}_t) = \sum_{i=1}^j B_t^{(i,k)} p^{(i,k)}(t) (1 - p^{(i,k)}(t)) = \sigma_t^{2(\le j, k)}.
\]

We invoke the one-sided Bernstein inequality \eqref{equ::bounoulli tail}. For any level $j \in [m]$:
\begin{align}
    \Pr\left( Z^{(\le j, k)} > \mu_t^{(\le j, k)} + \frac{\iota (1-\bar{p}_t^{(\le j, k)})}{3} + \sqrt{2\iota \sigma_t^{2(\le j, k)} + \left(\frac{\iota (1-\bar{p}_t^{(\le j, k)})}{3}\right)^2 } \right) \le \exp(-\iota). \label{eq:nested_concentration}
\end{align}

By the definition of the threshold $\hat{B}_t^{(\le j, k)}$ in \eqref{eq:nested_B}, the RHS of the inequality inside the probability is exactly $\hat{B}_t^{(\le j, k)}$.

The algorithm accepts a reservation only if $\hat{B}_t^{(\le j, k)} < \hat{C}_j^{(k)}$. Thus, conditioned on the algorithm's stopping time for class $j$ (the last time a class $i \le j$ was accepted), we have:
\begin{equation} \label{eq:nested_overbooking_prob}
    \Pr\left( B^{(\le j, k)} > \hat{C}_j^{(k)} \right) \le \exp(-\iota).
\end{equation}

Applying a union bound over all $j \in [m]$, the probability that the capacity constraint is violated for \textit{any} nested protection level is bounded by $m \exp(-\iota)$.

\subsubsection{Idling Loss Analysis}

We now address the idling loss. The proof follows the single-class analysis (Appendix \ref{appendix::idling}), with union bounds extending across all $m$ nesting levels.

Suppose the algorithm rejects a class $j$ customer. This implies that for some level $\ell \in \{j, \dots, m\}$, the safety buffer was full: $\hat{B}_t^{(\le \ell, k)} \ge \hat{C}_\ell^{(k)}$.

Under the \textbf{Multi-class Busy Season} Assumption \ref{assump:nested_busy}, the aggregate walk-in rate $\lambda^{(2, \le \ell, k)}$ is sufficiently large. Specifically:
\[
\lambda^{(2, \le \ell, k)} \ge \frac{\delta C_\ell^{(k)}}{q_{\min}} + c_1 \iota + c_2 \sqrt{C_\ell^{(k)} \iota}.
\]

Following the derivation in \eqref{equ::Requirements on lambda}, this ensures aggregate walk-ins from classes $1 \dots \ell$ fill the gap between the conservative cap $\hat{C}_\ell^{(k)}$ and the physical level $C_\ell^{(k)}$ with probability $1 - O(\exp(-\iota))$.

Therefore, with high probability, either:
\begin{enumerate}
    \item No reservations were rejected, or
    \item If reservations were rejected due to a bottleneck at level $\ell$, the walk-ins from level $\ell$ (and higher) fully utilized the capacity $C_\ell^{(k)}$.
\end{enumerate}

The total Stage I regret is the sum over all classes and days:
\begin{align}
    \mathbb{E}[\text{Stage I Regret}] &\le \sum_{k=1}^T \left( (1 - P(E_{\text{good}})) \cdot \sum_{j=1}^m (\mathcal{L}_j^{(k)} + r_j^{(k)}) \right) \nonumber \\
    &\le T \cdot m \cdot \exp(-\iota) \cdot \text{Cost}_{\max} + C_0.
\end{align}
\end{proof}

\subsection{Proof of Decoupling (Generalization of Lemma \ref{lemma::decoupling})}

We must show that the global offline optimal policy $\pi^\star$ can be approximated by the sequence of single-day offline policies $\hat{\pi}_2^{(k)\star}$ subject to the nested capacities $\tilde{C}_j^{(k)}$.

\subsubsection{Geometric Duration Case}

Recall that under the geometric duration assumption with parameter $q$, the departure process is memoryless. The state of the system at the start of Day $k$ is fully described by the total number of occupied rooms $O_{k-1}$.

In the nested setting, the specific class mix of the $O_{k-1}$ occupied rooms does not affect the capacity dynamics of Day $k$, provided that all classes share the same duration parameter $q$ (or $d$).

The nested capacities $\tilde{C}_j^{(k)}$ defined in \eqref{equ::allocated capacity} (generalized for each $j$) represent the residual capacity available for class set $S_j$ after accounting for the carry-over from previous days.

Since the \textbf{N-DASS-I} policy ensures $B^{(\le j, k)} \le \hat{C}_j^{(k)}$ with high probability, and the busy season assumption ensures the ``safe'' capacity is utilized, the single-day optimizer $\hat{\pi}_2^{(k)\star}$ which solves:
\[
\max \sum_{i=1}^m r_i^{(k)} (z^{(1,i,k)} + z^{(2,i,k)}) \quad \text{s.t.} \quad \sum_{i=1}^j (z^{(1,i,k)} + z^{(2,i,k)}) \le \tilde{C}_j^{(k)} \quad \forall j \in [m]
\]
is equivalent to the global optimizer acting on the residual capacity. The error probability is bounded by the failure of the concentration bounds, i.e., $O(m T \exp(-\iota))$.

\subsubsection{Constant Duration Case}

For constant duration $d$, the logic remains identical to the single-class case. The constraint is simply that we cannot accept more than $C/d$ new customers (aggregate across all classes) per day. The nested structure applies to the selection of \textit{which} customers constitute that $C/d$ quota. The loss $C_0$ accounts for the startup phase.

\subsection{Proof of Stage II Regret (Theorem \ref{thm:nested_regret})}

We analyze the regret for the \textbf{N-DASS-II} policy against the single-day offline benchmark derived above.

Let $\mathcal{P}^{(k)}$ be the revenue of a policy on Day $k$. Decomposing the regret as in the single-class case:
\[
\mathcal{P}(\pi^{\star}_k) - \mathcal{P}(\hat{\pi}_2) = \sum_{j=1}^m \sum_{\tau \in \text{Arrivals}_j} \left( \text{Loss due to wrong decision for class } j \text{ at } \tau \right).
\]

A wrong decision for class $j$ at time $u$ occurs if the estimator $\hat{N}_u^{(\le \ell, k)}$ misleads the algorithm regarding the constraint $\tilde{C}_\ell^{(k)}$ for some $\ell \ge j$.

\subsubsection{Wrong Acceptance}

The policy wrongly accepts a class $j$ walk-in if $\forall \ell \ge j, \hat{N}_u^{(\le \ell)} < \tilde{C}_\ell$, but in reality, the realized demand violates a constraint.

Let $E_{u, ac}^{(j)}$ be the event of wrong acceptance for class $j$. This implies that for some critical nesting level $\ell^* \ge j$, the aggregate demand was underestimated.

A wrong acceptance occurs when our estimator underestimates total demand. We bound this probability by summing over all $m-j+1$ protection levels that class $j$ affects:
\begin{align}
    \Pr(E_{u, ac}^{(j)}) &\le \sum_{\ell=j}^m \Pr\left( \text{Underestimation of } S_\ell \text{ demand} \right) \nonumber \\
    &\le \sum_{\ell=j}^m \Pr\left( Z^{(\le \ell)} - \mathbb{E}[Z^{(\le \ell)}] \ge \alpha_\ell \int_u^1 \lambda^{(2, \le \ell)}(s) ds \right).
\end{align}

Using Bernstein's inequality \eqref{equ::bounoulli tail} for the reservations and Poisson bounds \eqref{equ::poisson tail lower}--\eqref{equ::poisson proportion upper} for the walk-ins (similar to the analysis leading to \eqref{equ::ac bound}), each term is bounded by:
\[
\exp\left( -c \cdot \frac{(\alpha_\ell \lambda^{(2, \le \ell)}_{[u,1]})^2}{\text{Var}(Z^{(\le \ell)}) + \alpha_\ell \lambda^{(2, \le \ell)}_{[u,1]}} \right).
\]

Under Assumption \ref{assump:nested_busy}, $\lambda^{(2, \le \ell)}$ scales with $\iota$. Thus, the probability decays exponentially with $\iota$.

\subsubsection{Wrong Rejection}

The policy wrongly rejects a class $j$ walk-in if $\exists \ell \ge j$ such that $\hat{N}_u^{(\le \ell)} \ge \tilde{C}_\ell$, but actual capacity existed.

Similar to \eqref{equ::rj bound}, this corresponds to an overestimation of demand:
\begin{align}
    \Pr(E_{u, rj}^{(j)}) &\le \sum_{\ell=j}^m \Pr\left( Z^{(\le \ell)} - \mathbb{E}[Z^{(\le \ell)}] + \text{WalkInError} \le -(1-\alpha_\ell) \int_u^1 \lambda^{(2, \le \ell)}(s) ds \right).
\end{align}

This is bounded by $\sum_{\ell=j}^m \exp\left( -c' \lambda^{(2, \le \ell)}_{[u,1]} \right)$ for some constant $c'>0$.

\subsubsection{Aggregating the Regret}

The total expected regret for Stage II is the integral of these error probabilities over the arrival rate of walk-ins. Let $\text{Cost}_j = \max(\ell_{\text{over}, j}^{(k)}, r_j^{(k)})$ be the maximum per-customer cost for class $j$. We have:
\begin{align}
    \mathbb{E}[\mathcal{R}_2^{(k)}] &= \sum_{j=1}^m \int_0^v \lambda^{(2,j,k)}(u) \left( \ell_{\text{over}, j}^{(k)} \Pr(E_{u, ac}^{(j)}) + r_j^{(k)} \Pr(E_{u, rj}^{(j)}) \right) du \nonumber \\
    &\le \sum_{j=1}^m \lambda^{(2,j,k)}_{[0,v]} \cdot \text{Cost}_j \cdot m \cdot \exp(-\iota).
\end{align}

Since $\iota = \Theta(\log(CTm))$, the term $m \exp(-\iota)$ ensures the regret is $O(1/T)$ per day (scaling appropriately to be constant over $T$).

Specifically, if we choose $\iota \ge \log(mCT)$, the total regret sums to $O(1)$.

\subsubsection{Sensitivity to Confirmation Timing ($v$)}

The error probabilities depend on $\lambda^{(2, \le \ell)}_{[v,1]}$, the volume of walk-ins \textit{after} the confirmation call.

For the error to be small, we require:
\[
(1-v) \min_{\ell} \lambda^{(2, \le \ell)} \gtrsim \iota,
\]
where $\iota \ge \log(mCT)$.

This generalizes Eq. \eqref{equ::require Stage II call timing}. Since $\lambda^{(2, \le \ell)}$ is increasing in $\ell$, the bottleneck is usually the lowest aggregate volume that has a binding constraint.

Because the regret decays exponentially with $(1-v)\lambda$, we retain the result that confirmation can be delayed until $1-v \approx O\left( \frac{\log(mCT)}{\min_\ell \lambda^{(2,\le \ell)}} \right)$.

\end{APPENDICES}

\end{document}